%% file: paper6v2.tex
\documentclass[a4paper,11pt]{amsart}

\usepackage[top=30truemm,bottom=30truemm,left=30truemm,right=30truemm]{geometry}

\usepackage{hyperref}
\hypersetup{
    setpagesize=true, 
    bookmarksnumbered=false, 
    bookmarksopen=true, 
    colorlinks=true,
    linkcolor=MidnightBlue, 
    citecolor=PineGreen,
    }

\usepackage[capitalise]{cleveref}
\usepackage{aliascnt}

\input{mizuno_preamble2}
\crefname{isomorphism}{isomorphism}{isomorphisms}
\Crefname{isomorphism}{Isomorphism}{Isomorphisms}
\AtBeginDocument{\let\cref\Cref}

\usepackage[backend=bibtex, style=alphabetic, sorting=nyt, giveninits=true, backref=false, backrefstyle=none, doi=false]{biblatex} 

\AtBeginBibliography{\normalsize}
\DeclareFieldFormat{biblabeldate}{\mkbibparens{\mkbibbold{#1}}}
\DeclareFieldFormat[article,periodical]{volume}{\mkbibbold{#1}}
\DeclareFieldFormat[article,periodical]{number}{\mkbibbold{#1}}\renewbibmacro{in:}{}

\reversemarginpar

\title[Dualizing complexes over $\Zbb$-algebras]{Dualizing complexes over $\Zbb$-algebras}

\author{Yuki Mizuno}
\email{y.mizuno@uu.nl}
\date{}
\address{Mathematical~Institute, Utrecht~University, Budapestlaan~6, 3584CD~Utrecht, Netherlands}


\begin{document}

	\begin{abstract}
		In this paper, we introduce the notions of dualizing complexes and balanced dualizing complexes over $\Zbb$-algebras. 
		We prove that a noetherian connected $\Zbb$-algebra $A$ admits a balanced dualizing complex if and only if $A$ satisfies Artin--Zhang's $\chi$-condition, has finite local cohomology dimension, and possesses symmetric derived torsion as a bigraded $A$-$A$-bimodule. 
		As an application of our study of dualizing complexes, we show that any smooth noncommutative projective scheme associated to a $\Zbb$-algebra with a balanced dualizing complex admits a Serre functor.
	\end{abstract}

	\maketitle
	\tableofcontents
	
\section{Introduction}

\subsection{Background and motivation}
\label{subsec:background}
The notion of a dualizing complex first appeared in \cite{hartshorne1966residues} in the context of Grothendieck duality theory for schemes.
Motivated by \cite{hartshorne1966residues}, Yekutieli studied dualizing complexes over noncommutative graded algebras in \cite{yekutieli1992dualizing} to study noncommutative local cohomology.
The lack of flexibility in noncommutative algebras, such as the impossibility of freely taking localizations as in the commutative case, prevents us from proving a local duality theorem for general dualizing complexes.
To address this issue, he introduced the notion of a balanced dualizing complex.
A balanced dualizing complex has particularly favorable properties, including the fact that it induces a local duality theorem \cite[Theorem 4.18]{yekutieli1992dualizing}, and it has found important applications, such as Serre duality in noncommutative projective geometry (see, for example, \cite{yekutieli1997serre}, \cite{de2004ideal}).
Therefore, the question of when a noncommutative graded algebra admits a balanced dualizing complex is of considerable importance.
In \cite{vandenbergh1997existence}, M. Van den Bergh showed that a noetherian connected graded algebra admits a balanced dualizing complex if and only if it has finite local cohomology dimension and satisfies Artin-Zhang's $\chi$-condition (\cite[Section 3]{artin1994noncommutative}).
Nowadays, these conditions serve as fundamental guidelines in the study of noncommutative algebras and have become indispensable in determining which classes of algebras to study.
Since then, further developments and generalizations have been made; see, for instance, \cite{yekutieli1999dualizing}, \cite{yekutieli1999rings}, \cite{wu2001dualizing}, \cite{chan2002pre}, and, most recently, \cite{brown2025existence}, among others.

\medskip
A $\Zbb$-algebra over a field $k$ is a $k$-linear category $\Ccal$ whose objects are indexed by $\Zbb$.
Actually, if we identify the set of objects of $\Ccal$ with $\Zbb$, then $A:= \bigoplus_{i,j \in \Zbb} A_{ij}$ with $A_{ij} := \Hom_{\Ccal}(j,i)$ has a natural structure of a $k$-algebra without unit.
In this way, we regard a $\Zbb$-algebra as a bigraded $k$-algebra without unit and can define the category $\Grr{A}$ of right graded modules over a $\Zbb$-algebra (about basic notions, see \cref{sec:Z-algebras}).

We can think of a $\Zbb$-algebra as a generalization of a graded $k$-algebra as follows.
Let $B$ be a graded $k$-algebra.
Then, we can define a $\Zbb$-algebra $\zgrz{B}$ by $\zgrz{B}_{ij} := B_{j-i}$ for all $i,j \in \Zbb$.
Moreover, we have the equivalence $\Grr{B} \cong \Grr{\zgrz{B}}$ of categories, where $\Grr{B}$ is the category of graded right $B$-modules (for example, see \cite[Section 2]{sierra2011galgebras}).
So, we naturally expect that the module theory over a $\Zbb$-algebra will develop in a way similar to the theory over a graded $k$-algebra.
The aim of this paper is to advance this philosophy in the direction of dualizing complexes.

\subsection{Results}
\label{subsec:results}

Let $A$ be a noetherian connected $\Zbb$-algebra over a field $k$. 
In this paper, we define the notions of a {dualizing complex} and a {balanced dualizing complex} over $A$, generalizing the definitions in \cite{yekutieli1992dualizing}. 
A bounded complex $R_A$ of bigraded $A$-$A$-bimodules is called a \emph{dualizing complex} over $A$ if it satisfies the following conditions:
\begin{enumerate}
  \item $R_A$ has finite injective dimension both as a complex of graded right $A$-modules and as a complex of graded left $A$-modules,
  \item for each $i,j$, the restrictions
  \[
  e_j H^i(R_A) := \bigoplus_{l \in \Zbb} H^i(R_A)_{jl},
  \qquad
  H^i(R_A) e_j := \bigoplus_{l \in \Zbb} H^i(R_A)_{lj}
  \]
  of the cohomology bimodule $H^i(R_A)$ are finitely generated as graded right and left $A$-modules, respectively,
  \item the natural contravariant functors $\RHomintr{A}(-,R_A)$ and $\RHomintl{A}(-,R_A)$ induce equivalences between the bounded derived categories $D^b(\grr{A})$ and $D^b(\grl{A})$ of finitely generated graded right and left $A$-modules, respectively.
\end{enumerate}

Moreover, a dualizing complex $R_A$ over $A$ is called \emph{balanced} if
\[
\Rtorfun{A}(R_A) \cong \Rtorfun{A^{\op}}(R_A) \cong \D{A}
\quad \text{in } D(\Grlr{A}{A}),
\]
where $\D{A}$ denotes the graded $k$-dual of $A$, and $\torfun{A}$ and $\torfun{A^{\op}}$ denote the right and left torsion functors, respectively (see \cref{dfn:torsion-functor}). 
For further details on the definition of (balanced) dualizing complexes, see \cref{subsec:definition of dualizing complexes}.

To state our first main result, we introduce the following notions. 
We say that $A$ satisfies the \emph{$\chi$-condition} (see \cref{dfn:chi-condition}) if $\Extintr{A}^{i}(K,M)$ and $\Extintl{A}^{i}(K,N)$ are finite $k$-modules for all $i \in \Nbb$ and all finitely generated graded right $A$-modules $M$ and left $A$-modules $N$, respectively, where $K = A/A_{\geq 1}$. 
We say that $A$ has \emph{finite local cohomology dimension} if there exists $d \in \Nbb$ such that $\R^{i}\torfun{A}(M) = 0$ and $\R^{i}\torfun{A^{\op}}(N) = 0$ for all $i > d$, for all graded right $A$-modules $M$ and all graded left $A$-modules $N$. 
Finally, viewing $A$ as a bigraded $A$--$A$-bimodule, we say that it has \emph{symmetric derived torsion} if there exists an isomorphism 
\[
\epsilon_A : \Rtorfun{A}(A) \; \xrightarrow{\cong} \; \Rtorfun{A^{\op}}(A)
\quad \text{in } D(\Grlr{A}{A})
\]
such that the following diagram
\[
\begin{tikzcd}
	\Rtorfun{A}(A) \ar[rr, "\epsilon_A"] \ar[rd, "\sigma_A"'] & & \Rtorfun{A^{\op}}(A) \ar[ld, "\sigma^{\op}_{A}"]  \\
	& A &
\end{tikzcd}
\]
is commutative, where $\sigma_M, \sigma_{M^{\op}}$ are the natural morphisms
(see \cref{dfn:symmetric derived torsion functor}).

The first main result of this paper is the following theorem, which is a generalization of \cite[Theorem 6.3]{vandenbergh1997existence} to $\Zbb$-algebras.

\begin{thm}[= \cref{thm:main1}]
	\label{thm:intro-main1}
	Let $A$ be a noetherian connected $\Zbb$-algebra. 
	
	Then, $A$ has a balanced dualizing complex if and only if 
	$A$ satisfies the $\chi$-condition, has finite local cohomological dimension and has symmetric derived torsion as a bigraded $A$-$A$-bimodule.
\end{thm}


In the theory of noncommutative graded algebras, symmetric derived torsion follows from the $\chi$-condition. 
For $\Zbb$-algebras, however, the same method cannot be applied, since an $A$-$A$-bimodule $M$ is naturally $\Zbb^2$-graded, whereas a bimodule over a graded algebra is $\Zbb$-graded. 
At present, it seems that the existence of symmetric derived torsion does not depend on the $\chi$-condition for $\Zbb$-algebras (see also \cref{rmk:chi-condition_symmetric derived torsion}). 
Nevertheless, if $A$ is $r$-periodic for some $r$, then the $\chi$-condition does imply symmetric derived torsion (see \cref{prop:chi-condition_symmetric derived torsion-2}).

We give a class of $\Zbb$-algebras satisfying the conditions in \cref{thm:intro-main1}, which may be regarded as a $\Zbb$-algebra analogue of AS-Gorenstein algebras, inspired by 
\cite[Definition 4.15]{mori2025categorical} (in detail, see \cref{dfn:AS-Gorenstein}, \cref{prop:AS-Gorenstein}).
We also check that a noetherian connected graded algebra $B$ has a balanced dualizing complex in the sense of \cite[Definition 3.3 and 4.1]{yekutieli1992dualizing} if and only if the associated $\Zbb$-algebra $\zgrz{B}$ has a balanced dualizing complex (see \cref{prop:comparison of dualizing complexes}).
In this case, the dualizing complex $R_{\zgrz{B}}$ is isomorphic to $\widezgrz{R_B}$.

\medskip
As an application of our study of dualizing complexes, we prove that a smooth noncommutative projective scheme associated to a $\Zbb$-algebra has a Serre functor.
Let $A$ be a noetherian connected $\Zbb$-algebra.
We denote by $\gr(A)$ the category of finitely generated graded right $A$-modules and by $\tor(A)$ its full subcategory of torsion modules. 
Since $\tor(A)$ is a Serre subcategory of $\gr(A)$, we can form the quotient category
\[
  \qgr(A) \coloneqq \gr(A)/\tor(A).
\]
Denote the natural projection functor by $\pi_A : \gr(A) \rightarrow \qgr(A)$.
Then, we have a right adjoint functor $\omega_A : \qgr(A) \rightarrow \gr(A)$ of $\pi_A$. 
We call $\qgr(A)$ the \emph{noncommutative projective scheme} associated to $A$ (\cref{dfn:noncommutative projective scheme}).
We say that $\qgr(A)$ has \emph{finite global dimension} (\cref{dfn:global dimension}) if there exists $d \in \Nbb$ such that
\[
  \Ext^i_{\qgr(A)}(X,Y)=0 \quad \text{for all $i > d$ and all $X,Y \in \qgr(A)$}.
\]

A \emph{Serre functor} of the bounded derived category $D^b(\qgr(A))$ is an autoequivalence $\Scal_{D^b(\qgr(A))} : D^b(\qgr(A)) \rightarrow D^b(\qgr(A))$ such that there exists a natural isomorphism
\[
\Hom_{D^b(\qgr(A))}(X,Y) \cong \Hom_{D^b(\qgr(A))}(Y,\Scal_{D^b(\qgr(A))}(X))'
\]
for all $X,Y \in D^b(\qgr(A))$, where $(-)'$ denotes the $k$-dual (see \cref{dfn:Serre-functor}).

The second main result of this paper is the following theorem.

\begin{thm}[= \cref{thm:Serre-functor}]
	\label{thm:intro-Serre-functor}
	Let $A$ be a noetherian connected $\Zbb$-algebra. 
	We assume that $A$ has a balanced dualizing complex $R_A$ and $\qgr(A)$ has finite global dimension.

	Then, the functor $\pi_A(\Romega_A(-) \Ltenint{A} R_A)[-1]$ is a Serre functor of $D^b(\qgr(A))$.
\end{thm}

Furthermore, as in the case of graded algebras studied in \cite{yekutieli2019derived} and \cite{wu2001dualizing}, we develop the theory of homological algebra over $\Zbb$-algebras in the more general unbounded setting, and also establish the fundamental theory of local cohomology.
In particular, for local duality, which was proved under restricted assumptions in \cite{mori2021local}, we provide a proof in full generality.

\subsection{Outline}
\label{subsec:outline}

In \cref{sec:Z-algebras}, we review and organize basic notions from the theory of $\Zbb$-algebras and establish some further tools to study their module categories.
In \cref{sec:local-cohomology}, we develop a theory of local cohomology for connected $\Zbb$-algebras.
We also provide a characterization of the $\chi$-condition in terms of local cohomology.
In \cref{sec:dualizing complexes}, we define and study (balanced) dualizing complexes over noetherian connected $\Zbb$-algebras, and we prove \cref{thm:intro-main1}.
In \cref{sec:application to noncommutative projective geometry}, we apply \cref{thm:intro-main1} to noncommutative projective geometry and prove \cref{thm:intro-Serre-functor}.

\subsection{Acknowledgments}
This work was supported by JSPS KAKENHI (Grant Number 24K22841) and NWO (doi:10.61686/RZKLF82806).
The author thanks the anonymous referee for their careful reading and helpful comments.

\section{$\Zbb$-algebras}
\label{sec:Z-algebras}
Let $k$ be a field and assume that $\Nbb$ contains $0$ throughout this paper.

In this section, we recall and organize basic notions from the theory of $\Zbb$-algebras (\cite{vandenbergh2011noncommutative}, \cite{mori2021local}, \cite{mori2024corrigendum}, \cite{mori2025categorical}, \cite{sierra2011galgebras}, \cite{polishchuk2005noncommutative} and so on) and develop some theory to study their module categories.

\subsection{Basic notions}
\label{subsec:Z-algebras}

\begin{dfn}
	\label{dfn:Z-algebra}
	A \emph{$\Zbb$-algebra} is a $k$-algebra (without unit) together with a $k$-vector space decomposition $A = \bigoplus_{i,j\in \Zbb} A_{i,j}$ such that the multiplication has the property $A_{ij}A_{jk} \subset A_{ik}$ and $A_{ij}A_{kl}=0$ if $j \neq k$.
	We require that each subalgebra $A_{ii}$ has a unit $e_{i,A}$ (called a local unit) that acts as a right identity on $A_{ji}$ and a left identity on $A_{ij}$ for all $j$.
\end{dfn}

If $A$ is clear from the context, we simply write $e_i$ instead of $e_{i,A}$.
Let $A,B$ be $\Zbb$-algebras.
A $\Zbb$-algebra homomorphism $\varphi : A \rightarrow B$ is a $k$-algebra homomorphism $\varphi : A \rightarrow B$ such that $\varphi(A_{ij}) \subset B_{ij}$ and $\varphi(e_{i,A})=e_{i,B}$ for all $i,j \in \Zbb$.
$A$ is called \emph{connected} if $A_{ij}=0$ for all $i>j$ and $A_{ii}=k$ for all $i$.
$A$ is called \emph{locally finite} if $A_{ij}$ is a finite $k$-module for all $i,j$.
We define a connected $\Zbb$-algebra $K$ by 
\[
K_{ij} := 
\begin{cases}
k & \text{if } i=j,\\
0 & \text{otherwise.}
\end{cases}
\]
We also define the \emph{opposite $\Zbb$-algebra} $A^{\op}$ of $A$ by
\[
A^{\op}_{ij} := A_{ji} \quad \text{for all } i,j \in \Zbb 
\]
with the multiplication defined as $a \cdot b \coloneqq b a \in A^{\op}_{ik}$ for all $a \in A^{\op}_{ij}, b \in A^{\op}_{jk}$.

Let $A$ be a $\Zbb$-algebra.
A \emph{graded right $A$-module} is a right $A$-module $M$ together with a decomposition $M = \bigoplus_{i \in \Zbb} M_i$ such that $M_i A_{ij} \subset M_j$ for all $i,j \in \Zbb$, $e_i$ acts as the identity on $M_i$ for all $i \in \Zbb$ and $M_iA_{jk}=0$ if $i \neq j$.
The \emph{homomorphism of graded right $A$-modules} $\varphi : M \rightarrow N$ is a homomorphism of right $A$-modules such that $\varphi(M_i) \subset N_i$ for all $i \in \Zbb$.
A \emph{graded left $A$-module} and \emph{a homomorphism of graded left $A$-modules} are defined similarly.
We denote the category of graded right $A$-modules by $\Grr{A}$.
The category of graded left $A$-modules is naturally equivalent to the category $\Grl{A}$ of graded right $A^{\op}$-modules (\cite[Proposition 2.2]{mori2021local}).
So, we often identify the category $\Grl{A}$ with the category of left graded $A$-modules.
We write $\Homr{A}(-,-)$ for $\Hom_{\Grr{A}}(-,-)$.
In fact, the category of unitary ungraded right $A$-modules is equivalent to the category $\Grr{A}$ (\cite[Lemma 2.11]{mori2025categorical}, \cite{sierra2011galgebras}, \cite{vandenbergh2011noncommutative}).
Here, an ungraded right $A$-module $M$ is called \emph{unitary} if $MA = M$.
So, the notation $\Homr{A}(-,-)$ makes sense.
In addition, for any two complexes $M,N$ in $\Grr{A}$, we define the \emph{hom-complex} $\Homrcpx{A}({M},{N})$ by
\begin{align*}
	\Homrcpx{A}({M},{N}) &\coloneqq \bigoplus_{n \in \Zbb} \; \prod_{p \in \Zbb} \Homr{A}(M^p,N^{p+n}), \\
	d^n &= \prod_{p \in \Zbb} (d_{M}^{p-1}+(-1)^{n+1}d_{N}^{p+n}).
\end{align*}
For simplicity, we often use the same notation as $\Homr{A}(-,-)$ if the context is clear.

Let $A,B$ be $\Zbb$-algebras.
A \emph{bigraded $A$-$B$-bimodule} is an $A$-$B$-bimodule $M$ together with a decomposition $M = \bigoplus_{i,j \in \Zbb} M_{ij}$ such that $e_{i,A}M =\bigoplus_{k\in\Zbb}M_{ik}$ is a graded right $B$-module and $Me_{j,B} =\bigoplus_{k\in\Zbb}M_{kj}$ is a graded left  $A$-module for all $i,j \in \Zbb$.
If $A$ is connected, we define $A_{\geq n} \coloneqq\bigoplus_{j-i \geq n} A_{ij}$.
$A$ and $ A_{\geq n}$ are naturally bigraded $A$-$A$-bimodules for each $n \in \Nbb$.
We often write $\mathfrak{m}_A$ instead of $A_{\geq 1}$.
For a graded right $A$-module $M$ and a graded left $B$-module $N$, we define $M \ten{k} N$ as a bigraded $A$-$B$-bimodule by $(M \ten{k} N)_{ij} \coloneqq M_{i} \ten{k} N_{j}$.
A \emph{homomorphism of bigraded $A$-$B$-bimodules} $\varphi : M \rightarrow N$ is a homomorphism of $A$-$B$-bimodules such that $\varphi(M_{ij}) \subset N_{ij}$ for all $i,j \in \Zbb$.
The category of bigraded $A$-$B$-bimodules is denoted by $\Grlr{A}{B}$.
We write $\Homlr{A}{B}(-,-)$ for $\Hom_{\Grlr{A}{B}}(-,-)$.
We also define the \emph{hom-complex} $\Homlrcpx{A}{B}(M,N)$ for complexes of bigraded $A$-$B$-bimodules $M,N$ in the same way as above.
We often use the same notation as $\Homlr{A}{B}(-,-)$ if the context is clear.
It is well-known that $\Grr{A}$ and $\Grlr{A}{B}$ are Grothendieck categories (see \cite[page 3988]{vandenbergh2011noncommutative}, \cite[page 6]{mori2025categorical}).

We define the natural restriction functors
\begin{align*}
	\Resl{A} : \Grlr{A}{B} \rightarrow \Grl{A}, \quad M \mapsto \bigoplus_{i \in \Zbb} Me_{i,B}
\end{align*}
and
\begin{align*}
\Resr{B} : \Grlr{A}{B} \rightarrow \Grr{B}, \quad M \mapsto \bigoplus_{i \in \Zbb}e_{i,A}M.
\end{align*}
We can also define additional restriction functors $\Resll{A}{i}$, $\Resrr{B}{i}$ by
\[\Resll{A}{i} : \Grlr{A}{B} \rightarrow \Grl{A}, \quad M \mapsto Me_{i,B}\]
and
\[\Resrr{B}{i} : \Grlr{A}{B} \rightarrow \Grr{B}, \quad M \mapsto e_{i,A}M.\]

Let $A$ be a $\Zbb$-algebra.
We define a graded right $A$-module $P_{i,A}$ for each $i \in \Zbb$ by $P_{i,A} \coloneqq e_i A$.
We define a graded left $A$-module $Q_{i,A}$ for each $i \in \Zbb$ by $Q_{i,A} := A e_i$.
If $A$ is connected, we also define $S_{i,A} \coloneqq e_i A e_i=A_{ii}$ for each $i \in \Zbb$, which is naturally a graded right and left $A$-module.
If $A$ is clear from the context, we simply write $P_i$, $Q_i$ and $S_i$ instead of $P_{i,A}$, $Q_{i,A}$ and $S_{i,A}$.
Note that $\{P_i\}_i$ is a set of projective generators in $\Grr{A}$ and $\{Q_i\}_i$ is a set of projective generators in $\Grl{A}$.
If $A$ is connected, $\{S_i\}_i$ is the set of simple objects in $\Grr{A}$ and $\Grl{A}$. 
Moreover, $\{Ae_{i,A} \ten{k} e_{j,B}B\}_{i,j}$ is a set of projective generators in $\Grlr{A}{B}$ (\cite[Lemma 2.4]{mori2021local}).

We define $\Irr{A}{B}{i} : \Grr{B} \rightarrow \Grlr{A}{B}$ and
$\Ill{A}{B}{i} : \Grl{A} \rightarrow \Grlr{A}{B}$ by
\[
\Irr{A}{B}{i}(M) := Ae_{i,A} \ten{k} M, \quad 
\Ill{A}{B}{i}(M) := M \ten{k} e_{i,B}B.
\]
We also define $\Jrr{A}{B}{i} : \Grr{B} \rightarrow \Grlr{A}{B}$ and $\Jll{A}{B}{i} : \Grl{A} \rightarrow \Grlr{A}{B}$ by
\begin{align*}
\Jrr{A}{B}{i}(M) &= \bigoplus_{m,n \in \Zbb} \Jrr{A}{B}{i}(M)_{m,n} := \bigoplus_{m,n \in \Zbb} \Hom_k(A_{i,m}, M_n), \\
\Jll{A}{B}{i}(M) &= \bigoplus_{m,n \in \Zbb} \Jll{A}{B}{i}(M)_{m,n} := \bigoplus_{m,n \in \Zbb} \Hom_k(B_{n,i}, M_m).
\end{align*}

\begin{lem}
	\label{lem:restriction}
	Let $A,B$ be $\Zbb$-algebras.
	\begin{enumerate}
		\item $\Ill{A}{B}{i}$ is a left adjoint to $\Resll{A}{i}$
		and $\Irr{A}{B}{i}$ is a left adjoint to $\Resrr{B}{i}$.

	\item $\Jll{A}{B}{i}$ is a right adjoint to $\Resll{A}{i}$ and $\Jrr{A}{B}{i}$ is a right adjoint to $\Resrr{B}{i}$.

	\item If $M$ is a projective (resp. injective) bigraded $A$-$B$-bimodule, then $\Resll{A}{i}(M)$ is projective (resp. injective) in $\Grl{A}$ and $\Resrr{B}{i}(M)$ is projective (resp. injective) in $\Grr{B}$.

	\end{enumerate}

\end{lem}

\begin{proof}
	Item 1 is proved in the proof of \cite[Lemma 2.3]{mori2021local}.

	Item 2 follows from the following isomorphisms:
	\begin{align*}
	\Homl{A}(Pe_i, M) &\cong \Homlr{A}{B}(P, \Jll{A}{B}{i}(M)) \\
	f &\mapsto (p_{mn} \mapsto (b_{ni} \mapsto f(p_{mn}b_{ni}))), \\
	(g(p_{mi})(e_{i,B}) \mapsfrom p_{mi}) &\mapsfrom g,
	\end{align*}
	where $P \in \Grlr{A}{B}$, $M \in \Grl{A}$, $p_{mn} \in P_{mn}$ and $b_{ni} \in B_{ni}$.

	When $M$ is an injective bigraded $A$-$B$-bimodule, item 3 is \cite[Lemma 2.3]{mori2021local} (cf. \cite[Lemma 12.29.1]{stacks-project}).
	When $M$ is a projective bigraded $A$-$B$-bimodule, item 3 follows from the fact that $M \in \Grl{A}$ (resp. $\Grr{B}$) is projective if and only if $M$ is a direct summand of a direct sum of objects of the form $Ae_{i,A}$ (resp. $e_{i,B}B$).
\end{proof}

For a graded right $A$-module $M$, we define the \emph{graded $k$-dual} of $M$ by 
\[
\D{M} \coloneqq \bigoplus_{i \in \Zbb} \D{M_i},
\]
where $\D{M_i} := \Hom_k(M_i,k)$ is the $k$-linear dual of $M_i$.
$\D{M}$ is naturally a graded left $A$-module via $af(m) \coloneqq f(ma)$ for $f \in (\D{M})_i=\D{M_i}$, $m \in M_j$, and $a \in A_{ji}$.
For a bigraded $A$-$B$-bimodule $M$, we define the graded $k$-dual $\D{M}$ of $M$ by $\D{M} \coloneqq \bigoplus_{i,j \in \Zbb} \D{M_{ji}}$, where $\D{M_{ij}} := \Hom_k(M_{ij},k)$ is the $k$-linear dual of $M_{ij}$.
$\D{M}$ is naturally a bigraded $B$-$A$-bimodule.

For a $\Zbb$-graded $k$-module $M = \bigoplus_{i \in \Zbb} M_i$, we define the \emph{shift} $M(n)$ of $M$ for each $n \in \Zbb$ by $M(n)_i \coloneqq M_{n+i}$.
For a $\Zbb^2$-graded $k$-module $M = \bigoplus_{i,j \in \Zbb} M_{ij}$, we define the \emph{shift} $M(n,m)$ of $M$ for each $n,m \in \Zbb$ by $M(n,m)_{i,j} \coloneqq M_{n+i,m+j}$.
Then, a shift $A(n,n)$ of a $\Zbb$-algebra $A$ is also a $\Zbb$-algebra (\cite[Lemma 2.14]{mori2025categorical}).
If $M \in \Grr{A}$, then the shift $M(n)$ is naturally a graded right $A(n,n)$-module, which is not a graded right $A$-module in general.

For $\Zbb$-algebras $A,B$ and $C$, a homomorphism of $\Zbb$-algebras $\varphi:A \rightarrow B$ and $M \in \Grr{B}$, we define a \emph{twist} $M_{\varphi} \in \Grr{A}$ of $M$ by $M_{\varphi}  \coloneqq M$ as a graded $k$-module with the action defined by $ma \coloneqq m\varphi(a)$ for all $m \in M, a \in A$.
For another $\Zbb$-algebra $D$, a homomorphism $\psi:C \rightarrow D$ and $M \in \Grlr{D}{B}$, we define a twist ${}_{\psi}M_{\varphi} \in \Grlr{C}{A}$ by ${}_{\psi}M_{\varphi} \coloneqq M$ as a bigraded $k$-module with the action defined by $cma \coloneqq \psi(c)m\varphi(a)$ for all $m \in M, a \in A, c \in C$.
When $C=D$ and $\psi=\Id_C$, then we simply write $M_\varphi$.

If $A$ is isomorphic to a shift $A(n,n)$ for some $n \in \Zbb$, then we say that $A$ is \emph{$n$-periodic}.
If $A$ is $n$-periodic and $\varphi:A \rightarrow A(n,n)$ is an isomorphism of $\Zbb$-algebras, then we have the following autoequivalence of the category $\Grr{A}$ (\cite[Section 2.5]{mori2025categorical}):
\begin{align*}
\Grr{A} &\xrightarrow{\cong} \Grr{A}, \quad M \mapsto M(n)_{\varphi}, \\
\Grlr{B}{A} &\xrightarrow{\cong} \Grlr{B}{A}, \quad M \mapsto M(0,n)_{\varphi}, \\
\Grlr{A}{A} &\xrightarrow{\cong} \Grlr{A}{A}, \quad M \mapsto {}_{\varphi}M(n,n)_{\varphi}.
\end{align*}
In addition, $M \in \Grlr{A}{A}$ is called \emph{$n$-periodic} if $M \cong {}_{\varphi}M(n,n)_{\varphi}$ for some $n \in \Zbb$ and an isomorphism $\varphi: A \rightarrow A(n,n)$.

Let $A$ be a $\Zbb$-graded algebra, as in \cref{subsec:background}, we can define the associated $\Zbb$-algebra $\zgrz{A}$ by $\zgrz{A}_{ij} \coloneqq A_{j-i}$ for all $i,j \in \Zbb$.
Then, $\zgrz{A}$ is $1$-periodic.

\begin{lem}[{\cite[Section 2]{sierra2011galgebras}, \cite[Section 1]{polishchuk2005noncommutative}, \cite[Lemma 2.17]{mori2025categorical}}]
	\label{lem:associated-Z-algebras}
	Let $A$ be a $\Zbb$-graded algebra.
	Then, we have the following equivalences of categories
\begin{align*}
	\Grr{A} &\xrightarrow{\cong} \Grr{\zgrz{A}}, \quad M \mapsto \zgrz{M}  = \bigoplus_{i \in \Zbb} \zgrz{M}_i := \bigoplus_{i \in \Zbb} M_{i}, \\
	\Grl{A} &\xrightarrow{\cong}  \Grl{\zgrz{A}}, \quad M \mapsto \zgrz{M} = \bigoplus_{i \in \Zbb} \zgrz{M}_{i} := \bigoplus_{i \in \Zbb} M_{-i}, 
\end{align*}
where $\Grr{A}$ and $\Grl{A}$ are the categories of graded right $A$-modules and graded left $A$-modules, respectively. 
In addition, we have the following functor
\begin{align*}
	\Grlr{A}{B} &\rightarrow \Grlr{\zgrz{A}}{\zgrz{B}}, \quad M \mapsto \zgrz{M} = \bigoplus_{i,j \in \Zbb} \zgrz{M}_{i,j} := \bigoplus_{i,j \in \Zbb} M_{j-i},
\end{align*}
where $B$ is another $\Zbb$-graded algebra and $\Grlr{A}{B}$ are the categories of graded $A$-$B$-bimodules.
\end{lem}
Note that although the functors $\Grr{A} \rightarrow \Grr{\zgrz{A}}$ and $\Grl{A} \rightarrow \Grl{\zgrz{A}}$ are equivalences, the functor $\Grlr{A}{B} \rightarrow \Grlr{\zgrz{A}}{\zgrz{B}}$ is not an equivalence in general.

\subsection{Module categories and derived categories}

Let $A, B$ and $C$ be $\Zbb$-algebras.
For a bigraded $B$-$A$-bimodule $M$ and a graded right $A$-module $N$, we define a graded right $B$-module $\Homintr{A}(M,N)$ called an \emph{internal Hom module} by
\[
\Homintr{A}(M,N) = \bigoplus_{i \in \Zbb} \Homintr{A}(M,N)_i \coloneqq \bigoplus_{i \in \Zbb} \Homr{A}(e_iM,N),
\]
where its module structure is given by
$(fb)(m) \coloneqq f(bm) \ \text{for } f \in \Homintr{A}(M,N)_i, b \in B_{ij}, m \in e_jM$.
For a graded right $A$-module $M$ and a bigraded $B$-$A$-bimodule $N$, we define a graded left $B$-module $\Homintr{A}(M,N)$ by
\[\Homintr{A}(M,N) = \bigoplus_{i \in \Zbb} \Homintr{A}(M,N)_i \coloneqq \bigoplus_{i \in \Zbb}\Homr{A}(M,e_iN),
\]
where its module structure is given by $(bf)(m) \coloneqq bf(m) \ \text{for } f \in \Homintr{A}(M,N)_i, b \in B_{ji}, m \in M$.
If $M$ is a bigraded $B$-$A$-bimodule and $N$ is a bigraded $C$-$A$-bimodule, we define a bigraded $C$-$B$-bimodule $\Homintr{A}(M,N)$ by
\[
\Homintr{A}(M,N) = \bigoplus_{i,j \in \Zbb} \Homintr{A}(M,N)_{i,j} \coloneqq \bigoplus_{i,j \in \Zbb} \Homr{A}(e_jM,e_iN).
\]

We also define \emph{internal Hom-complexes} for complexes of modules in the same way.
But, we often use the same notation as the internal Hom modules.

\begin{rmk}
If $M$ is an object in $\Grlr{A}{B}$, then
\[
\Homintr{K}(M,K) \cong M'
\] 
as objects in $\Grlr{B}{A}$.
\end{rmk}

Let $M$ be a graded right $A$-module and $N$ be a graded left $A$-module.
Then, we define the \emph{tensor product} $M \otimes_A N$ by 
\begin{align*}
M \ten{A} N &= \Cok \left(\bigoplus_{i,j} M_i \otimes_{A_{ii}} A_{ij} \ten{A_{jj}} N_j \xrightarrow{\psi} \bigoplus_{k \in \Zbb} M_k \ten{A_{kk}} N_k \right), \\
&\psi(m \otimes a \otimes n) \coloneqq m \otimes a n- m a \otimes n,
\end{align*}
where $\Cok$ is the cokernel of the map inside the parentheses.

For $M$ a graded right $A$-module and $N$ a bigraded $A$-$B$-bimodule, we define a graded right $B$-module $M \ten{A} N$ by
\[
M \tenint{A} N \coloneqq \bigoplus_{i \in \Zbb} M \ten{A} Ne_i.
\]
For $M$ a bigraded $B$-$A$-bimodule and $N$ a graded left $A$-module, we define a graded left $B$-module $M \ten{A} N$ by
\[
M \tenint{A} N \coloneqq \bigoplus_{i \in \Zbb} e_iM \ten{A} N.
\]
For $M$ a bigraded $B$-$A$-bimodule and $N$ a bigraded $A$-$C$-bimodule, we define a bigraded $B$-$C$-bimodule $M \ten{A} N$ by
\[
M \tenint{A} N \coloneqq \bigoplus_{i,j \in \Zbb} e_iM \ten{A} Ne_j.
\]

The following proposition is basic but important.

\begin{prop}[{\cite[Section 4 and 5]{mori2021local}, \cite[Lemma 2.4]{mori2025categorical}}]
	\label{prop:tensor-product-properties}
	Let $A,B$ be $\Zbb$-algebras.
	Then, the following hold:
	\begin{enumerate}
		\item For any $M$ in $\Grr{A}$, $M \tenint{A} A \cong M$ so that $(M \tenint{A} A)_i \cong M_i$ for all $i \in \Zbb$.
		\item For any $N$ in $\Grl{A}$, $A \tenint{A} N \cong N$ so that $(A \tenint{A} N)_i \cong N_i$ for all $i \in \Zbb$.
		\item For any $M$ in $\Grr{A}$, any $N$ in $\Grr{B}$ and any $L$ in $\Grlr{A}{B}$, we have 
		\[
		\Homr{B}(M \tenint{A} L,N) \cong \Homr{A}(M, \Homintr{B}(L,N)).
		\]
	\end{enumerate}
\end{prop}

Let $\Ccal$ be an abelian category.
Then, we denote by $K(\Ccal)$ the homotopy category of complexes in $\Ccal$.
We denote by $K^+(\Ccal)$, $K^-(\Ccal)$ and $K^b(\Ccal)$ the full subcategories of $K(\Ccal)$ consisting of complexes with bounded below, bounded above and bounded, respectively.
We denote by $D(\Ccal)$ the derived category of $\Ccal$.
We denote by $D^+(\Ccal)$, $D^-(\Ccal)$ and $D^b(\Ccal)$ the full subcategories of $D(\Ccal)$ consisting of complexes with bounded below, bounded above and bounded, respectively.

To consider the derived functors of the above functors, we define the notions of \emph{K-projective} and \emph{K-injective} objects in $\Grr{A}$ and $\Grlr{A}{B}$.
These notions are well-known in the theory of homological algebra, but we give the definitions for the sake of completeness.

\begin{dfn}
	\label{dfn:K-projective, K-injective}
	Let $A,B$ be $\Zbb$-algebras.
	\begin{enumerate}
		\item An object $P \in D(\Grlr{A}{B})$ (resp. $D(\Grr{B})$) is called \emph{K-projective} if for any acyclic complex $M \in D(\Grlr{A}{B})$ (resp. $D(\Grr{B})$), the Hom complex $\Homlr{A}{B}(P,M)$ (resp. $\Homr{B}(P,M)$) is acyclic.
		\item An object $I \in D(\Grlr{A}{B})$ (resp. $D(\Grr{B})$) is called \emph{K-injective} if for any acyclic complex $M \in D(\Grlr{A}{B})$ (resp. $D(\Grr{B})$), the Hom complex $\Homlr{A}{B}(M,I)$ (resp. $\Homr{B}(M,I)$) is acyclic.
		\item An object $F \in D(\Grr{B})$ is called \emph{K-flat} if for any acyclic complex $M \in D(\Grl{B})$, the tensor product $F \ten{B} M$ is acyclic.
	\end{enumerate}
\end{dfn}

\begin{eg}
A bounded above complex of projective objects in $\Grr{B}$ (resp. $\Grlr{A}{B}$) is K-projective.
A bounded below complex of injective objects in $\Grr{B}$ (resp. $\Grlr{A}{B}$) is K-injective.
\end{eg}

\begin{prop}
	\label{prop:restriction-K-projective-K-injective}
	Let $A,B$ be $\Zbb$-algebras.
	Then, 
	\begin{enumerate}
		\item If $P$ is a K-projective object in $D(\Grlr{A}{B})$, then $\Resll{A}{i}(P)$  (resp. $\Resrr{B}{i}(P)$) is a K-projective object in $D(\Grl{A})$ (resp. $D(\Grr{B})$) for all $i \in \Zbb$.
		\item If $I$ is a K-injective object in $D(\Grlr{A}{B})$, then $\Resll{A}{i}(I)$ (resp. $\Resrr{B}{i}(I)$) is a K-injective object in $D(\Grl{A})$ (resp. $D(\Grr{B})$) for all $i \in \Zbb$.
		\item If $P$ is a K-projective object in $D(\Grr{B})$, then $P$ is K-flat in $D(\Grr{B})$.
	\end{enumerate}
\end{prop}

\begin{proof}
	(1) 
	From \cref{lem:restriction}, for any acyclic complex $M \in D(\Grl{A})$, we have 
	\begin{align*}
		\Homlcpx{A}(\Resll{A}{i}(P), M) &\cong \Homlrcpx{A}{B}(P, \Jll{A}{B}{i}(M))
	\end{align*}
	The statement follows from this isomorphism and the definition of K-projective objects.

	When we prove that $\Resrr{B}{i}(P)$ is K-projective, we can show the statement in the same way.

	(2) We can show the statement by using the adjoint property in \cref{lem:restriction} 
	\[
	\Homlrcpx{A}{B}(\Ill{A}{B}{i}(-), I) \cong \Homlcpx{A}(-, \Resll{A}{i}(I))
	\]
	and the definition of K-injective objects in the same way as (1).

	(3) If $M$ is an acyclic complex in $D(\Grl{B})$, then we have 
	\begin{align*}
		\Homrcpx{B}(P, \Homintr{K}(\Irr{B}{K}{i}(M), e_i K)) \cong \Homrcpx{K}(P \tenint{B} \Irr{B}{K}{i}(M), e_i K)
	\end{align*}
	by \cref{prop:tensor-product-properties}.
	Here, for any $i \in \Zbb$, 
	\begin{enumerate}[label=(\alph*)]
		\item $\Homrcpx{K}(P \tenint{B} \Irr{B}{K}{i}(M), e_i K)$ is acyclic if and only if $P \ten{B} M$ is acyclic,
		\item if $M$ is acyclic then, $\Homintr{K}(\Irr{B}{K}{i}(M), e_i K)$ is acyclic.
	\end{enumerate}
	Thus, we have that $P$ is K-flat in $D(\Grr{B})$ (see also \cite[Proposition 10.3.4]{yekutieli2019derived}).
\end{proof}

The following proposition is standard. 
For the sake of completeness, we give an explicit statement and a proof of the proposition.

\begin{prop}
	\label{prop:existence-K-projective-K-injective}
	Let $A,B$ be $\Zbb$-algebras.
	Then, 
	\begin{enumerate}
		\item For any object $M$ in $D(\Grlr{A}{B})$ (resp. $D(\Grr{B})$), there exists a K-projective object $P$ in $D(\Grlr{A}{B})$ (resp. $D(\Grr{B})$) which is quasi-isomorphic to $M$.
		\item For any object $M$ in $D(\Grlr{A}{B})$ (resp. $D(\Grr{B})$), there exists a K-injective object $I$ in $D(\Grlr{A}{B})$ (resp. $D(\Grr{B})$) which is quasi-isomorphic to $M$.
	\end{enumerate}
\end{prop}

\begin{proof}
	Because $\Grr{A}$, $\Grl{A}$ and $\Grlr{A}{B}$ satisfy (AB4) and (Ab4*), the existence of K-projective and K-injective resolutions follows, for example, from \cite{bokstedt1993homotopy} (cf. \cite{spaltenstein1988resolutions}, \cite{gortz2023algebraic}).
	As for (2), the claim also comes from \cite{serpe2003resolution} since $\Grr{A}$, $\Grl{A}$ and $\Grlr{A}{B}$ are Grothendieck categories.
\end{proof}

\begin{rmk}
\label{rmk:existence-K-projective-K-injective}
From \cite{bokstedt1993homotopy}, we can take a K-projective (resp. K-injective) resolution $P \rightarrow M$ (resp. $I \rightarrow M$) such that each $P^j$ (resp. $I^j$) is a projective (resp. an injective) object for every $M$ in $D(\Grlr{A}{B})$ or $D(\Grr{B})$.
\end{rmk}

The following proposition is a generalization of \cite[Proposition 6.6 and 6.8]{mori2021local}.
\begin{prop}
\label{prop:derived-functors}
Let $A, B$ and $C$ be $\Zbb$-algebras.
Then, the following hold:
\begin{enumerate}
	\item $\Homintr{A}(-,-)$ has the right derived bifunctor
	\begin{align*}
	\RHomintr{A}(-,-) : D(\Grlr{B}{A})^{\op} \times D(\Grlr{C}{A}) &\rightarrow D(\Grlr{C}{B}).
	\end{align*}

	If $P$ is a K-projective object in $D(\Grlr{B}{A})$ or $I$ is a K-injective object in $D(\Grlr{C}{A})$, then
	\begin{align*}
		\RHomintr{A}(P,I) &\cong \Homintr{A}(P,I)
	\end{align*}
	in $D(\Grlr{C}{B})$.

	\item $(- \, \tenint{A}-)$ has the left derived bifunctor $(- \, \Ltenint{A}-)$
	\begin{align*}
	(- \, \Ltenint{A}-) : D(\Grlr{B}{A}) \times D(\Grlr{A}{C}) &\rightarrow D(\Grlr{B}{C}).
	\end{align*}
	If $P_1$ is a K-projective object in $D(\Grlr{B}{A})$ or $P_2$ is a K-projective object in $D(\Grlr{A}{C})$, then
	\begin{align*}
		P_1 \Ltenint{A} P_2 &\cong P_1 \tenint{A} P_2
	\end{align*}
	in $D(\Grlr{B}{C})$.
\end{enumerate}

\end{prop}

\begin{proof}
	(1) We want to apply \cite[Theorem 9.3.1]{yekutieli2019derived} (cf. \cite[Proposition 9.3.10, Theorem 12.2.1]{yekutieli2019derived}, \cite[Proposition 17.3, 17.4]{hoshino1997derived}, \cite[Lemma 5.13]{guisado2025gospel}, \cite{hartshorne1966residues}).
	We need to check the following conditions:
	\begin{enumerate}[label=(\alph*)]
		\item $\Grlr{A}{C}$ (resp. $\Grlr{B}{A}$) has enough $K$-injective objects (resp. $K$-projective objects).
		\item Let $f: P_1 \rightarrow P_2$ and $g: I_1 \rightarrow I_2$ be isomorphisms in $D(\Grlr{B}{A})$ and $D(\Grlr{C}{A})$, respectively.
		If either $P_1, P_2$ are K-projective or $I_1, I_2$ are K-injective, then the induced morphism naturally defined by $f$ and $g$
		\[
		\Homintr{A}({P_1},g) \circ \Homintr{A}(f,{I_1}) :  \Homintr{A}(P_2,I_1) \rightarrow \Homintr{A}(P_1,I_2)
		\]
		is an isomorphism in $D(\Grlr{C}{B})$.
	\end{enumerate}
The condition (a) follows from \cref{prop:existence-K-projective-K-injective}.
So, it is enough to check (b).
We show that $\Homintr{A}({P_1},g)$ and $\Homintr{A}(f,{I_1})$ are quasi-isomorphisms.

We assume that $I_1, I_2$ are K-injective.
It is enough to check it degree by degree.
Firstly, note that $e_iI_1, e_iI_2$ are K-injective objects in $\Grr{A}$ for all $i \in \Zbb$ from \cref{prop:restriction-K-projective-K-injective}.

As for $\Homintr{A}({P_1},g)$, since $e_ig: e_iI_1 \rightarrow e_iI_2$ is a homotopy equivalence (\cite[Proposition 1.5]{spaltenstein1988resolutions}, \cite[Proposition 2.3]{serpe2003resolution}), we have
\begin{align*}
	H^l(\Homr{A}(e_j{P_1},e_ig)) : H^l(\Homr{A}(e_jP_1,e_iI_1)) \xrightarrow{\sim} H^l(\Homr{A}(e_jP_1,e_iI_2))
\end{align*}
for all $i,j,l \in \Zbb$.
So, the isomorphism follows from 
\begin{align*}
	H^l(\Homintr{A}(P_1,I_1))_{i,j} \cong H^l(\Homr{A}(e_jP_1,e_iI_1)), \\
	H^l(\Homintr{A}(P_1,I_2))_{i,j} \cong H^l(\Homr{A}(e_jP_1,e_iI_2))
\end{align*}
and that $H^l(\Homr{A}(e_j{P_1},e_ig))$ corresponds to the morphism $H^l(\Homintr{A}({P_1},g))_{i,j}$ for all $i,j,l \in \Zbb$.

As for $\Homintr{A}(f,{I_1})$, let $P_3$ be the cone of $f$.
Then, $P_3$ is an acyclic complex.
So, we have
\begin{align*}
	H^l(\Homr{A}(e_jP_3,e_iI_1)) \cong 0
\end{align*}
for all $i,j,l \in \Zbb$.
This induces the isomorphism
\begin{align*}
	H^l(\Homr{A}(e_jf,e_i{I_1})) : H^l(\Homr{A}(e_jP_2,e_iI_1)) \xrightarrow{\sim} H^l(\Homr{A}(e_jP_1,e_iI_1))
\end{align*}
for all $i,j,l \in \Zbb$.
Therefore, we have the isomorphism from 
\begin{align*}
	H^l(\Homintr{A}(P_2,I_1))_{i,j} \cong H^l(\Homr{A}(e_jP_2,e_iI_1)), \\
	H^l(\Homintr{A}(P_1,I_1))_{i,j} \cong H^l(\Homr{A}(e_jP_1,e_iI_1))
\end{align*}
and that $H^l(\Homr{A}(e_jf,e_i{I_1}))$ corresponds to the morphism $H^l(\Homintr{A}(f,{I_1}))_{i,j}$ for all $i,j,l \in \Zbb$.
In the case that $P_1, P_2$ are K-projective, the claim can be shown in the same way.

(2) We can show the statement in the same way as (1).
	However, note that we use (3) of \cref{lem:restriction}.
\end{proof}

\begin{rmk}
\label{rmk:derived-functors}
From the proof of \cref{prop:derived-functors}, if you want to calculate $\RHomintr{A}(M,N)$ for $M \in D(\Grlr{B}{A})$ and $N \in D(\Grlr{C}{A})$, for example, it is enough to take a resolution of $N \rightarrow I$ in $D(\Grlr{C}{A})$ such that each $e_i I$ is K-injective in $D(\Grr{A})$ for all $i \in \Zbb$.
\end{rmk}

$M \in \Grr{A}$ is called \emph{free} if it is isomorphic to a direct sum of objects in $\{e_iA\}_{i \in \Zbb}$.
Let $\Pcal_A$  be the set of finite direct sums of objects in $\{e_iA\}_{i \in \Zbb}$.
Then, we say that a graded right $A$-module $M$ is \emph{finitely generated} if there exists an epimorphism $P \rightarrow M$ with $P \in \Pcal_A$.
We denote the category of finitely generated graded right (resp. left) $A$-modules by $\grr{A}$ (resp. $\grl{A}$).
A graded right A-module M is said to be \emph{locally finite} if each component $M_i$ is a finite-dimensional k-vector space for all \(i \in \Zbb\).
We denote by \(\Grlfr{A}\) the category of locally finite graded right A-modules.
For another \(\Zbb\)-algebra B, we similarly define locally finite bigraded A-B-bimodules and denote their category by \(\Grlflr{A}{B}\).
Moreover, we say that a graded right $A$-module $M$ is \emph{cofinite} if $\D{M}$ is a finitely generated graded left $A$-module.
We denote the category of cofinite graded right $A$-modules by $\grcofr{A}$.

From \cref{subsec:Z-algebras}, recall that for a connected $\Zbb$-algebra $A$, we have the simple graded right (left) $A$-module $S_i$ defined by $S_i = e_i A e_i$ for each $i \in \Zbb$.
The set $\{S_i\}_{i \in \Zbb}$ is the set of simple graded right (left) $A$-modules.

\begin{dfn}[{\cite[Definition 3.1]{mori2021local},\cite[Definition 3.1, 3.3]{mori2021acategorical}}]
	\label{dfn:minimal-free resolution, Ext-finite}
	Let $A$ be a connected $\Zbb$-algebra.
	\begin{enumerate}
		\item A \emph{finitely generated minimal free resolution} of a graded right $A$-module $M$ is an exact sequence 
		\[
		\cdots \xrightarrow{d_3} F_2 \xrightarrow{d_2} F_1 \xrightarrow{d_1} F_0 \xrightarrow{d_0} M \rightarrow 0
		\]
		such that each $F_i$ is a finite free $A$-module and $\Img(d_i) \subset F_{i-1}A_{>0}$.
		\item $A$ is called \emph{right (resp. left) Ext-finite} if there exists a finitely generated minimal free resolution of $S_i$ in $\Grr{A}$ (resp. $\Grl{A}$) for any $i \in \Zbb$.
		If $A$ is right and left Ext-finite, then we say that $A$ is \emph{Ext-finite}.
	\end{enumerate}
\end{dfn}

\begin{dfn}[{\cite[Definition 2.1]{vandenbergh2011noncommutative}}]
	\label{dfn:noetherian-Z-algebra}
	Let $A$ be a $\Zbb$-algebra.
	$A$ is \emph{right noetherian} (resp. \emph{left noetherian})  if $\Grr{A}$ (resp. $\Grl{A}$) is a locally noetherian category.
	If $A$ is right and left noetherian, then we say that $A$ is \emph{noetherian}.
\end{dfn}

\begin{rmk}[{\cite[Section 2.3]{mori2025categorical}}]
The condition that $A$ is right noetherian is equivalent to the following conditions: 
\begin{itemize}
	\item $e_i A$ is a noetherian object in $\Grr{A}$ for all $i \in \Zbb$, or
	\item $\grr{A}$ is a noetherian category.
\end{itemize}

\end{rmk}

\begin{rmk}[{\cite[Section 3.1]{mori2025categorical}}]
\label{rmk:relation among finite-locally finite-Ext-finite}
	Let $A$ be a connected $\Zbb$-algebra.
	Then, the following holds:
	\[
	A \text{ is right noetherian } \Rightarrow A \text{ is right Ext-finite } \Rightarrow A \text{ is locally finite}.
	\]
\end{rmk}

Let A, B be \(\Zbb\)-algebras.
We denote by \(D_{lf}(\Grr{A})\) the derived category of graded right A-modules whose cohomology modules lie in \(\Grlfr{A}\).
Similarly, we denote by \(D_{lf}(\Grlr{A}{B})\) the derived category of bigraded A-B-bimodules whose cohomology modules lie in \(\Grlflr{A}{B}\).
If A is right noetherian, we write \(D_{f}(\Grr{A})\) and \(D_{cof}(\Grl{A})\) for the derived categories of graded right A-modules whose cohomology modules lie in \(\grr{A}\) and \(\grcofl{A}\), respectively.

The following Matlis duality is important in this paper.
The case of graded algebras is well-known (for example, see \cite[Theorem 15.2.34]{yekutieli2019derived}, \cite[Proposition 3.1]{vandenbergh1997existence}).

\begin{thm}[Matlis duality]
	\label{thm:matlis-duality}
Let $A, B$ be $\Zbb$-algebras.
\begin{enumerate}
	\item The functor $\D{(-)}$ induces equivalences of categories
	\begin{align*}
	\D{(-)} : D_{lf}(\Grr{A})^{\op} &\xlongrightarrow{\sim} D_{lf}(\Grl{A}), \\
    \D{(-)} : D_{lf}(\Grlr{A}{B})^{\op} &\xlongrightarrow{\sim} D_{lf}(\Grlr{B}{A}).
	\end{align*}
\end{enumerate}

We assume that $A$ is right noetherian below.
\begin{enumerate}
\setcounter{enumi}{1}
	\item $\grr{A}$ and $\grcofr{A}$ are Serre subcategories of $\Grlfr{A}$ and $\Grlfl{A}$, respectively.
	\item The functor $\D{(-)}$ induces an equivalence of categories
	\begin{align*}
	    \D{(-)} : D_{f}(\Grr{A})^{\op} \xlongrightarrow{\sim} D_{cof}(\Grl{A}).
	\end{align*}
\end{enumerate}

\end{thm}

\begin{proof}
	(1) There are natural morphisms
	\begin{align*}
		\Id_{\Grlfr{A}^{\op}} &\rightarrow \DD{(-)}, \\
        \Id_{\Grlflr{A}{B}^{\op}} &\rightarrow \DD{(-)},
	\end{align*}
	which are isomorphisms.
	The equivalences extend to the desired ones.

	(2) Since $A$ is right noetherian, $\grr{A}$ forms an abelian subcategory of $\Grlfr{A}$.
	The remaining part about $\grr{A}$ is straightforward.
	As for $\grcofl{A}$, from (1), we have an equivalence between $\Grlfr{A}^{\op}$ and $\Grlfl{A}$.
	This induces an equivalence of categories between $\grr{A}^{\op}$ and $\grcofl{A}$ and we have the claim.

	(3) The argument for (3) is analogous to that of (2).
\end{proof}

\begin{cor}
    \label{cor:matlis-duality}
    Let $A,B$ and $C$ be $\Zbb$-algebras.
    Let $M \in D_{lf}(\Grlr{B}{A})$ (resp. $D_{lf}(\Grr{A})$) and $N \in D_{lf}(\Grlr{C}{A})$ (resp. $D_{lf}(\Grr{A})$).
    Then, there is an isomorphism
    \begin{align*}
        \RHomintr{A}(M,N) \cong \RHomintl{A}(\D{N},\D{M}).
    \end{align*}
	This isomorphism is functorial in $M,N$.
\end{cor}

\begin{proof}
We prove the claim only when $M \in D_{lf}(\Grlr{B}{A})$ and $N \in D_{lf}(\Grlr{C}{A})$.
The result follows from the following functorial isomorphisms
    \begin{align*}
    H^l(\RHomintr{A}(M,N))_{i,j} 
    &\cong \Hom_{D(\Grr{A})} (e_jM, e_iN[l]) \\
    &\cong \Hom_{D(\Grl{A})} (\D{N}e_i[-l], \D{M}e_j) \quad (\text{\cref{thm:matlis-duality}}) \\
    &\cong H^l(\RHomintl{A}(\D{N}, \D{M}))_{i,j}.
    \end{align*}
\end{proof}

\section{Local cohomology}
\label{sec:local-cohomology}
In this section, we study local cohomology for $\Zbb$-algebras and its properties.
We develop a theory on local cohomology for connected $\Zbb$-algebras and generalize known results such as local duality in more general settings (cf. \cite{mori2021local}, \cite{mori2024corrigendum}).
We also consider the $\chi$-condition for $\Zbb$-algebras and give a characterization of the $\chi$-condition by using local cohomology.

\subsection{Definition of local cohomology}

Let $A$ be a connected $\Zbb$-algebra.

For a graded right $A$-module $M$, $M$ is called \emph{right (resp. left) bounded} if $M_{\geq n} =0$ (resp. $M_{\leq n} =0$) for some $n \in \Zbb$.
$M$ is called \emph{bounded} if $M$ is right and left bounded.



\begin{dfn}
	\label{dfn:torsion-functor}
	Let $A$ be a connected $\Zbb$-algebra.
	Let $M \in \Grr{A}$ (resp. $M \in \Grl{A}$).
	\begin{enumerate}
		\item $m \in M$ is called \emph{right torsion (resp. left torsion)} if $m A_{\geq n} = 0$ (resp. $A_{\geq n} m  = 0$) for sufficiently large $n \gg 0$.
		\item The set of all right (resp. left) torsion elements in $M$ is denoted by $\torfun{A}(M)$ (resp. $\torfun{A^{\op}}(M)$) and called the \emph{$\mathfrak{m}_A$-torsion submodule} (resp. \emph{$\mathfrak{m}_{A^{\op}}$-torsion submodule}) of $M$.
		\item $M$ is called \emph{$\mathfrak{m}_A$-torsion} (resp. \emph{$\mathfrak{m}_{A^{\op}}$-torsion}) if $\torfun{A}(M) = M$ (resp. $\torfun{A^{\op}}(M) = M$). 
	\end{enumerate}
		We often omit the words ``right'' and ``left''.
		One may also write torsion over $A$ and $A^{\op}$ for $\mathfrak{m}_A$-torsion or $\mathfrak{m}_{A^{\op}}$-torsion, respectively.
		If the context is clear, we can omit the words ``over $A$'' and ``over $A^{\op}$''.

\end{dfn}


We denote by $\Tor(A)$ the full subcategory of $\Grr{A}$ consisting of torsion modules.
Then, $\torfun{A}(-)$ determines a left exact functor
\[
\torfun{A} : \Grr{A} \longrightarrow \Tor(A),
\]
which is a right adjoint functor of the inclusion functor $\iota: \Tor(A) \hookrightarrow \Grr{A}$.

\begin{prop}[cf. {\cite[Lemma 5.8]{mori2021local}}]
\label{prop:torsion-functor}
Let $A$ be a connected $\Zbb$-algebra.
There is an isomorphism of functors
\[
\torfun{A}(-) \cong \varinjlim_{n}\Homintr{A}(A/A_{\geq n},-).
\]
\end{prop}


We also define a functor $\Soc(-)$ by 
\begin{align*}
\Soc: \Grr{A} \longrightarrow \Grr{A}, \quad
M \longmapsto \Homint_{A}(K, M).
\end{align*}
From the definition, $\Soc(M)$ is a graded right $A$-submodule of $M$. Here, we use the connectedness of $A$ and the isomorphism $\Homint_{A}(A, M) \cong M$. 
In addition, for every $M \in D(\Grr{A})$, we define $\Soc(M)$ in the same way, which is also a subcomplex of graded right $A$-modules of $M$.

Let $B$ be a $\Zbb$-algebra.
In the same way, we can define a left exact functor 
\begin{align*}
\torfun{A} : \Grlr{B}{A} &\longrightarrow \Grlr{B}{A}, \\
M &\longmapsto  \torfun{A}(\bigoplus_i e_iM).
\end{align*}
For a bigraded $B$-$A$-bimodule $M$, this induces $\Resr{A}(\torfun{A}(M)) \in \Tor(A)$.
Note that we have $\bigoplus_{i} \torfun{A}(e_iM) \cong \torfun{A}(\bigoplus_i e_iM)$ because $\torfun{A}$ is a right adjoint of $\iota$ and so it commutes with direct limits (\cite[Lemma 4.24.5]{stacks-project}).

\begin{rmk}
	\label{rmk:torsion-functor}
	In \cite[Lemma 5.8, Lemma 6.9]{mori2021local}, the authors assume that $A$ is right Ext-finite, but this condition is not necessary for \cref{prop:torsion-functor} and to define the functor $\torfun{A}(-)$ for bigraded $B$-$A$-bimodules.

\end{rmk}

Since $\torfun{A}(-)$ is a left exact functor and  
$D(\Grr{A})$ has enough K-injective complexes by \cref{prop:existence-K-projective-K-injective}, we can define the right derived functor of $\torfun{A}(-)$:
\begin{align*}
	\Rtorfun{A}(-) : D(\Grr{A}) \longrightarrow D(\Grr{A}).
\end{align*}
For $M \in D(\Grr{A})$, if $M \rightarrow I$ is a K-injective resolution of $M$, then $\Rtorfun{A}(M) \cong \torfun{A}(I)$.

The following quasi-compactness of $\Rtorfun{A}(-)$ is important in this paper.

\begin{prop}[Quasi-compactness of $\torfun{A}(-)$, {\cite[Lemma 5.9]{mori2021local}}]
	\label{prop:quasi-compactness-torfunctor}
	Let $A$ be a right Ext-finite connected $\Zbb$-algebra.
	Then, $\R^i\torfun{A}(-)$ commutes with direct limits for all $i \geq 0$.
\end{prop}

We can also define the right derived functor $\Rtorfun{A}(-)$ for bigraded $B$-$A$-bimodules and we can calculate $\Rtorfun{A}(M)$ by taking a K-resolution $M \rightarrow I$ in $D(\Grlr{B}{A})$.
Note that as in \cref{rmk:derived-functors}, it is enough to take a resolution of $M$ in $D(\Grlr{B}{A})$ such that each $e_iI$ is K-injective in $D(\Grr{A})$ for all $i \in \Zbb$ when we calculate $\Rtorfun{A}(M)$.

The \emph{(right) cohomological dimension $\cd(\torfun{A})$} of $\torfun{A}(-)$ is defined by
\[ 
\supr \{ i \in \Zbb \mid \R^i\torfun{A}(M) \neq 0 \text{ for some } M \in \Grr{A} \}.
\]

\begin{rmk}
	The analogous definitions and facts for $\torfun{A^{\op}}$ hold, and we will use them without further comment.
\end{rmk}

\subsection{Basic properties}

In this subsection, we study basic properties of torsion modules.

\begin{lem}
	\label{lem:basic lemma-1}
	Let $A$ be a connected $\Zbb$-algebra.
	\begin{enumerate}
		\item We have isomorphisms of functors:
		\begin{align*}
		\D{(-)} \cong \Homintr{K}(-, K) &\cong \Homintr{A}(-, \D{A}), \\
		\Homr{K}(-, \Resr{K}(K)) &\cong \Homr{A}(-, \Resr{A}(\D{A})).
		\end{align*}
		\item $\Resr{A}(\D{A})$ and $e_i\D{A}$ are injective objects in $\Grr{A}$ for all $i \in \Zbb$.
	\end{enumerate}
\end{lem}

\begin{proof}
(1)
From \cref{prop:tensor-product-properties}, we obtain an isomorphism of $k$-modules
\begin{align*}
\Homr{K}(e_i M \tenint{A} A, e_jK) \cong \Homr{A}(e_i M, \Homintr{K}(A, e_j K))
\end{align*}
for any $M \in \Grr{A}$.
Moreover,
\begin{align*}
	\Homintr{K}(M, K)_{ji} &\cong \Homr{K}(e_i M \tenint{A} A, e_j K), \\
	\Homintr{A}(M, \D{A})_{ji} &\cong \Homr{A}(e_i M, \Homintr{K}(A, e_j K)).
\end{align*}
Thus, we have the first isomorphism.
In the same way, the second isomorphism follows as well.

(2)
Since the functor $\Homr{K}(-, \Resr{K}(K))$ is exact, so $\Resr{A}(\D{A})$ is an injective object in $\Grr{A}$ from (1).
For $e_i\D{A}$, since it is a direct summand of the injective object $\Resr{A}(\D{A})$, it is also injective.
\end{proof}


\begin{lem}
	\label{lem:basic lemma-2}
	Let $A$ be a right noetherian $\Zbb$-algebra.
	Let $\{I_j\}_{j \in J}$ be a set of injective objects in $\Grr{A}$.
	Then, $\bigoplus_{j \in J}I_j$ is an injective object in $\Grr{A}$.
\end{lem}

\begin{proof}
Since $\Grr{A}$ is a locally noetherian Grothendieck category, injective objects are closed under direct sums. 
For example see, \cite[Proposition 4.3]{stenstrom1975rings} or \cite{popescu1973abelian}.
\end{proof}

\begin{lem}
		\label{lem:basic lemma-3}
		Let $W \in \Grr{K}$.
		Then, for every $i \in \Zbb$, there exists a set $X_i$ such that $W = \bigoplus_{i \in \Zbb} \bigoplus_{x \in X_i} e_i K$.	
\end{lem}

\begin{proof}
	Let $W = \bigoplus_{i \in \Zbb} W_i$.
	Let $\widetilde{K}_l$ be a $\Zbb$-algebra defined by 
	\[
	(\widetilde{K}_l)_{ij} = 
	\begin{cases}
		k & \text{if } i = j = l, \\
		0 & \text{otherwise}
	\end{cases}	
	\]
	and we define $\widetilde{W}_l$ by
	\[
	(\widetilde{W}_l)_i = 
	\begin{cases}
		W_l & \text{if } i = l, \\
		0 & \text{otherwise}.
	\end{cases}
	\]
	Then, $\widetilde{W}_l$ is a $\widetilde{K}_l$-module.
	Any $\widetilde{K}_l$-module concentrated in degree $l$ is of the form $\bigoplus_{x \in X_l} e_l K$ for some set $X_l$.
	This induces the lemma.
\end{proof}

Let $A$ be a connected $\Zbb$-algebra and $M \in \Grr{A}$.
A graded submodule $N \subset M$ is \emph{essential} if for every graded submodule $L \subset M$, $L \cap N$ is non-zero.

\begin{lem}
	\label{lem:basic lemma-4}
	Let $A$ be a connected $\Zbb$-algebra.
	Let $M \in \Grr{A}$ and put $W = \Soc(M)$.
	Then, if $M$ is torsion, then $W \subset M$ is an essential graded submodule.
\end{lem}

\begin{proof}
	The proof is the same as \cite[Lemma 16.3.8]{yekutieli2019derived}.
	But, we give a proof for the convenience of the reader.

	Let $N \subset M$ be a graded submodule.
	Let $n \in N$ be a non-zero homogeneous element.
	Then, there exists a unique $i \in \Zbb$ such that $nA_{\geq i+1} = 0$ and $nA_{\geq i} \neq 0$.
	So, there is a homogeneous element $a \in A_{i}$ such that $na \neq 0$.
	This element $na$ belongs to $W$.
\end{proof}

\begin{lem}
	
	\label{lem:basic lemma-5}
	Let $A$ be a right noetherian connected $\Zbb$-algebra and let $W \in \Grr{K}$.
	Set $I = W \tenint{K} \D{A} \in \Grr{A}$.
	Then, 
	\begin{enumerate}
		\item $I$ is an injective object in $\Grr{A}$.
		\item $I$ is torsion.
		\item $\Soc(I) = W$.
		\item $W \subset I$ is an essential graded submodule.
		
	\end{enumerate}
	
\end{lem}

\begin{proof}
	(1) From \cref{lem:basic lemma-3}, we have $W = \bigoplus_{i \in \Zbb} \bigoplus_{x \in X_i} e_i K$ for some sets $X_i \ (i \in \Zbb)$.
	Then, 
	\[
	W \tenint{K} \D{A} = \bigoplus_{i \in \Zbb} \bigoplus_{x \in X_i} e_i K \tenint{K} \D{A} \cong \bigoplus_{i \in \Zbb} \bigoplus_{x \in X_i} e_i \D{A}.
	\]
	So, from (2) of \cref{lem:basic lemma-1} and \cref{lem:basic lemma-2}, $I$ is an injective object in $\Grr{A}$.

(2) Because each $e_i \D{A}$ is torsion and $A$ is right Ext-finite, we obtain
\begin{align*}
\torfun{A}(I) &\cong \torfun{A}\left(\bigoplus_{i \in \Zbb} \bigoplus_{x \in X_i} e_i \D{A}\right) 
\cong \bigoplus_{i \in \Zbb} \bigoplus_{x \in X_i} \torfun{A}(e_i \D{A}) 
\cong \bigoplus_{i \in \Zbb} \bigoplus_{x \in X_i} e_i \D{A} \cong I.
\end{align*}

(3) As in the proof of (1), we have $W= \bigoplus_{i \in \Zbb} \bigoplus_{x \in X_i} e_i K$ and $I = \bigoplus_{i \in \Zbb} \bigoplus_{x \in X_i} e_i \D{A}$ for some sets $X_i \ (i \in \Zbb)$.
Then, for each $j \in \Zbb$, we have
\begin{align*}
\Soc(I)_{j} 
&\cong \Homintr{A}\left(K, \bigoplus_{i \in \Zbb} \bigoplus_{x \in X_i} e_i \D{A}\right)_j 
\cong  \Homr{A}\left(e_jK, \bigoplus_{i \in \Zbb} \bigoplus_{x \in X_i} e_i \D{A}\right)\\
&\cong \bigoplus_{i \in \Zbb} \bigoplus_{x \in X_i} \Homr{A}(e_jK, e_i \D{A}) \\
&\cong \bigoplus_{i \in \Zbb} \bigoplus_{x \in X_i} \Soc(e_i \D{A})_j,
\end{align*}
where the penultimate isomorphism follows from the fact that $e_jK$ is a finitely generated graded right $A$-module and so $\Homr{A}(e_jK, -)$ commutes with direct sums.

By the definition of $\D{(-)}$ and \cref{prop:tensor-product-properties}, we compute
\begin{align*}
	\Soc(e_i\D{A})_j &= \Homintr{A}(K, e_i\D{A})_j \cong  \Homr{A}(e_jK, e_i\D{A}) \\
	&\cong \Homr{A}(e_jK, \Homintr{K}(A, e_iK)) \\
	&\cong \Homr{K}(e_jK \tenint{A} A, e_iK) \\
	&\cong \Homr{K}(e_jK, e_iK).
\end{align*}
Thus, $\Soc(e_i\D{A})$ is isomorphic to $e_iK$.
This shows (3).

(4) The claim follows from (2), (3) and \cref{lem:basic lemma-4}.
\end{proof}

\begin{rmk}
	\label{rmk:basic lemma-5}
	(2), (4) need only the assumption that $A$ is right Ext-finite.
	(3) needs only the assumption that $A$ is connected.
\end{rmk}

\begin{prop}
	\label{lem:basic Proposition-6}
	Let $A$ be a right noetherian connected $\Zbb$-algebra and $M \in \Grr{A}$.
	Set $W =\Soc(M)$.
	Then, the following conditions are equivalent:
	\begin{enumerate}
		\item $M$ is torsion.
		\item $W \subset M$ is an essential graded submodule.
		\item There exists an injective morphism 
		\begin{align*}
		f : M \hookrightarrow W \tenint{K} \D{A}
		\end{align*}
		which gives an essential graded submodule of $W \tenint{K} \D{A}$ and $f|_{W} = \Id_W$.
	\end{enumerate} 
\end{prop}

\begin{proof}
	(1) $\Rightarrow$ (2):
	This follows from \cref{lem:basic lemma-4}.

	(2) $\Rightarrow$ (3): 
	Denote by $i: W \hookrightarrow M$ the inclusion morphism.
	By \cref{lem:basic lemma-5}, $W \tenint{K} \D{A}$ is an injective object in $\Grr{A}$.
	Tensoring the map $K \hookrightarrow \D{A}$ with $W \tenint{K} -$ yields an injective morphism $j: W \hookrightarrow W \tenint{K} \D{A}$.
	So, we obtain a morphism $f : M \rightarrow W \tenint{K} \D{A}$ such that $j = f \circ i$.
	From the essentiality of $W \subset M$, we can show that $f$ is injective. 

	(3) $\Rightarrow$ (1): This follows by combining the assumption that $f$ is injective with the facts that each $e_i\D{A}$ is torsion and $\Tor(A)$ is a Serre subcategory of $\Grr{A}$ (\cite[Lemma 3.5]{mori2021local}), which is closed under direct sums.
\end{proof}

\begin{rmk}
\label{rmk:basic lemma-6}
	When we prove $(2) \Rightarrow (3)$, we use the assumption that $A$ is right noetherian.
	When we prove $(3) \Rightarrow (1)$, we only use the assumption that $A$ is right Ext-finite.
\end{rmk}

Let $A$ be a connected $\Zbb$-algebra and $M \in \Grr{A}$.
Then, the \emph{injective hull} $E(M)$ of $M$ is a graded right $A$-module containing $M$ which is an injective object in $\Grr{A}$ and $M$ is an essential graded submodule of $E(M)$.
Every graded right $A$-module has an injective hull since $\Grr{A}$ is a Grothendieck category (see \cite[Section V.2 Examples]{stenstrom1975rings}, \cite[Theorem 10.10]{popescu1973abelian}).

\begin{lem}
	\label{lem:basic lemma-7}
	Let $A$ be a right noetherian connected $\Zbb$-algebra and let $M$ be a torsion graded right $A$-module.
	Then, $E(M) \cong W \tenint{K} \D{A}$, where $W = \Soc(M)$.
\end{lem}

\begin{proof}
	We can obtain the lemma by combining \cref{lem:basic lemma-5}, \cref{lem:basic Proposition-6} and the uniqueness of injective hulls.
\end{proof}

\begin{lem}
	\label{lem:basic lemma-8}
	Let $A$ be a right noetherian connected $\Zbb$-algebra and let $I$ be an injective object in $\Grr{A}$.
	Then, $\torfun{A}(I) \cong W \tenint{K} \D{A}$, where $W = \Soc(I)$.
\end{lem}

\begin{proof}

	Considering the following commutative diagram:
	\[
	\begin{tikzcd}
		W \ar[r, hook, "i"] \ar[d, hook', "j"'] & W \tenint{K} \D{A} \ar[ld,dotted, "\varphi"] \\
		I &
	\end{tikzcd}
	\]
	, where $i,j$ are the inclusion morphisms and $\varphi$ is obtained by the definition of injective hulls.
	Then, $\varphi$ is injective since $W$ is an essential graded submodule of $W \tenint{K} \D{A}$ from (4) of \cref{lem:basic lemma-5}.
	We decompose $I$ as $I = \Img(\varphi) \oplus J$ by using the injectivity of $W \tenint{K} \D{A}$, which is from (1) of \cref{lem:basic lemma-5}.
	Then, we have
	\[
	\Soc(\torfun{A}(J)) = \Soc(J) = W \cap J \subset \torfun{A}(J),
	\]
	which is an essential injection from $(1) \Rightarrow (2)$ of \cref{lem:basic Proposition-6}.
	By combining this with the fact that $\Img(\varphi)$ is torsion from (2) of \cref{lem:basic lemma-5}, we have 
	\[
	\torfun{A}(I) \cong \Img(\varphi) \cong W \tenint{K} \D{A}.
	\]
\end{proof}

\begin{cor}[Stableness of $\torfun{A}(-)$]
	\label{cor:basic lemma-9}
	Let $A$ be a right noetherian connected $\Zbb$-algebra and let $I$ be an injective object in $\Grr{A}$.
	Then, $\torfun{A}(I)$ is an injective object in $\Grr{A}$. 
\end{cor}

\begin{proof}
	This is from (1) of \cref{lem:basic lemma-5} and \cref{lem:basic lemma-8}.
\end{proof}


\begin{prop}
	\label{prop:basic proposition-10}
Let $A$ be a right noetherian connected $\Zbb$-algebra.
Let $B,C$ be connected $\Zbb$-algebras.
Let $M \in D^{\star}(\Grlr{B}{A})$ and $N \in D^+(\Grlr{C}{A})$,
where $\star = \emptyset$ if $\cd(\torfun{A})$ is finite, and $\star = +$ otherwise.
Assume that $\Resr{A}(H^i(M))$ is torsion for all $i \in \Zbb$.
Then, the following hold:
\begin{enumerate}
  \item There is a natural morphism $ \Rtorfun{A}(M) \rightarrow M$ in $D(\Grlr{B}{A})$, and this morphism is an isomorphism.
  \item $\RHomintr{A}(M,N) \simeq \RHomint_A(M, \Rtorfun{A}(N))$ in $D(\Grlr{C}{B})$.
\end{enumerate}
\end{prop}

\begin{proof}

(1)
We consider the unbounded case.
First of all, note that the natural morphism is obtained from the universal property of the derived functor $\Rtorfun{A}$.

We take a Cartan-Eilenberg resolution $M \rightarrow I^{\mydot \mydot}$, which induces a quasi-isomorphism $M \rightarrow \operatorname{Tot}(I^{\mydot \mydot})$ in $C(\Grlr{B}{A})$ (\cite[Theorem A.3]{weibel1996cyclic}).
We also truncate $I^{\mydot \mydot}$ in the second degree at $\cd(\torfun{A})+1$. 
Let $\tilde{I}^{\mydot \mydot}$ denote the truncated complex. 
Then, $\tilde{I}^{\mydot\mydot}$ is a double complex with finitely many nonzero components in each total degree, and each $e_i \tilde{I}^{p,q}$ is $\torfun{A}$-acyclic.
Thus, we have the following spectral sequence (\cite[page 150, 5.7.9]{weibel1994introduction}, \cite[Lemma 12.25.3]{stacks-project}):
\begin{align*}
 E_2^{p,q} = \R^p\torfun{A}(H^q(M))
 &\Rightarrow 
 E^{p+q}=H^{p+q}(\operatorname{Tot}(\torfun{A}(\tilde{I}^{\mydot \mydot}))) 
 \cong H^{p+q}(\torfun{A}(\operatorname{Tot}(\tilde{I}^{\mydot \mydot}))).
\end{align*}
From \cite[Lemma 12.25.4]{stacks-project}, $M \rightarrow \operatorname{Tot}(\tilde{I}^{\mydot \mydot})$ is a quasi-isomorphism and it follows that the natural morphism $\torfun{A}(\operatorname{Tot}(\tilde{I}^{\mydot \mydot})) \rightarrow \operatorname{Tot}(\tilde{I}^{\mydot \mydot})$ corresponds to the natural morphism $\Rtorfun{A}(M) \rightarrow M$.
Since $\cd(\torfun{A})$ is finite and each $e_i\operatorname{Tot}(\tilde{I}^{\mydot \mydot})^n = \bigoplus_{k+l=n} e_i\tilde{I}^{k,l}$ is a $\torfun{A}$-acyclic object, we obtain
\[
H^{p+q}(\torfun{A}(\operatorname{Tot}(\tilde{I}^{\mydot \mydot}))) \cong \R^{p+q}\torfun{A}(M).
\]
Hence, it is enough to show that $\R^{i}\torfun{A}(M)=0$ for all $i >0$ and $M \in \Grlr{B}{A}$ whose restriction to $A$ is torsion.
This follows from \cite[Lemma 5.10]{mori2021local}.

When $M$ is bounded below, it is sufficient to use the Cartan-Eilenberg resolution to obtain the spectral sequence.
We do not need to take any truncation.

(2) 
We only prove the unbounded case.
The bounded below case is proved in the same way.
Take a K-injective resolution of $M \rightarrow I_M$ in $D(\Grlr{B}{A})$. 
Here, we can assume that each $I_M^j$ is an injective object in $\Grlr{B}{A}$ by \cref{rmk:existence-K-projective-K-injective}.
Then, each $e_i I_M^j$ is an injective object in $\Grr{A}$ by \cref{prop:restriction-K-projective-K-injective}.
In the same way, we can take a K-injective resolution of $N \rightarrow I_N$ in $D^+(\Grlr{C}{A})$ with $e_i I_N^j$ being an injective object in $\Grr{A}$.
Then, we have
\begin{align*}
\RHomint_{A}(M, N) &\cong \Homintr{A}(I_M, I_N), \\
\Rtorfun{A}(N) &\cong \torfun{A}(I_N).
\end{align*}
Moreover, since $I_N$ is bounded below and each $e_i I_N^j$ is injective in $\Grr{A}$, $e_i \torfun{A}(I_N) \cong \torfun{A}(e_i I_N)$ is also K-injective in $D(\Grr{A})$ by \cref{cor:basic lemma-9}.
Thus, we have
\begin{align*}
	\RHomint_{A}(M, \Rtorfun{A}(N)) &\cong \Homintr{A}(I_M, \torfun{A}(I_N))
\end{align*}
from \cref{rmk:derived-functors}.
We also have the following quasi-isomorphisms:
\begin{align*}
	\Homintr{A}(I_M, I_N) &\cong \Homintr{A}(\torfun{A}(I_M), I_N), \\
	\Homintr{A}(I_M, \torfun{A}(I_N)) &\cong \Homintr{A}(\torfun{A}(I_M), \torfun{A}(I_N))
\end{align*}
from (1) and an isomorphism
\begin{align*}
	\Homintr{A}(\torfun{A}(I_M), I_N) &\cong \Homintr{A}(\torfun{A}(I_M), \torfun{A}(I_N))
\end{align*}
because any homomorphism from $\torfun{A}(I_M)$ to $I_N$ factors through $\torfun{A}(I_N)$.
The claim follows from the above isomorphisms.
\end{proof}

\begin{rmk}
	\label{rmk:basic proposition-10}
	In the proof of (1) of \cref{prop:basic proposition-10}, we do not need the assumption that 
	$A$ is right noetherian (it is enough to assume that $A$ is right Ext-finite).
\end{rmk}

\subsection{Local duality}
\label{subsec:local duality}

\begin{lem}
	\label{lem:local duality-1}
	Let $A$ be a right Ext-finite connected $\Zbb$-algebra.
	Let $B$ be  a connected $\Zbb$-algebra.
	We assume that $\cd(\torfun{A})$ is finite.
	Then, for $M \in D(\Grlr{A}{A})$ and $N \in D(\Grlr{B}{A})$, there is an isomorphism in $D(\Grlr{B}{A})$:
	\begin{align*}
		 N \Ltenint{A} \Rtorfun{A}(M) \cong \Rtorfun{A}(N \Ltenint{A} M).
	\end{align*}
	This isomorphism is functorial in $M,N$.
\end{lem}

\begin{proof}
	First of all, note that we can compute $\Rtorfun{A}$ over $D(\Grlr{B}{A})$ by using complexes of $\torfun{A}$-acyclic bigraded $B$-$A$-bimodules since $\cd(\torfun{A})$ is finite (\cite[Corollary 10. 5.11]{weibel1994introduction}, \cite[Corollary 5.3 ($\gamma$)]{hartshorne1966residues}).

	Let $M \rightarrow I$ be a K-injective resolution of $M$ in $D(\Grlr{A}{A})$ with each $I^j$ being an injective object in $\Grlr{A}{A}$ and $P \rightarrow N$ be a K-projective resolution of $N$ in $D(\Grlr{B}{A})$ with each $P^j$ being a projective object in $\Grlr{B}{A}$ from \cref{rmk:existence-K-projective-K-injective}.
	Then, we have
	\[
	N \Ltenint{A} M \cong P \tenint{A} I
	\]
	from \cref{prop:derived-functors}. 
	For any $P^j$, there exists a bigraded $B$-$A$-bimodule $Q^j$ such that $P^j \oplus Q^j$ is isomorphic to $\bigoplus_{(n_1,n_2) \in J}(Be_{n_2} \otimes_k e_{n_1}A)$ for some set $J$ of integers which may contain duplicates.
	Then, 
	\[
	(P^j \oplus Q^j) \tenint{A} I^i \cong \bigoplus_{(n_1,n_2) \in J}(Be_{n_2} \otimes_k e_{n_1}A) \tenint{A} I^i \cong \bigoplus_{(n_1,n_2) \in J}(Be_{n_2} \otimes_k e_{n_1}I^i).
	\]
	So, since $e_kI^i$ is injective in $\Grr{A}$ (\cref{lem:restriction}) and $\R^m\torfun{A}$ commutes with direct sums (\cref{prop:quasi-compactness-torfunctor}), we have
	\begin{align*}
	\R^m\torfun{A}((P^j \oplus Q^j) \tenint{A} I^i) 
	&\cong \R^m\torfun{A} \left( \bigoplus_{(n_1,n_2) \in J}(Be_{n_2} \otimes_k e_{n_1}A) \tenint{A} I^i \right) \\
	&\cong \R^m\torfun{A} \left( \bigoplus_{(n_1,n_2) \in J}(Be_{n_2} \otimes_k e_{n_1}I^i) \right) \\
	&\cong \bigoplus_{(n_1,n_2) \in J} \R^m\torfun{A}(Be_{n_2} \otimes_k e_{n_1}I^i) \cong 0.
	\end{align*}
	Note that in the last equality in the above calculation, we use the isomorphism
	\[
	\Resr{A}(\R^m\torfun{A}(Be_{n_2} \otimes_k e_{n_1}I^i)) \cong \R^m\torfun{A}(\Resr{A}(Be_{n_2} \otimes_k e_{n_1}I^i))
	\]
	and the fact $\Resr{A}(Be_{n_2} \otimes_k e_{n_1}I^i)$ is a direct sum of $e_{n_1}I^i$.
	Thus, $P^j \tenint{A} I^i$ is a $\torfun{A}$-acyclic bigraded $B$-$A$-bimodule for all $i,j \in \Zbb$.
	Because a direct sum of $\torfun{A}$-acyclic bigraded $B$-$A$-bimodules is $\torfun{A}$-acyclic, we have
	\[
	\Rtorfun{A}(N \Ltenint{A} M) \cong \torfun{A}(P \tenint{A} I).
	\]
	By a similar argument, we can obtain
	\[
		P \tenint{A} \torfun{A}(I) \cong \torfun{A}(P \tenint{A} I).
	\]
	Indeed, we have the following isomorphisms
	\begin{align*}
		\torfun{A}((P^j \oplus Q^j) \tenint{A} I^i) 
		&\cong\torfun{A} \left( \bigoplus_{(n_1,n_2) \in J}(Be_{n_2} \otimes_k e_{n_1}A) \tenint{A} I^i \right) \\
		&\cong \torfun{A} \left( \bigoplus_{(n_1,n_2) \in J} Be_{n_2} \otimes_k e_{n_1}I^i \right) \\
		&\cong \bigoplus_{(n_1,n_2) \in J} Be_{n_2} \otimes_k e_{n_1}\torfun{A}(I^i) \\
		&\cong \left(\bigoplus_{(n_1,n_2) \in J} Be_{n_2} \otimes_k e_{n_1}A \right) \tenint{A}  \torfun{A}(I^i) \\
		&\cong (P^j \oplus Q^j) \tenint{A} \torfun{A}(I^i).
	\end{align*}
	The above isomorphism is induced by the natural morphism
	\[
		(P^j \oplus Q^j) \tenint{A} \torfun{A}(I^i) \rightarrow \torfun{A}((P^j \oplus Q^j) \tenint{A} I^i),
	\]
	which yields the isomorphism
	\[
	P \tenint{A} \torfun{A}(I) \xrightarrow{\cong} \torfun{A}(P \tenint{A} I)
	\]
	and we obtain the isomorphism
	\[
	\eta: N \Ltenint{A} \Rtorfun{A}(M) \xrightarrow{\cong} \Rtorfun{A}(N \Ltenint{A} M) 
	\]

	We prove the functoriality in $M$.
	Let $\varphi : M \rightarrow M'$ be a morphism in $D(\Grlr{A}{A})$ and $f_1 : M \rightarrow I, f_2 : M' \rightarrow I'$ be K-injective resolutions of $M,M'$ in $D(\Grlr{A}{A})$ respectively.
	Then, there exists a morphism $\tilde{\varphi}: I \rightarrow I'$ in $D(\Grlr{A}{A})$ making the following diagram commute:
	\[
	\begin{tikzcd}
		M \ar[r, "f_1"] \ar[d, "\varphi"'] & I \ar[d, "\tilde{\varphi}"] \\
		M' \ar[r, "f_2"] & I'
	\end{tikzcd}.
	\]
	By \cite[Proposition 1.5]{spaltenstein1988resolutions}, this morphism $\tilde{\varphi}$ is represented by a morphism of complexes $\psi: I \rightarrow I'$.
	Regarding this morphism $\psi$, the isomorphism $\eta$ is functorial, i.e. we have the following commutative diagram of complexes:
	\[
	\begin{tikzcd}
		P \tenint{A} \torfun{A}(I) \ar[r, "\cong"] \ar[d] & \torfun{A}(P \tenint{A} I) \ar[d]\\
		P \tenint{A} \torfun{A}(I') \ar[r, "\cong"] & \torfun{A}(P \tenint{A} I')
	\end{tikzcd}.
	\]
	This induces the desired functoriality.
	The functoriality in $N$ is proved similarly, using \cite[Proposition 1.4]{spaltenstein1988resolutions} instead of \cite[Proposition 1.5]{spaltenstein1988resolutions}.
\end{proof}


The following lemma is the derived Tensor-Hom adjunction.
We omit the proof because it is similar to the proof of \cref{prop:tensor-product-properties}.

\begin{lem}
	\label{lem:local duality-2}
	Let $A,B,C$ and $D$ be connected $\Zbb$-algebras.
	For $M \in D(\Grlr{C}{B})$, $N \in D(\Grlr{B}{A})$ and $L \in D(\Grlr{D}{A})$, the following isomorphism holds in $D(\Grlr{D}{C})$:
	\begin{align*}
		\RHomint_{B}(M, \RHomint_{A}(N, L)) \cong \RHomint_A(M \Ltenint{B} N, L).
	\end{align*}
	This isomorphism is functorial in $L,M,N$.
	
\end{lem}


The following proposition is a generalization of \cite[Theorem 2.1]{mori2024corrigendum}, where the authors assumed that $M$ is bounded below and $B \cong K$.
We establish the local duality in full generality.

\begin{thm}[Local duality]
\label{thm:local duality}
Let $A$ be a right Ext-finite connected $\Zbb$-algebra.
Let $B$ be a connected $\Zbb$-algebra.
We assume that $\cd(\torfun{A})$ is finite.
Then, for $M \in D(\Grlr{B}{A})$, we have
\begin{align*}
	\D{\Rtorfun{A}(M)} \cong \RHomint_A(M, \D{\Rtorfun{A}(A)})
\end{align*}
in $D(\Grlr{A}{B})$.
This isomorphism is functorial in $M$.

\end{thm}

\begin{proof}
	By using \cref{lem:local duality-1}, \cref{lem:local duality-2} and the isomorphism (\cref{lem:basic lemma-1})
	\[
	\Homint_K(-,K) \cong \RHomint_A(-, \D{A}),
	\]
	we have the following functorial isomorphisms in $D(\Grlr{A}{B})$:
\begin{align*}
	\D{\Rtorfun{A}(M)} &\cong \Homint_K(\Rtorfun{A}(M), K) \\
	&\cong \RHomint_A(\Rtorfun{A}(M), \D{A}) \quad (\text{\cref{lem:basic lemma-1}}) \\
	&\cong \RHomint_A(\Rtorfun{A}(M \Ltenint{A} A), \D{A}) \\
	&\cong \RHomint_A(M \Ltenint{A} \Rtorfun{A}(A), \D{A}) \quad (\text{\cref{lem:local duality-1}}) \\
	&\cong \RHomint_A(M, \RHomint_A(\Rtorfun{A}(A), \D{A})) \quad (\text{\cref{lem:local duality-2}})\\
	&\cong \RHomint_A(M, \D{\Rtorfun{A}(A)}) \quad (\text{\cref{lem:basic lemma-1}}).
\end{align*}
\end{proof}

\begin{rmk}
	\label{rmk:local duality}
	Let $\varphi: \D{\Rtorfun{A}(A)} \rightarrow I$ be a K-injective resolution of $\D{\Rtorfun{A}(A)}$ in $D(\Grlr{A}{A})$, and let $\sigma_{A}: \Rtorfun{A}(A) \rightarrow A$ be the natural morphism. 
	Then, under the isomorphism in \cref{thm:local duality}, for any $i \in \Zbb$, the morphism $e_i(\varphi \circ \D{\sigma_A}): e_i \D{A} \rightarrow e_i I$ corresponds to the element $\tau_i \in e_i \DD{A} \cong e_i\D{\Rtorfun{A}(\D{A})}$ viewed as the map $\tau_i:e_i\D{A}\rightarrow e_iK$ defined by $\tau_i(f) = f(e_i)$ for any $f \in e_i\D{A}$.
	This follows by chasing the isomorphisms in the proof.
\end{rmk}

We can also show the generalization of \cite[Theorem 3.21]{mori2025categorical} in the unbounded case.
\begin{cor}
	\label{cor:local-duality-2}
	Let $A$ be a right Ext-finite connected $\Zbb$-algebra.
	We assume that $\cd(\torfun{A})$ is finite.
	Then, for $M \in D(\Grr{A})$, we have 
	\[
	\D{\Rtorfun{A}(M)} \cong \RHomint_A(M, \D{\Rtorfun{A}(A)})
	\]
	in $D(\Grl{A})$.
	This isomorphism is functorial in $M$.
\end{cor}

\begin{proof}
The proof is exactly the same as \cite[Theorem 3.21]{mori2025categorical}.
However, we give a proof for the convenience of the reader.
In fact, the corollary is given by the following functorial isomorphisms:
\begin{align*}
	\D{\Rtorfun{A}(M)} &\cong \D{\Rtorfun{A}(e_0(Ke_0 \ten{k} M))} \\
	&\cong \D{(e_0 \Rtorfun{A}(Ke_0 \ten{k} M))} \\
	&\cong (\D{\Rtorfun{A}(Ke_0 \ten{k} M)})e_0 \\
	&\cong \RHomint_A(Ke_0 \ten{k} M, \D{\Rtorfun{A}(A)})e_0 \quad (\text{\cref{thm:local duality}}) \\
	&\cong \RHomint_A(e_0(Ke_0 \ten{k} M), \D{\Rtorfun{A}(A)}) \\
	&\cong \RHomint_A(M, \D{\Rtorfun{A}(A)}).
\end{align*}
\end{proof}

\subsection{Symmetricity of derived torsion functors}

\begin{dfn}
\label{dfn:symmetric derived torsion functor}
	Let $A$ be a noetherian connected $\Zbb$-algebra.
	Let $M$ be an object in $D(\Grlr{A}{A})$.
	\begin{enumerate}
	\item $M$ has \emph{weak symmetric derived torsion} if for every $i \in \Zbb$, $\Resl{A}(H^i(\Rtorfun{A}(M)))$ is $\mathfrak{m}_{A^{\op}}$-torsion and $\Resr{A}(H^i(\Rtorfun{A^{\op}}(M)))$ is $\mathfrak{m}_{A}$-torsion.
	 \item $M$ has \emph{symmetric derived torsion} if there exists an isomorphism
	 \[
	 \epsilon_M : \Rtorfun{A}(M) \; \xrightarrow{\cong} \; \Rtorfun{A^{\op}}(M)\quad \text{in } D(\Grlr{A}{A})
	 \]
	 such that the following diagram
	 \[
	 \begin{tikzcd}
		\Rtorfun{A}(M) \ar[rr, "\epsilon_M"] \ar[rd, "\sigma_M"'] & & \Rtorfun{A^{\op}}(M) \ar[ld, "\sigma^{\op}_{M}"]  \\
		& M &
	\end{tikzcd}
	\]
	is commutative, where $\sigma_M, \sigma^{\op}_{M}$ are the natural morphisms.
\end{enumerate}
\end{dfn}

\begin{prop}
	\label{prop:symmetric derived torsion functor}
	Let $A$ be a noetherian connected $\Zbb$-algebra.
	Let $M$ be an object in $D(\Grlr{A}{A})$.
	We assume that $\cd(\torfun{A})$ and $\cd(\torfun{A^{\op}})$ are finite.
	Then, the following conditions are equivalent:
	\begin{enumerate}
	\item $M$ has weak symmetric derived torsion.
	\item $M$ has symmetric derived torsion.
	\end{enumerate}
\end{prop}

\begin{proof}

	If $M$ has symmetric derived torsion, then 
	\begin{align*}
		\Resl{A}(H^i(\Rtorfun{A}(M))) &\cong \Resl{A}(H^i(\Rtorfun{A^{\op}}(M))), \\
		\Resr{A}(H^i(\Rtorfun{A^{\op}}(M))) &\cong \Resr{A}(H^i(\Rtorfun{A}(M))).
	\end{align*}
	Thus, $M$ has weak symmetric derived torsion.

	Next, assume that $M$ has weak symmetric derived torsion.
	We have the following commutative diagram in $D(\Grlr{A}{A})$:
	\[
	\begin{tikzcd}
		\Rtorfun{A^{\op}}(A) \Ltenint{A} (M \Ltenint{A} \Rtorfun{A}(A)) \ar[d, "f_1"'] \ar[rr, "f_0"] 
		& 
		& (\Rtorfun{A^{\op}}(A) \Ltenint{A} M) \Ltenint{A} \Rtorfun{A}(A) \ar[d, "f_2"]  \\
		A \Ltenint{A} (M \Ltenint{A} \Rtorfun{A}(A)) \ar[d, "f_3"'] 
		& 
		&(\Rtorfun{A^{\op}}(A) \Ltenint{A} M) \Ltenint{A} A \ar[d, "f_4"]\\
		M \Ltenint{A} \Rtorfun{A}(A) \ar[rd, "f_5"']
		& 
		& \Rtorfun{A^{\op}}(A) \Ltenint{A} M \ar[ld, "f_6"] \\
		& M &
	\end{tikzcd}
	\]
	, where 
	\begin{itemize}
		\item $f_0$ comes from the associativity of derived tensor products.
		\item $f_1, f_2$ come from the natural morphisms:
		\begin{align*}
			\sigma^{\op}_{A} : \Rtorfun{A^{\op}}(A) \rightarrow A, \quad
			\sigma_A : \Rtorfun{A}(A) \rightarrow A.
		\end{align*}
		\item $f_3,f_4$ come from the fact that $A$ is the unit of derived tensor products.
		\item $f_5, f_6$ are obtained by combining the above natural morphisms $\sigma_A,\sigma^{\op}_{A}$ and the fact that $A$ is the unit of the derived tensor product.
	\end{itemize}

	From \cref{lem:local duality-1}, it follows that
	\begin{align*}
	N \Ltenint{A} \Rtorfun{A}(A) \cong \Rtorfun{A}(N), \quad
	\Rtorfun{A^{\op}}(A) \Ltenint{A} N &\cong \Rtorfun{A^{\op}}(N)
	\end{align*}
	for any $N \in D(\Grlr{A}{A})$.

	By applying \cref{prop:basic proposition-10} and the isomorphisms to the above diagram, we have the following diagram in $D(\Grlr{A}{A})$:
	\[
	\begin{tikzcd}
		\Rtorfun{A^{\op}}(\Rtorfun{A}(M)) \ar[rr, "g_0"', "\cong"] \ar[d, "g_{1,3}"', "\cong"]
		&
		& \Rtorfun{A}(\Rtorfun{A^{\op}}(M)) \ar[d, "g_{2,4}", "\cong"']  \\
		\Rtorfun{A}(M) 
		& 
		& \Rtorfun{A^{\op}}(M). 
	\end{tikzcd}
	\]
	Therefore, we have the desired isomorphism
	\[
	g_{2,4} \circ g_0 \circ g_{1,3}^{-1} : \Rtorfun{A}(M) \xlongrightarrow{\cong} \Rtorfun{A^{\op}}(M).
	\]
	By construction, this isomorphism is compatible with the natural morphisms to \(M\).
\end{proof}

\subsection{The $\chi$-condition}

\begin{dfn}
	\label{dfn:chi-condition}
	Let $A$ be a connected $\Zbb$-algebra.

	\begin{enumerate}
	\item Assume that $A$ is right noetherian.
	$A$ satisfies the \emph{right $\chi$-condition} if for every $M \in \gr(A)$ and for every integer $i$, $\Extintr{A}^i(K,M)$ is a finite $k$-module.
	\item Assume that $A$ is left noetherian.
	If $A^{\op}$ satisfies the right $\chi$-condition, we say that $A$ satisfies the \emph{left $\chi$-condition}.
	\item $A$ satisfies the left and right $\chi$-conditions, then we say that $A$ satisfies the \emph{$\chi$-condition}.
	\end{enumerate}
\end{dfn}

Let $A$ be a connected $\Zbb$-algebra.
Then, a \emph{minimal complex of injective $A$-modules} is a bounded below complex $(I, d_I)$ of injective right $A$-modules such that the submodule $\Ker(d^j) \subset I^j$ is a graded essential submodule for every $j \in \Zbb$.
Let $M \in D^+(\Grr{A})$.
Then, a \emph{minimal injective resolution} of $M$ is a minimal complex of injective $A$-modules $(I, d_I)$ which is quasi-isomorphic to $M$. 
Since every graded right $A$-module has an injective hull, we can construct a minimal injective resolution of $M$ in the same way as \cite[Proposition 13.26]{yekutieli2019derived}.



\begin{lem}
\label{lem:minimal-injective-resolution-2}
	Let $A$ be a connected $\Zbb$-algebra and $M \in D^+(\Grr{A})$.
	Let $(I, d_I)$ be a minimal injective resolution of $M$. 
	We consider the subcomplex $(\Soc(I), d_{\Soc(I)})$ of $(I, d_I)$.
	Then, $d_{\Soc(I)} =0$.
\end{lem}

\begin{proof}
	We can prove in the same way as \cite[Lemma 16.5.12]{yekutieli2019derived}.
	But, we give a proof for the convenience of the reader.

	Take a homogeneous degree $i$ element $0 \neq x \in (\Soc(I)^j)_i$.
	Then, we have
	\[
	x K  = x A \subset (\Soc(I)^j)_i \subset I^j,
	\]
	that is, $x K$ is a graded submodule of $I^j$.
	Since $x K \cap \Ker(d_I^j) \neq 0$ (essentiality of $\Ker(d_I^j)$), we obtain $x \in \Ker(d_I^j)$.
\end{proof}

\begin{prop}
	\label{prop:chi-condition}
	Let $A$ be a right noetherian connected $\Zbb$-algebra.
	Let $M \in D^+(\Grr{A})$ and $i$ be an integer.
	Then,
	\begin{enumerate}
		\item If $\Extint^i_{A}(K,M)$ is a finite $k$-module, then $\R^i\torfun{A}(M)$ is a cofinite $A$-module.
		\item If $\R^j\torfun{A}(M)$ is a cofinite $A$-module for every $j \leq i$, then $\Extintr{A}^j(K, M)$ is a finite $k$-module for every $j \leq i$.
	\end{enumerate}
\end{prop}

\begin{proof}
	(1) Take a minimal injective resolution $M \rightarrow I$ of $M$ and put $W = \Soc(I)$.
	Then, from \cref{lem:minimal-injective-resolution-2}, we have
	\begin{align*}
		\Extintr{A}^i(K,M) \cong H^i(\Homint_A(K, I)) \cong H^i(W) \cong W^i.
	\end{align*}
	Thus, $W^i$ is a finite $k$-module and $W^i \cong \bigoplus_{p \in \Zbb} \bigoplus_{x \in X^i_p} e_p K$ for some finite sets $X^i_p \ (p \in \Zbb)$. 
	In addition, $X^i_p = \emptyset$ for all but finitely many $p$.

	On the other hand, 
	\[
	\R^i\torfun{A}(M) \cong H^i(\torfun{A}(I)).
	\]
	From \cref{lem:basic lemma-8}, we have 
	\[
	\torfun{A}(I^i) \cong W^i \tenint{K} \D{A} \cong \bigoplus_{p \in \Zbb} \bigoplus_{x \in X^i_p} e_p \D{A}.
	\]
	Thus, because $e_p \D{A}$ is a cofinite $A$-module for every $p \in \Zbb$, $\torfun{A}(I^i)$ is a cofinite $A$-module and so is $\R^i\torfun{A}(M)$.

	(2) Take a minimal injective resolution $M \rightarrow I$ of $M$.
	Let $i_0$ be the smallest integer such that $H^{i_0}(M) \neq 0$.
	Put $W = \Soc(I)$.
	Then, as in (1), we have
	\begin{gather*}
		W \subset \torfun{A}(I) \subset I, \quad
		\torfun{A}(I^i) \cong W^i \tenint{K} \D{A}, \quad
		\Extintr{A}^i(K,M) \cong W^i.
	\end{gather*}

	We prove that $W^j$ is a finite $k$-module for any $j \leq i$.
	Because $i_0$ is the smallest integer such that $I^{i_0} \neq 0$, it is enough to start from $j=i_0$.

	If $j=i_0$, we have an inclusion 
	\[
	W^{i_0} \subset \R^{i_0}\torfun{A}(M)
	\]
	because of the fact that $I^{i_0-1} =0$ and \cref{lem:minimal-injective-resolution-2}.
	So, $W^{i_0}$ is a cofinite $A$-module.
	We assume that
	\[
	W^{i_0} \cong \bigoplus_{p \in \Zbb} \bigoplus_{x \in X^{i_0}_p} e_p (A/A_{>0}) \cong \bigoplus_{p \in \Zbb} \bigoplus_{x \in X^{i_0}_p} e_p K
	\]
	for some sets $X^{i_0}_p (p \in \Zbb)$ and $\bigcup_{p \in \Zbb} X^{i_0}_p$ is an infinite set.
	Then, we have a descending chain of submodules of $W^{i_0}$ of infinite length.
	On the other hand, by \cref{thm:matlis-duality}, every cofinite $A$-module is an artinian object in $\Grr{A}$.
	This is a contradiction.
	Thus, $W^{i_0}$ is a finite $k$-module.

	Take an integer $j$ such that $i_0 < j \leq i$ and assume that 
	$W^{j-1}$ is a finite $k$-module.
	Then, we have the following diagram from the definition of $\R^{j}\torfun{A}(M)$, the above isomorphism and \cref{lem:minimal-injective-resolution-2}:
	\[
	\begin{tikzcd}
		\torfun{A}(I^{j-1}) \ar[r] \ar[d, equal, "{\rotatebox[origin=c]{-90}{$\sim$}}"] & \Ker(d^{j}_{\torfun{A}(I)}) \ar[r] & \R^{j}\torfun{A}(M) \ar[r] & 0 \quad \text{(exact)}. \\
		W^{j-1} \tenint{K} \D{A} & W^{j} \ar[u, hook]
	\end{tikzcd}
	\]
	Thus, $\Ker(d^{j}_{\torfun{A}(I)})$ is a cofinite $A$-module and so is $W^j$.
	Therefore, $W^j$ is a finite $k$-module as in the case $j=i_0$.
\end{proof}

From the proposition, we can give another characterization of the $\chi$-condition.

\begin{cor}
	\label{cor:chi-condition}
	Let $A$ be a right (resp. left) noetherian connected $\Zbb$-algebra.
	Then, $A$ satisfies the right (resp. left) $\chi$-condition if and only if for every $M \in \grr{A}$ (resp. $\grl{A}$) and every $i$, $\R^i\torfun{A}(M) \in \grcofr{A}$ (resp. $\R^i\torfun{A^{\op}}(M) \in \grcofl{A}$).
\end{cor}

\begin{proof}
	The ``only if'' part follows from \cref{prop:chi-condition} (1).
	The ``if'' part follows from \cref{prop:chi-condition} (2).
\end{proof}

We denote by $D_{(f,f)}^{\star}(\Grlr{A}{B})$ the full subcategory of 
$D^{\star}(\Grlr{A}{B})$ consisting of objects $M$ whose restrictions 
$\Resll{A}{j}H^i(M)$ and $\Resrr{B}{j}H^i(M)$ of the cohomology modules $H^i(M)$ 
are finite over $A$ and $B$, respectively. 
Here, $\star \in \{\emptyset, +, -, b\}$.

\begin{rmk}
\label{rmk:chi-condition_symmetric derived torsion}
Let $A$ be a noetherian connected $\Zbb$-algebra.
Even if $A$ satisfies the $\chi$-condition, it is unclear whether an object 
$M \in D^b_{(f,f)}(\Grlr{A}{A})$ has weak symmetric derived torsion.
This phenomenon differs from the graded case 
(\cite[Proposition~16.5.19]{yekutieli2019derived}, 
\cite[Corollary~4.8]{vandenbergh1997existence}).
Note that, by \cref{prop:symmetric derived torsion functor}, 
$M$ has weak symmetric derived torsion if and only if it has symmetric derived torsion.

In the case of a noetherian connected graded algebra $R$, for any graded $R$-bimodule $M$ that is finitely generated on both sides, if $R$ satisfies the $\chi$-condition in the sense of \cite{vandenbergh1997existence} or \cite{artin1994noncommutative}, then 
$\R^i\torfun{R}(M)$ and $\R^i\torfun{R^{\op}}(M)$ are right limited, that is, 
$\R^i\torfun{R}(M)_j = 0$ and $\R^i\torfun{R^{\op}}(M)_j = 0$ for $j \gg 0$. 
Here, $\R^i\torfun{R}(M)$ and $\R^i\torfun{R^{\op}}(M)$ denote the $i$-th local cohomology modules of $M$ in the theory of graded algebras. 
This implies that $\R^i\torfun{R}(M)$ and $\R^i\torfun{R^{\op}}(M)$ are torsion on both sides.

For $M \in \Grlr{A}{A}$ such that $Me_j \in \grl{A}$ and $e_jM \in \grr{A}$ for each $j \in \Zbb$, \cref{cor:chi-condition} induces that $\R^i\torfun{A}(e_jM)$ and $\R^i\torfun{A^{\op}}(Me_j)$ are right limited for any $i,j \in \Zbb$. 
However, we cannot conclude from the same argument as in the graded case that 
$\R^i\torfun{A}(M)$ and $\R^i\torfun{A^{\op}}(M)$ are torsion on both sides, 
since $M$ is $\Zbb^2$-graded.
\end{rmk}

However, for a noetherian connected $\Zbb$-algebra $A$, then we can overcome the difficulty in \cref{rmk:chi-condition_symmetric derived torsion} if we consider a $r$-periodic bigraded $A$-$A$-bimodule at least.

\begin{prop}
\label{prop:chi-condition_symmetric derived torsion-2}
Let $A$ be a noetherian connected $\Zbb$-algebra.
Let $M \in D^{\star}_{(f,f)}(\Grlr{A}{A})$, where $\star = \emptyset$ if $\cd(\torfun{A})$ is finite, and $\star = +$ otherwise.
If $A$ satisfies the $\chi$-condition and for any $i$, $H^i(M)$ is $r_i$-periodic for some $r_i$, then $M$ has weak symmetric derived torsion.
\end{prop}

\begin{proof}
	When $M$ is unbounded, as in the proof of \cref{prop:basic proposition-10}, we have the spectral sequence and isomorphisms
	\begin{align*}
		E_2^{p,q} = \R^p\torfun{A}(H^q(M))
		\Rightarrow 
		E^{p+q} &= H^{p+q}(\operatorname{Tot}(\torfun{A}(\tilde{I}^{\mydot \mydot}))) \\
		&\cong H^{p+q}(\torfun{A}(\operatorname{Tot}(\tilde{I}^{\mydot \mydot}))) \\
		&\cong \R^{p+q}\torfun{A}(M),
	\end{align*}
	where $\tilde{I}^{\mydot \mydot}$ is the truncated Cartan-Eilenberg resolution.
	So, we can assume that $M \in \Grlr{A}{A}$.
	If $M$ is bounded below, then we can also obtain deduction by considering a Cartan-Eilenberg resolution without truncation.

	We also assume that $M$ is $r$-periodic for some $r > 0$.
	We can always do this because $M$ is $r$-periodic if and only if $M$ is $(-r)$-periodic (cf. \cite[Lemma 2.14]{mori2025categorical}).
	If $\varphi :A \xlongrightarrow{\cong} A(r,r)$ is an isomorphism of $\Zbb$-algebras which gives $M \cong {}_{\varphi}M(r,r)_{\varphi}$, then we have an isomorphism
	\[
	e_{p+lr}M \xlongrightarrow{\cong} e_{p}M(-lr)_{\varphi^{-l}},
	\]
	where $l \in \Zbb$ (\cref{subsec:Z-algebras}).
	Moreover, if $M \rightarrow I$ is a K-injective resolution, then so is $M(-lr)_{\varphi^{-l}} \rightarrow I(-lr)_{\varphi^{-l}}$.
	Thus, we obtain
	\begin{align*}
		\R^i\torfun{A}(e_{p+lr}M) &\cong \R^i\torfun{A}(e_{p} M(-lr)_{\varphi^{-l}}) \\
		&\cong \R^i\torfun{A}(e_{p} M)(-lr)_{\varphi^{-l}}.
	\end{align*}
	If $\R^i\torfun{A}(e_{p} M)_n=0$ if $n > n_{i,p}$, then $\R^i\torfun{A}(e_{p+lr} M)_{n}=0$ if $n > n_{i,p}+lr$.

	Now suppose
	\[
	  m \in e_p \R^i\torfun{A}(M) e_{n_1} \cong \R^i\torfun{A}(e_p M)_{n_1},
	\]
	and assume that $\R^i\torfun{A}(e_q M)_{n_2} = 0$ for $p-r+1 \leq q \leq p$ and $n_2 > n_{i,p,r}$. 
	Then, by the above observation, $\R^i\torfun{A}(e_q M)_{n_2} = 0$ for $p+(l-1)r+1 \leq q \leq p+lr$ and $n_2 > n_{i,p,r}+lr$. 
	Hence, if we choose $l$ such that $n_1 > n_{i,p,r}+lr$, then $\R^i\torfun{A}(e_q M)_{n_1} = 0$ for $p+(l-1)r+1 \leq q \leq p+lr$. 
	Moreover,
	\[
	e_q \R^i\torfun{A}(M) e_{n_1} \cong \R^i\torfun{A}(e_q M)_{n_1} = 0
   \qquad (q \geq p+lr).
   \]
   Thus, $m$ is left $A$-torsion, and $\R^i\torfun{A}(M)$ is both left and right torsion.

	In the same way, you can show that $\R^i\torfun{A^{\op}}(M)$ is left and right torsion.
\end{proof}


\section{Dualizing complexes}
\label{sec:dualizing complexes}

In this section, we define (balanced) dualizing complexes over noetherian connected $\Zbb$-algebras and study their properties.
In particular, we prove the existence theorem of balanced dualizing complexes over a noetherian connected $\Zbb$-algebra $A$ which satisfies the $\chi$-condition, has finite local cohomological dimension and has symmetric derived torsion functor as a bigraded $A$-$A$-bimodule.

\subsection{Definition of (balanced) dualizing complexes}
\label{subsec:definition of dualizing complexes}

\begin{dfn}
	\label{dfn:dualizing complexes}
	Let $A$ be a noetherian connected $\Zbb$-algebra.
	$R_A \in D^b(\Grlr{A}{A})$ is a \emph{dualizing complex} over $A$ if 
	\begin{enumerate}
		\item For every $i,j$, $\Resrr{A}{i}H^j(R_A)$ is finite over $A$ and $\Resll{A}{i}H^j(R_A)$ is finite over $A^{\op}$.
		\item There exists a complex $I \in D^b(\Grlr{A}{A})$ such that $I \cong R_A$ in $D(\Grlr{A}{A})$, $\Resr{A}(I^j)$ and $\Resl{A}(I^j)$ are injective for every $j$ in $\Grr{A}$ and $\Grl{A}$, respectively.
		\item The natural morphisms
		\begin{align*}
			\Psi_{A} : A &\rightarrow \RHomint_{A}(R_A,R_A), \\
			\Psi_{A}^{\op} : A &\rightarrow \RHomint_{A^{\op}}(R_A,R_A)
		\end{align*}
		are isomorphisms in $D(\Grlr{A}{A})$.
	\end{enumerate}
	Moreover, $R_A$ is \emph{balanced} if 
	\[
	\Rtorfun{A}(R_A) \cong \Rtorfun{A^{\op}}(R_A) \cong \D{A}
	\]
	in $D(\Grlr{A}{A})$.
\end{dfn}

\begin{rmk}
Indeed, it coincides with the definition in \cref{subsec:results}; see \cref{lem:finite injective dimension} and \cref{rmk:duality}.
\end{rmk}

\begin{dfn}
\label{dfn:finite injective dimension}
	Let $\Acal$ be an abelian category.
	An object $M$ in $D(\Acal)$ has \emph{finite injective dimension} (resp. \emph{finite projective dimension}) if there exists an integer $i_0$ such that $\Ext_{\Acal}^i(N,M) = 0$ (resp. $\Ext_{\Acal}^i(M,N) = 0$) for every $|i| > i_0$ and $N \in \Acal$.

	Moreover, for an object $M \in \Acal$, we define the \emph{injective dimension} $\injdim_{\Acal} (M)$ (resp. \emph{projective dimension} $\projdim_{\Acal}(M)$) of $M$ as the minimum integer $i_0$ such that $\Ext_{\Acal}^i(N,M) = 0$ (resp. $\Ext_{\Acal}^i(M,N) = 0$) for every $i > i_0$ and $N \in \Acal$.
\end{dfn}

(2) of \cref{dfn:dualizing complexes} is equivalent to the finiteness condition in \cref{lem:finite injective dimension}.
\cite[Proposition 2.4]{yekutieli1992dualizing} and \cite[Proposition 7.6]{hartshorne1966residues} are slightly generalized in terms of boundedness in \cref{lem:finite injective dimension} below.

\begin{lem}
	\label{lem:finite injective dimension}
	Let $A$ be a noetherian $\Zbb$-algebra and $M$ be an object in $D(\Grlr{A}{A})$.
	Then, the following conditions are equivalent:
	\begin{enumerate}
		\item There exists a complex $I \in D^b(\Grlr{A}{A})$ such that $I \cong M$ in $D(\Grlr{A}{A})$, $\Resr{A}(I^j)$ and $\Resl{A}(I^j)$ are injective for every $j$ in $\Grr{A}$ and $\Grl{A}$, respectively.
		\item $\Resr{A}(M)$ and $\Resl{A}(M)$ have finite injective dimension in $\Grr{A}$ and $\Grl{A}$, respectively. 
	\end{enumerate}
\end{lem}

\begin{proof}
	Much of the proof is similar to \cite[Proposition 2.4]{yekutieli1992dualizing} and \cite[Proposition 7.6]{hartshorne1966residues}.
	However, we need some modifications because we treat unbounded complexes and modules over $\Zbb$-algebras.

	(1) $\Rightarrow$ (2):
	Since we can use the isomorphism $\Resr{A}(M) \cong \Resr{A}(I)$ in $D(\Grr{A})$ to calculate $\Extr{A}^i(N, \Resr{A}(M))$ for $N \in \Grr{A}$, $\Resr{A}(M)$ has finite injective dimension.
	Similarly, $\Resl{A}(M)$ has finite injective dimension.

	(2) $\Rightarrow$ (1):
	Take a K-injective resolution $M \rightarrow I$ of $M$ in $D(\Grlr{A}{A})$ whose terms are injective objects in $\Grlr{A}{A}$ for every degree (\cref{prop:existence-K-projective-K-injective}, \cref{rmk:existence-K-projective-K-injective}).
	From the assumption, there exists a non-negative integer $i_0$ such that 
	\[
	\Extr{A}^i(N_1, \Resr{A}(M)) = 0, \quad \Extl{A}^i(N_2, \Resl{A}(M)) = 0
	\] for every $|i| > i_0$ and $N_1 \in \Grr{A}$, $N_2 \in \Grl{A}$.

	Firstly, we show that $H^i(I) = 0$ for every $|i| > i_0$.
	If $H^j(I) \neq 0$ for some $|j| > i_0$, then we have the following commutative diagram:
	\[
	\begin{tikzcd}
		B^j(\Homr{A}(Z^j(\Resr{A}(I)), \Resr{A}(I))) \ar[r, "\cong"] \ar[d]& Z^j(\Homr{A}(Z^j(\Resr{A}(I)), \Resr{A}(I))) \ar[d, "\cong"] \\
		\Homr{A}(Z^j(\Resr{A}(I)),B^j(\Resr{A}(I))) \ar[r, "\subsetneqq"] & \Homr{A}(Z^j(\Resr{A}(I)),Z^j(\Resr{A}(I))),
	\end{tikzcd}
	\]
	where $B^j(C):= \Img(d_{C}), Z^j(C):= \Ker(d_{C})$ for any complex $(C, d_C)$.
	The horizontal arrow in the top comes from the assumption.
	The horizontal arrow in the bottom is a proper inclusion because $B^j(I) \subsetneqq Z^j(I)$.
	Thus, we have a contradiction from the above diagram.
	Therefore, $H^i(I) = 0$ for every $|i| > i_0$ and there exists a complex $\tilde{I}$ of injective objects in $\Grlr{A}{A}$ such that $\tilde{I} \cong I$ in $D(\Grlr{A}{A})$ and $\tilde{I}^j = 0$ for every $j < -i_0$ (\cite[Lemma 13.15.5]{stacks-project}).
	It also holds that the smart truncation $\tau_{\leq i_0+1}(\tilde{I}) \cong M$ in $D(\Grlr{A}{A})$.

	Next, we show that the smart truncation $J = \tau_{\leq i_0+1}(\tilde{I})$ is a bounded complex such that $\Resr{A}(J^j)$ and $\Resl{A}(J^j)$ are injective over $A$ and $A^{\op}$ for every $j$, respectively.
	It is enough to show that $\Resr{A}(B^{i_0+1}(J))$ and $\Resl{A}(B^{i_0+1}(J))$ are injective over $A$ and $A^{\op}$, respectively.
	Because we have the following exact sequence of a chain complex
	\[
	0 \rightarrow \tau_{\leq i_0}(\tilde{I}) \rightarrow \sigma_{\leq i_0}(\tilde{I}) \rightarrow B^{i_0+1}(J)[-i_0] \rightarrow 0,
	\]
	where $\sigma_{\leq i_0}(\tilde{I})$ is the stupid truncation of $\tilde{I}$ at $i_0$, 
	we obtain a long exact sequence 
	 \begin{align*}
	 \cdots \rightarrow  \Extr{A}^{i_0+1}(N_1, \Resr{A}(\sigma_{\leq i_0}(\tilde{I}))) 
	 \rightarrow \Extr{A}^{i_0+1}(N_1, \Resr{A}(B^{i_0+1}(J)[-i_0])) \\
	 \rightarrow \Extr{A}^{i_0+2}(N_1, \Resr{A}(\tau_{\leq i_0}(\tilde{I}))) \rightarrow \cdots.
	 \end{align*}
	 Because $\Resr{A}(\tilde{I}^j)$ is injective in $\Grr{A}$ for every $j$ from \cref{lem:restriction} and \cref{lem:basic lemma-2} (using the right noetherian assumption), we have
	 \[
	 \Extr{A}^{i_0+1}(N_1, \Resr{A}(\sigma_{\leq i_0}(\tilde{I}))) =0.
	 \]
	Since $\tau_{\leq i_0}(\tilde{I}) \cong \tilde{I}$ in $D(\Grlr{A}{A})$, we have
	 \[
	 \Extr{A}^{i_0+2}(N_1, \Resr{A}(\tau_{\leq i_0}(\tilde{I}))) = 0.
	 \]
	 Thus, we obtain
	\[
	\Extr{A}^1(N_1, \Resr{A}(B^{i_0+1}(J))) \cong \Extr{A}^{i_0+1}(N_1, \Resr{A}(B^{i_0+1}(J)[-i_0])) = 0 \quad (N_1 \in \Grr{A})
	\]
	and $\Resr{A}(B^{i_0+1}(J))$ is injective in $\Grr{A}$.
	 Similarly, we can show that $\Resl{A}(B^{i_0+1}(J))$ is injective in $\Grl{A}$.
	 Note that, in the proof, we need to use the assumption that $A$ is left noetherian.
	 Therefore, $J$ is the desired complex.	
\end{proof}

\begin{rmk}
\label{rmk:finite injective dimension}
	It is not automatic in general that a K-injective resolution $M \rightarrow I$ of $M \in D(\Grlr{A}{A})$ can be used to compute $\RHomr{A}(N,\Resr{A}(M))$.
	However, when $A$ is right noetherian, $\Resrr{A}{i}(I')$ of an injective object $I'$ in $\Grlr{A}{A}$ is injective in $\Grr{A}$ for every $i$ by \cref{lem:restriction}, and direct sums of injective objects in $\Grr{A}$ are again injective in $\Grr{A}$ by \cref{lem:basic lemma-2}.
	These properties allow us to compute $\RHomr{A}(N,\Resr{A}(M))$ by using a suitable truncation of the restriction of a K-injective resolution whose terms are injective objects in $\Grlr{A}{A}$ as in the proof of \cref{lem:finite injective dimension}.
	This is the reason why we assume that $A$ is right noetherian in \cref{lem:finite injective dimension}.
\end{rmk}

We denote by $D_{(f,-)}^{\star}(\Grlr{A}{B})$ (resp. $D_{(-,f)}^{\star}(\Grlr{A}{B})$) the full subcategory of $D^\star(\Grlr{A}{B})$ consisting of the objects $M$ whose restrictions $\Resll{A}{j}H^i(M)$ (resp. $\Resrr{B}{j}H^i(M)$) of the cohomologies $H^i(M)$ are finite over $A$ (resp. $B$).
Here, $\star \in \{\emptyset, +, -, b\}$.

Let $A,B$ be connected $\Zbb$-algebras with $A$ noetherian.
Let $R_A$ be a dualizing complex over $A$.
Then, we define the functors $D_A, D_{A^{\op}}$ by 
\begin{align*}
	\dfun{A} : D(\Grlr{B}{A})^{\op} &\rightarrow D(\Grlr{A}{B}),\quad \dfun{A}(-) = \RHomint_{A}(-,R_A),\\
	\dfunop{A} : D(\Grlr{A}{B})^{\op} &\rightarrow D(\Grlr{B}{A}),\quad \dfunop{A}(-) = \RHomint_{A^{\op}}(-,R_A).
\end{align*}

\subsection{Properties of  (balanced) dualizing complexes}
First of all, we recall the definition of way-out functors because we often use a lemma about way-out functors.

Let $\Acal, \Bcal$ be abelian categories.
Let $\Acal', \Bcal'$ be thick (i.e. extension-closed) subcategories of $\Acal$ and $\Bcal$, respectively.
We denote by $D^{\star}_{\Acal'}(\Acal)$ and $D^{\star}_{\Bcal'}(\Bcal)$ the full subcategories of $D^{\star}(\Acal)$ and $D^{\star}(\Bcal)$ consisting of objects whose cohomology sheaves are supported on $\Acal'$ and $\Bcal'$, respectively. 
Here, $\star \in \{\emptyset, +, -\}$.

\begin{dfn}[{\cite[Chapter 1, Section 7]{hartshorne1966residues}, \cite[Definition 22.1]{hoshino1997derived}}]
	\label{dfn:way-out functors}
	Let $F: D^{\star}(\Acal) \rightarrow D(\Bcal)$ be a $\partial$-functor, where $\star \in \{\emptyset, +, -\}$.
	Then, $F$ is \emph{way-out right (resp. left)} if for every $n \in \Zbb$, there exists $m \in \Zbb$ such that for every $M \in D^{\star}(\Acal)$ with $H^i(M) = 0$ for all $i < m$ (resp. $i > m$), we have $H^i(F(M)) = 0$ for all $i < n$ (resp. $i > n$).
	$F$ is called \emph{way-out on both directions} if $F$ is way-out right and left.

\end{dfn}

\begin{lem}[{The dual version of \cite[Proposition 7.3]{hartshorne1966residues}, \cite[Proposition 23.6]{hoshino1997derived}}]
	\label{lem:way-out functors}
	Let $F, G: D^{\star}(\Acal) \rightarrow D(\Bcal)$ be $\partial$-functors, where $\star \in \{\emptyset, +, -\}$.
	Let $\eta: F \rightarrow G$ be a morphism of functors.
	Let $\Pcal$ be a collection of objects in $\Acal'$.
	We assume that 
	\begin{enumerate}[label=(\alph*)]
		\item Any object in $\Acal'$ admits an epimorphism from an object in $\Pcal$,
		\item $F$ is way-out left (resp. way-out on both directions).
	\end{enumerate}

	Then, the following hold:
	\begin{enumerate}
		\item If $F(P) \in D^{}_{\Bcal'}(\Bcal)$ for every $P \in \Pcal$, then $F(M) \in D^{}_{\Bcal'}(\Bcal)$ for every $M \in D^{-}_{\Acal'}(\Acal) \cap D^{\star}(\Acal)$ (resp. $D^{}_{\Acal'}(\Acal) \cap D^{\star}(\Acal)$).
		\item If $\eta_P: F(P) \rightarrow G(P)$ is an isomorphism for every $P \in \Pcal$, then $\eta_M: F(M) \rightarrow G(M)$ is an isomorphism for every $M \in D^{-}_{\Acal'}(\Acal) \cap D^{\star}(\Acal)$ (resp. $D^{}_{\Acal'}(\Acal) \cap D^{\star}(\Acal)$).
	\end{enumerate}

	The lemma also holds when we replace epimorphisms by monomorphisms, way-out left by way-out right and $D^{-}$ by $D^{+}$.
\end{lem}

\begin{prop}
	\label{prop:duality}
	Let $A,B$ be connected $\Zbb$-algebras with $A$ noetherian.
	Let $R_A$ be a dualizing complex over $A$.
	Then, 
	\begin{enumerate}
		\item If $M \in D^{}_{(-,f)}(\Grlr{B}{A})$ (resp. $D^{}_{(f,-)}(\Grlr{A}{B})$), then $\dfun{A}(M) \in D^{}_{(f,-)}(\Grlr{A}{B})$ (resp. $D^{}_{(-,f)}(\Grlr{B}{A})$).
		\item If $M \in D^{}_{(-,f)}(\Grlr{B}{A})$ (resp. $D^{}_{(f,-)}(\Grlr{A}{B})$), then the natural morphism $\varphi_1 : M \rightarrow \dfunop{A}(\dfun{A}(M))$ (resp. $\varphi_2 : M \rightarrow \dfun{A}(\dfunop{A}(M))$) is an isomorphism.

		In particular, we have an equivalence $D^{}_{(-,f)}(\Grlr{B}{A})^{\op} \cong D^{}_{(f,-)}(\Grlr{A}{B})$ (resp. $D^{}_{(f,-)}(\Grlr{A}{B})^{\op} \cong D^{}_{(-,f)}(\Grlr{B}{A})$).
	\end{enumerate}
\end{prop}

\begin{proof}
	(1) We only show that if $M \in D^{}_{(-,f)}(\Grlr{B}{A})$, then $\dfun{A}(M) \in D^{}_{(f,-)}(\Grlr{A}{B})$.

	We denote by $\Gr_{(-,f)}(B-A)$ (resp. $\Gr_{(f,-)}(A-B)$) the full subcategory of $\Grlr{B}{A}$ (resp. $\Grlr{A}{B}$) consisting of objects $N$ such that $\Resrr{A}{j}(N)$ is finite over $A$ for every $j \in \Zbb$.
	For each $N \in \Gr_{(-,f)}(B-A)$ and $i \in \Zbb$, there exists a finite set $J_i$ of integers that may contain duplicates such that
	there exists an epimorphism 
	\[
	\bigoplus_{j \in J_i} e_j A \twoheadrightarrow \Resrr{A}{i}(N).
	\]
	In addition, we have a natural homomorphism of $B$-$A$-bimodules
	\[
	Ke_i \ten{k} \Resrr{A}{i}(N) \longrightarrow N, \quad b \ten{} n \mapsto bn.
	\]
	Here, we use the connectedness of $B$.
	Thus, we have an epimorphism
	\[
	\bigoplus_{i \in \Zbb} \bigoplus_{j \in J_i} Ke_i \ten{k} e_j A \twoheadrightarrow N
	\]	
	and the set 
	\[
	\Pcal := \left\{
		\bigoplus_{i \in \Zbb} \bigoplus_{j \in J_i} Ke_i \ten{k} e_j A \ \middle| \  J_i :
		\text{
			$\begin{gathered}
			\text{a finite set of integers} \\
			\text{that may contain duplicates}
		\end{gathered}$}
		\text{ for each } i \in \Zbb
		\right\}
	\]
	satisfies the condition (a) in \cref{lem:way-out functors} when we take $\Acal = \Grlr{B}{A}$ and $\Acal' = \Gr_{(-,f)}(B-A)$.
	Here, note that $\dfun{A}$ is way-out on both directions because we can use the quasi-isomorphism $R_A \cong I$ in (2) of \cref{dfn:dualizing complexes} to calculate $\dfun{A}$ by \cref{rmk:derived-functors}.
	So, by \cref{lem:way-out functors} (1), it is enough to show that $\dfun{A}(P) \in D^{}_{(f,-)}(\Grlr{A}{B})$ for every $P \in \Pcal$.

	Take $P =\bigoplus_{i \in \Zbb} \bigoplus_{j \in J_i} Ke_i \ten{k} e_j A \in \Pcal$.
	Then, for any $l \in \Zbb$, we have
	\begin{align*}
		\Resll{A}{l}\left(\dfun{A}\left( P \right)\right)
		&\cong \RHomintr{A}\left(\bigoplus_{i \in \Zbb} \bigoplus_{j \in J_i} Ke_i \ten{k} e_j A, R_A\right)e_l \\
		&\cong \RHomintr{A}\left(\bigoplus_{j \in J_l} e_j A, R_A\right) \\
		&\cong \bigoplus_{j \in J_l} \RHomintr{A}(e_j A, R_A) \quad (\text{by } |J_l| < \infty) \\
		&\cong \bigoplus_{j \in J_l} R_A e_j.
	\end{align*}
	From (1) of \cref{dfn:dualizing complexes}, $H^k(R_A e_j) \in \grl{A}$ for all $k \in \Zbb$ and it follows that $\dfun{A}(P) \in D^{}_{(f,-)}(\Grlr{A}{B})$.
	Therefore, we obtain $\dfun{A}(M) \in D^{}_{(f,-)}(\Grlr{A}{B})$ for every $M \in D^{}_{(-,f)}(\Grlr{B}{A})$.

	\smallbreak
	(2) We only show that if $M \in D^{}_{(-,f)}(\Grlr{B}{A})$, then 
	\[
	\dfunop{A}(\dfun{A}(M)) \cong M \in D(\Grlr{B}{A}).
	\]
	As in (1), the set $\Pcal$ satisfies the condition (a) in \cref{lem:way-out functors} when we take $\Acal = \Grlr{B}{A}$ and $\Acal' = \Gr_{(-,f)}(B-A)$.
	In addition, $\dfunop{A} \circ \dfun{A}$ is way-out on both directions.
	So, it is enough to show that the natural morphism
	\[
	M \longrightarrow \dfunop{A}(\dfun{A}(M))
	\]
	is an isomorphism in $D(\Grlr{B}{A})$ when $M \in \Pcal$.

	Take $P=\bigoplus_{i \in \Zbb} \bigoplus_{j \in J_i} Ke_i \ten{k} e_j A \in \Pcal$.
	Then, we have
	\begin{align*}
		\Resrr{A}{l}\left(\dfunop{A}\left(\dfun{A}\left(P \right)\right)\right) 
		&= e_l\RHomint_{A^{\op}}\left(\RHomint_{A}\left(\bigoplus_{i \in \Zbb} \bigoplus_{j \in J_i} Ke_i \ten{k} e_j A, R_A\right), R_A\right) \\
		&\cong \RHomint_{A^{\op}}\left(\RHomint_{A}\left(\bigoplus_{i \in \Zbb} \bigoplus_{j \in J_i} Ke_i \ten{k} e_j A, R_A\right)e_l, R_A\right) \\
		&\cong \RHomint_{A^{\op}}\left(\bigoplus_{j \in J_l} \RHomintr{A}(e_j A, R_A), R_A \right) \quad (\text{see (1)})\\
		&\cong \bigoplus_{j \in J_l} \RHomint_{A^{\op}}(\RHomintr{A}(e_j A, R_A), R_A) \quad (\text{by } |J_l| < \infty)\\
		&\cong \bigoplus_{j \in J_l} \RHomint_{A^{\op}}(R_A e_j, R_A) \\ &\cong \bigoplus_{j \in J_l} e_j \RHomint_{A^{\op}}(R_A, R_A) \\
		&\cong \bigoplus_{j \in J_l} e_j A \quad (\text{by (3) of \cref{dfn:dualizing complexes}}).
	\end{align*}
	Since the family of exact functors
	\[
	\Resrr{A}{l} :\Gr(B-A)\to \Gr(A) \qquad (l\in\mathbb Z)
	\]
	is conservative, the induced family of functors on derived categories is also conservative. 
	Since we have shown that \(e_l\varphi_1\) is an isomorphism for every \(l\in\mathbb Z\), it follows that \(\varphi_1\) is an isomorphism in \(D(\Gr(B-A))\).

\end{proof}

\begin{rmk}
	Note that when we prove the proposition for $M \in D^{}_{(f,-)}(\Grlr{A}{B})$, we use the set $\Pcal^{\op}$ of objects in $\Gr_{(f,-)}(A-B)$ defined by
	\[
	\Pcal^{\op} := \left\{
		\bigoplus_{i \in \Zbb} \bigoplus_{j \in J_i} Ae_j \ten{k} e_i K \ \middle| \  J_i :
		\text{
			$\begin{gathered}
			\text{a finite set of integers} \\
			\text{that may contain duplicates}
		\end{gathered}$}
		\text{ for each } i \in \Zbb
		\right\}
	\]
	instead of $\Pcal$.

\end{rmk}

\begin{cor}
\label{cor:duality}
Let $A$ be a noetherian connected $\Zbb$-algebra.
Let $R_A$ be a dualizing complex over $A$.
Then, 
	\begin{enumerate}
		\item If $M \in D_f(\Grr{A})$ (resp. $D_f(\Grl{A})$), then $\dfun{A}(M) \in D_f(\Grl{A})$ (resp. $D_f(\Grr{A})$).
		\item If $M \in D_f(\Grr{A})$ (resp. $D_f(\Grl{A})$), then the natural morphism $\varphi_1 : M \rightarrow \dfunop{A}(\dfun{A}(M))$ (resp. $\varphi_2 : M \rightarrow \dfun{A}(\dfunop{A}(M))$) is an isomorphism.
	\end{enumerate}
\end{cor}

\begin{rmk}
\label{rmk:duality}
We can replace condition (3) of \cref{dfn:dualizing complexes} by condition (2) of \cref{cor:duality}. 
Indeed, suppose $R \in D^b(\Grlr{A}{A})$ satisfies conditions (1) and (2) of \cref{dfn:dualizing complexes} together with condition (2) of \cref{cor:duality}. 
Then, $R$ satisfies condition (3) of \cref{dfn:dualizing complexes}.
To see this, take $M=\Resr{A}(A)$ (resp.\ $M=\Resl{A}(A)$) in condition (2) of \cref{cor:duality}, and use that $\Resr{A}$ and $\Resl{A}$ reflect isomorphisms, i.e., $\Psi_A$ (resp.\ $\Psi_{A^{\op}}$) is an isomorphism if and only if $\Resr{A}(\Psi_A)$ (resp.\ $\Resl{A}(\Psi_{A^{\op}})$) is an isomorphism.
\end{rmk}

\begin{thm}
	\label{thm:local duality-2}
	Let $A$ be a noetherian connected $\Zbb$-algebra and $R_A$ be a balanced dualizing complex over $A$.
	Let $B$ be another connected $\Zbb$-algebra.
	Then, we have an isomorphism in $D(\Grlr{A}{B})$
	\begin{align*}
		\RHomint_A(M,R_A) \cong \D{\Rtorfun{A}(M)}
	\end{align*}
	for all $M \in D_{(-,f)}(\Grlr{B}{A})$.
	Similarly,
	we have an isomorphism in $D(\Grlr{B}{A})$
	\begin{align*}
		\RHomint_{A^{\op}}(M,R_A) \cong \D{\Rtorfun{A^{\op}}(M)}
	\end{align*}
	for all $M \in D_{(f,-)}(\Grlr{A}{B})$.
	These isomorphisms are functorial in $M$.
\end{thm}

\begin{proof}
	We follow the strategy in the proofs of \cite[Proposition 3.4]{chan2002pre} or \cite[Theorem 17.2.7]{yekutieli2019derived}.
	We only prove the first isomorphism.
	The second isomorphism can be proved in the same way.

	Firstly, we construct a morphism from $\RHomint_A(-,R_A)$ to $\D{\Rtorfun{A}(-)}$.
	Take the following K-injective resolutions 
	\begin{align*}
		\varphi_1: R \longrightarrow J, \quad
		\varphi_2: \torfun{A}(J) \longrightarrow K , \quad 
		\varphi_M: M \longrightarrow I.
	\end{align*}
	We also have natural morphisms
	\begin{align*}
	\psi : \Homintr{A}(I,J) &\longrightarrow \Homintr{A}(\torfun{A}(I),\torfun{A}(J)), \\
	(\varphi_2)_{*} : \Homintr{A}(\torfun{A}(I),\torfun{A}(J)) &\longrightarrow \Homintr{A}(\torfun{A}(I),K).
	\end{align*}
	By considering the composition of morphisms $\psi, (\varphi_2)_{*},  \alpha, \beta$, we obtain a morphism $\xi_M$
	\[
	\begin{tikzcd}
		\RHomintr{A}(M,R_A) \ar[rr, "(\varphi_2)_{*} \circ \psi"] \ar[ddrr, bend right=25, "\xi_M"'] 
		& & \RHomintr{A}(\Rtorfun{A}(M), \Rtorfun{A}(R)) \ar[d, "\cong", "\alpha"'] \\
		& \mathlarger{\circlearrowleft} & \RHomintr{A}(\Rtorfun{A}(M), \D{A}) \ar[d, "\cong", "\beta"']  \\
		&& \D{\Rtorfun{A}(M)}
	\end{tikzcd}
	\]
	, i.e. $\xi_M = \beta \circ \alpha \circ (\varphi_2)_{*} \circ \psi$.
	Here, $\alpha$ is an isomorphism in the balancedness of $R_A$ in \cref{dfn:dualizing complexes} and $\beta$ is \cref{lem:basic lemma-1}.
	This $\xi_M$ is functorial on $M$.
	Hence, we have a morphism $\xi$ of functors from $\RHomint_A(-,R_A)$ to $\D{\Rtorfun{A}(-)}$.

	Secondly, we show that $\xi$ is an isomorphism.
	We put $F = \D{\Rtorfun{A}(-)}$.
	From \cref{prop:duality}, $\dfun{A} \circ \dfunop{A} \cong \Id_{D_{(f,-)}(\Grlr{A}{B})}$.
	Thus, it is enough to show that the natural morphism
	\[
	 \xi \circ \Id_{\dfunop{A}} : \dfun{A} \circ \dfunop{A} \rightarrow F \circ \dfunop{A}
	\]
	is an isomorphism.
	$F$ and $\dfunop{A}$ are way-out on both directions.
	So, by \cref{lem:way-out functors} (2), it is enough to show that 
	\[
	(\xi \circ \Id_{\dfunop{A}})_P : (\dfun{A} \circ \dfunop{A})(P)\longrightarrow (F \circ \dfunop{A})(P)
	\]
	is an isomorphism for every $P \in \Pcal^{\op}$.

	Take $P = \bigoplus_{i \in \Zbb} \bigoplus_{j \in J_i} Ae_j \ten{k} e_i K \in \Pcal^{\op}$.
	Then, for any $l \in \Zbb$, we have
	\begin{align*}
		\Resll{A}{l}((F \circ \dfunop{A})(P)) &\cong \D{\Rtorfun{A}(\RHomintl{A}(P, R_A))} \ e_l \\
		&\cong \D{(e_l\Rtorfun{A}(\RHomintl{A}(P, R_A)))} \\
		&\cong \D{\Rtorfun{A}(e_l\RHomintl{A}(P, R_A))} \\
		&\cong \D{\Rtorfun{A}(\RHomintl{A}(Pe_l, R_A))} \\
        &\cong \D{\Rtorfun{A}\left(\RHomintl{A}\left(\bigoplus_{j \in J_l} Ae_j, R_A \right)\right)} 
        \\
		&\cong \bigoplus_{j \in J_l}\D{\Rtorfun{A}\left(\RHomintl{A}\left(Ae_j, R_A \right)\right)} \\
		&\cong \bigoplus_{j \in J_l}\D{\Rtorfun{A}(e_j R_A)} 
		\cong \bigoplus_{j \in J_l}\D{\Rtorfun{A}(R_A)} \ e_j \\
		&\cong \bigoplus_{j \in J_l} A e_j 
		\cong \bigoplus_{j \in J_l} \Resll{A}{l}(P).
	\end{align*}
	This isomorphism shows that $\Resll{A}{l}((\xi \circ \Id_{\dfunop{A}})_P)$ is an isomorphism of objects in $D(\Grr{A})$ for every $l \in \Zbb$.
	Therefore, $(\xi \circ \Id_{\dfunop{A}})_P$ is an isomorphism in $D(\Grlr{A}{B})$.
	Hence, $\xi \circ \Id_{\dfunop{A}}$ is an isomorphism and $\xi$ is also an isomorphism.
\end{proof}

As \cref{cor:local-duality-2}, we can show the following corollary.

\begin{cor}
	\label{cor:local-duality-2-1}
	Let $A$ be a noetherian connected $\Zbb$-algebra and $R_A$ be a balanced dualizing complex over $A$.
	Then, we have an isomorphism in $D_f(\Grl{A})$
	\[
	\RHomint_A(M,R_A) \cong \D{\Rtorfun{A}(M)}
	\]
	for all $M \in D_f(\Grr{A})$.
	This isomorphism is functorial in $M$.
\end{cor}

The following corollary is important for us.
\begin{cor}
	\label{cor:local-duality-2-2}
	Let $A$ be a noetherian connected $\Zbb$-algebra and assume that $A$ has a balanced dualizing complex $R_A$.
	Then, 
	\begin{enumerate}
		\item $R_A \cong \D{\Rtorfun{A}(A)} \cong \D{\Rtorfun{A^{\op}}(A)}$ in $D(\Grlr{A}{A})$. 
		Moreover, $A$ has symmetric derived torsion.
		\item $A$ satisfies the $\chi$-condition.
		\item $\cd(\torfun{A})$ and $ \cd(\torfun{A^{\op}})$ are finite. 
	\end{enumerate}
\end{cor}

\begin{proof}
	First, we prove (3). 
	We only show that $\cd(\torfun{A})$ is finite.
	We can show that $\cd(\torfun{A^{\op}})$ is finite in the same way.

	Note that we can use the resolution $R_A \cong I$ in (2) of \cref{dfn:dualizing complexes} to calculate $\dfun{A}$ by \cref{rmk:derived-functors}.
	If $i_0$ is the minimal integer such that $I^{i_0} \neq 0$, then by using \cref{cor:local-duality-2-1},
	\[
	H^{i}(\dfun{A}(M)) \cong \R^{-i}\torfun{A}(M) = 0
	\]
	for all $i < i_0$ and $M \in \grr{A}$.
	Thus, we have $\cd(\torfun{A}) \leq -i_0$.

	(1) We apply \cref{thm:local duality-2} to $M = A$.
	Then, we have 
	\begin{align*}
	\D{\Rtorfun{A}(A)} &\cong \RHomint_A(A,R_A) \cong R_A, \\
	\D{\Rtorfun{A^{\op}}(A)} &\cong \RHomint_{A^{\op}}(A,R_A) \cong R_A
	\end{align*}
	in $D(\Grlr{A}{A})$.
	This shows the first claim of (1).
	In addition, we have an isomorphism $\Rtorfun{A}(A) \cong \Rtorfun{A^{\op}}(A)$ from \cref{thm:matlis-duality}, which implies that $A$ has weak symmetric derived torsion.
	Hence, we obtain the second claim from \cref{prop:symmetric derived torsion functor} and (3).

	(2) We only show that $A$ satisfies the right $\chi$-condition.
	We can show that $A$ satisfies the left $\chi$-condition in the same way.

	Let $M \in \grr{A}$.
	From \cref{cor:duality}, 
	\[
	H^{-i}{\dfun{A}(M)} = \Extintr{A}^{-i}(M,R_A) \in \grl{A}
	\] 
	for all $i \in \Zbb$.
	Then, from \cref{cor:local-duality-2-1}, we have
	\[
	\R^i\torfun{A}(M) \cong \D{\Extintr{A}^{-i}(M,R_A)} \in \grcofr{A},
	\]
	for all $i \in \Zbb$.
	Hence, $A$ satisfies the right $\chi$-condition from \cref{cor:chi-condition}.
\end{proof}

\subsection{Existence of balanced dualizing complexes}

Finally, we show the existence of balanced dualizing complexes, which is the first main theorem of this paper.

\begin{thm}
	\label{thm:main1}
	Let $A$ be a noetherian connected $\Zbb$-algebra. 
	
	Then, $A$ has a balanced dualizing complex if and only if 
	$A$ satisfies the $\chi$-condition, $\cd(\torfun{A}), \cd(\torfun{A^{\op}})$ are finite and $A$ has symmetric derived torsion.
\end{thm}

\begin{proof}
	The "only if" part is \cref{cor:local-duality-2-2}.
	We show the "if" part.
	We prove that 
	\[
	R :=\D{\Rtorfun{A}(A)} \cong \D{\Rtorfun{A^{\op}}(A)}
	\] 
	is a balanced dualizing complex over $A$.
	Note that $R \in D^b(\Grlr{A}{A})$ because $\cd(\torfun{A})$ and $\cd(\torfun{A^{\op}})$ are finite.

	Firstly, we show that $R$ satisfies (1) of \cref{dfn:dualizing complexes}.
	Since $A$ satisfies the $\chi$-condition, from \cref{cor:chi-condition}, we have 
	\begin{align*}
	e_j\R^i\torfun{A}(A) &\cong \R^i\torfun{A}(e_jA) \in \grcofr{A}, \\
	\R^i\torfun{A^{\op}}(A)e_j &\cong \R^i\torfun{A^{\op}}(Ae_j) \in \grcofl{A}
	\end{align*}
	for all $i,j \in \Zbb$.
	This means that $\Resll{A}{j}(H^i(R)) \in \grl{A} , \Resrr{A}{j}(H^i(R)) \in \grr{A}$ for all $i,j \in \Zbb$.

	Secondly, we show that $R$ satisfies (2) of \cref{dfn:dualizing complexes}.
	By \cref{cor:local-duality-2}, we have 
	\begin{align*}
	\RHomint_A(-,\Resr{A}(R)) &\cong \D{\Rtorfun{A}(-)}, \\
	 \RHomint_{A^{\op}}(-,\Resl{A}(R)) &\cong \D{\Rtorfun{A^{\op}}(-)}.
	\end{align*}
	Since $\cd(\torfun{A})$ and $\cd(\torfun{A^{\op}})$ are finite, the above formulas imply that $\Resr{A}(R)$ and $\Resl{A}(R)$ have finite injective dimensions over $A$ and $A^{\op}$, respectively.
	Hence, from \cref{lem:finite injective dimension}, $R$ satisfies (2) of \cref{dfn:dualizing complexes}.

	Finally, we show that $R$ satisfies (3) of \cref{dfn:dualizing complexes}.
	In fact, we have the following diagram:
	\[
	\begin{tikzcd}
		A \ar[rr, "\Psi_A"]\ar[d, equal]  & & \RHomintr{A}(R,R) \ar[d, equal] \\
		A \ar[rr, "\Psi_A"] \ar[d, equal]  & & \RHomintr{A}(\D{\Rtorfun{A}(A)}, \D{\Rtorfun{A}(A)}) \ar[d, "\cong", "\tau_1"'] & (\text{Definition of $R$}) \\
		A \ar[rr, "\Phi_1"] \ar[d, equal]   & &  \RHomintl{A}(\Rtorfun{A}(A), \Rtorfun{A}(A)) \ar[d, "\cong", "\tau_2"'] & (\text{\cref{thm:matlis-duality}, \cref{cor:matlis-duality}}) \\
		A \ar[rr, "\Phi_2"] \ar[d, equal]   & & \RHomintl{A}(\Rtorfun{A^{op}}(A), \Rtorfun{A^{\op}}(A)) \ar[d, "\cong", "\tau_3"'] & ( \epsilon_A : \Rtorfun{A}(A) \xrightarrow{\cong} \Rtorfun{A^{\op}}(A)) \\
		A \ar[rr, "\Phi_3"] \ar[d, equal]   & & \RHomintl{A}(\Rtorfun{A^{op}}(A), A)  \ar[d, "\cong", "\tau_4"'] & (\text{\cref{prop:basic proposition-10}}) \\ 
		A \ar[rr, "\Phi_4"] \ar[d, equal]   & & \RHomintr{A}(\D{A}, \D{\Rtorfun{A^{op}}(A)}) \ar[d, "\cong", "\tau_5"'] & (\text{\cref{cor:matlis-duality}}) \\
		A \ar[rr,"\Phi_5"] \ar[d, equal]   & &  \RHomintr{A}(\D{A}, \D{\Rtorfun{A}(A)}) \ar[d, "\cong", "\tau_6"'] & (\epsilon_A : \Rtorfun{A}(A) \xrightarrow{\cong} \Rtorfun{A^{\op}}(A))  \\
		A \ar[rr, "\Phi_6"] \ar[d, equal]   & &  A''  \ar[d, "\cong", "\tau_7"'] & (\text{\cref{thm:local duality}}) \\ 
		A \ar[rr, equal]   & &  A & (\text{\cref{thm:matlis-duality}}),
	\end{tikzcd}
	\]
	where 
	\begin{itemize}
		\item $\Phi_1$ is the homothety map of ${\Rtorfun{A}(A)}$,
		\item $\Phi_2$ is the homothety map of $\Rtorfun{A^{op}}(A)$,
		\item $\Phi_3$ is obtained from the homothety map of $\Rtorfun{A^{op}}(A)$ and the natural map $\sigma^{\op}_{A} : \Rtorfun{A^{\op}}(A) \rightarrow A$,
		\item $\Phi_4$ is obtained from the homothety map of $\D{A}$ and the dual of the natural map $\sigma^{\op}_{A} : \Rtorfun{A^{\op}}(A) \rightarrow A$,
		\item $\Phi_5$ is obtained from the homothety map of $\D{A}$ and the dual of the natural map $\sigma_{A} : \Rtorfun{A}(A) \rightarrow A$,
		\item $\Phi_6$ is the natural map from $A$ to $\DD{A}$.
	\end{itemize}
	The maps $\tau_1, \cdots, \tau_7$ are isomorphisms for the reasons indicated next to each map.
	
	The first square is trivially commutative.
	The second square is commutative because the homothety map is compatible with the Matlis duality.
	The third square is commutative because the homothety map is compatible with the natural map $\epsilon_A$.
	The fourth square is commutative because the composition of $\Phi_2$ and $\tau_3$ coincides with $\Phi_3$ by definition.
	The fifth square is commutative because the homothety map is compatible with the Matlis duality and the following natural diagram is commutative:
	\[
	\begin{tikzcd}
		\RHomintl{A}(\Rtorfun{A^{\op}}(A), \Rtorfun{A^{\op}}(A)) \ar[r] \ar[d] & \RHomintl{A}(\Rtorfun{A^{op}}(A), A) \ar[d] \\
		\RHomintr{A}(\D{\Rtorfun{A^{\op}}(A)}, \D{\Rtorfun{A^{\op}}(A)}) \ar[r] & \RHomintr{A}(\D{A}, \D{\Rtorfun{A^{op}}(A)}).
	\end{tikzcd}
	\]
	The sixth square is commutative because the homothety map is compatible with the natural map $\epsilon_A$.
	The seventh square is commutative because of the naturality of the local duality (\cref{thm:local duality}) and \cref{rmk:local duality}.	
	The last square is commutative because $\Phi_6$ and $\tau_7$ are inverses of each other.
	This shows that $\Psi_A$ in (3) of \cref{dfn:dualizing complexes} is an isomorphism.
	
	In the same way, we can show that $\Psi^{\op}_A$ in (3) of \cref{dfn:dualizing complexes} is also an isomorphism.

	For the balancedness of $R$, it follows from \cref{thm:matlis-duality}, \cref{thm:local duality} and  \cref{dfn:dualizing complexes} (3).

	Therefore, $R$ is a balanced dualizing complex over $A$.
\end{proof}

\subsection{Examples}

We give a class of $\Zbb$-algebras which have balanced dualizing complexes.
We also compare the notion of balanced dualizing complexes over $\Zbb$-algebras with that over graded algebras.

The following definition is inspired by \cite[Definition 4.15]{mori2025categorical} (see also \cite[Definition 3.1]{mori2024corrigendum}).

\begin{dfn}
	\label{dfn:AS-Gorenstein}
	Let $A$ be a connected $\Zbb$-algebra.
	We say that $A$ is \emph{AS-Gorenstein} of dimension $d$ and Gorenstein parameter $l$ if the following conditions hold:
	\begin{enumerate}
		\item $\injdim_{\Grr{A}} \Resr{A}(A) = \injdim_{\Grr{A^{\op}}} \Resl{A}(A) = d$,
		\item $\D{\Rtorfun{A}(A)} \cong \D{\Rtorfun{A^{\op}}(A)} \cong A(0,-l)_{\varphi}[d]$ in $D(\Grlr{A}{A})$ for some $l,d \in \Nbb$ and some isomorphism of $\Zbb$-algebras $\varphi: A \rightarrow A(-l,-l) $.
		$\varphi$ is called a \emph{Nakayama automorphism} of $A$.
	\end{enumerate}
\end{dfn}

\begin{prop}
\label{prop:AS-Gorenstein}
Let $A$ be a noetherian connected $\Zbb$-algebra.
If $A$ is AS-Gorenstein of dimension $d$ and Gorenstein parameter $l$ with a Nakayama automorphism $\varphi$, then $A$ has a balanced dualizing complex $R_A$, which is isomorphic to $A(0,-l)_{\varphi}[d]$.
\end{prop}

\begin{proof}
	We prove that $A(0,-l)_{\varphi}[d]$ is a balanced dualizing complex of $A$.

	Firstly, from \cite[Lemma 2.16]{mori2025categorical}
	\begin{align*}
	e_iA(0,-l)_{\varphi} &\cong (e_{i}A)(-l)_\varphi \\
	&\cong e_{i+l}A \in \grr{A}, \\
	A(0,-l)_{\varphi}e_i &\cong Ae_{i-l} \in \grl{A}.
	\end{align*}
	This shows that $A(0,-l)_{\varphi}[d]$ satisfies (1) of \cref{dfn:dualizing complexes}.

	From the condition $\injdim_{\Grr{A}} \Resr{A}(A) = \injdim_{\Grr{A^{\op}}} \Resl{A}(A)$ and \cref{lem:finite injective dimension}, we obtain an object $I \in D^b(\Grlr{A}{A})$ and a quasi-isomorphism $A \rightarrow I$ in $D(\Grlr{A}{A})$ such that $\Resr{A}(I^i)$ and $\Resl{A}(I^i)$ are injective over $\Grr{A}$ and $\Grr{A^{\op}}$, respectively, for all $i \in \Zbb$.
	So, $I(0, -l)_{\varphi}[d]$ and  $A(0,-l)_{\varphi}[d] \rightarrow I(0, -l)_{\varphi}[d]$ have the same properties.
	This shows that $A(0,-l)_{\varphi}[d]$ satisfies (2) of \cref{dfn:dualizing complexes}.

	We have the natural isomorphisms in the following:
	\begin{align*}
		A &\xrightarrow{\cong} \Homintr{A}(A(0,-l)_{\varphi}, A(0,-l)_{\varphi}) \\
		  &\xrightarrow{\cong} \Homintr{A}(A(0,-l)_{\varphi}[d], A(0,-l)_{\varphi}[d]) \\
		  &\xrightarrow{\cong} \RHomintr{A}(A(0,-l)_{\varphi}[d], A(0,-l)_{\varphi}[d]) 
	\end{align*}
	in $D(\Grlr{A}{A})$.
	The composition of these isomorphisms is exactly the homothety morphism
	\[
	A \to \mathbf{R}\operatorname{Hom}_A(A(0,-l)_\varphi[d],A(0,-l)_\varphi[d]).
	\]
	In the same way, the homothety morphism of $A(0,-l)_{\varphi}[d]$
	\[
	A \rightarrow \RHomintl{A}(A(0,-l)_{\varphi}[d], A(0,-l)_{\varphi}[d])
	\] 
	is an isomorphism.
	Thus, $A(0,-l)_{\varphi}[d]$ satisfies (3) of \cref{dfn:dualizing complexes}.

	Finally, from \cref{thm:matlis-duality}, \cref{thm:local duality} and the above calculation, we have
	\begin{align*}
		\Rtorfun{A}(A(0,-l)_{\varphi}[d]) &\cong \D{\RHomint_A(A(0,-l)_{\varphi}[d], A(0,-l)_{\varphi}[d])} \\
		&\cong \D{A}
	\end{align*}
	in $D(\Grlr{A}{A})$.
	In the same way, we have an isomorphism $\Rtorfun{A^{\op}}(A(0,-l)_{\varphi}[d]) \cong \D{A}$.

	Therefore, $A(0,-l)_{\varphi}[d]$ is a balanced dualizing complex over $A$.
\end{proof}

Next, we consider a graded $k$-algebra $B = \bigoplus_{i \in \Zbb} B_i$.
We denote by $\Grr{B}$ and $\Grl{B}$ the categories of graded right and left $B$-modules, respectively, and by $\torfun{B}$ and $\torfun{B^{\op}}$ the corresponding torsion functors.
For other analogous notions, we use the same notation as in the $\Zbb$-algebra case.

For completeness, we include some basic properties relevant to our study, together with their proofs, although they are likely well known to experts.

\begin{lem}
	\label{lem:comparison of dualizing complexes}
	Let $B,C,D$ be connected graded $k$-algebras.
	Then, the following hold:

	\begin{enumerate}
		\item For an object $M$ in $D(\Grlr{C}{B})$ (resp. $D(\Grr{B})$), we have 
		\[
		\widezgrz{\Rtorfun{B}(M)} \cong \Rtorfun{\zgrz{B}}(\zgrz{M})
		\]
		in $D(\Grlr{\zgrz{C}}{\zgrz{B}})$ (resp. $D(\Grr{\zgrz{B}})$).
		(see also \cite[Lemma 4.12]{mori2025categorical})

		\item For objects $M \in D(\Grlr{C}{B})$ and $N \in D(\Grlr{D}{B})$ (resp. $D(\Grr{B})$), we have 
		\[
		\widezgrz{\RHomintr{B}(M,N)} \cong \RHomintr{\zgrz{B}}(\zgrz{M}, \zgrz{N})
		\]
		in $D(\Grlr{\zgrz{D}}{\zgrz{C}})$ (resp. $D(\Grr{\zgrz{C}})$).
	\end{enumerate}
\end{lem}

\begin{proof}
	(1) 
	We show the underived version of the claim.
	Let $m$ be an element of $M$ and $\zgrz{m} \in \zgrz{M}$ be the corresponding element.
	Let $b \in B$ be an element and $\zgrz{b} \in \zgrz{B}$ be the corresponding element.
	Then, the claim follows from the fact that $\widezgrz{mb} = \zgrz{m}\zgrz{b}$.
	As for the derived version, the claim follows from the isomorphism of functors $\R(\zgrz{(-)} \circ \torfun{B}) \cong \zgrz{(-)} \circ \R(\torfun{B})$ due to exactness of the functor $\zgrz{(-)}$ and associativity of derived functors.

	(2) We show the underived version of the claim.
	For $M \in \Grlr{C}{B}$ and $N \in \Grlr{D}{B}$, we have
	\begin{align*}
		\Homintr{\zgrz{B}}(\zgrz{M}, \zgrz{N})_{i,j} &\cong \Homr{\zgrz{B}}(e_j\zgrz{M}, e_i\zgrz{N}) \\
		&\cong \Homr{B}(M(-j), N(-i)) \\
		&\cong \Homr{B}(M, N(j-i)) \\
		&\cong \Homintr{B}(M,N)_{j-i} \\
		&\cong \widezgrz{\Homintr{B}(M,N)}_{i,j}.
	\end{align*}
	In this calculation, we used the fact that $\widezgrz{M(-j)} = e_j\zgrz{M}(-j)$ and $\widezgrz{N(-i)} = e_i\zgrz{N}(-i)$ (\cite[Lemma 2.17,2.18]{mori2025categorical}).
	If $N \in \Grr{B}$, we can prove the claim in the same way.
	As for the derived version, the claim follows from the isomorphism of functors $\R(\widezgrz{(-)} \circ \Homintr{B}(-,-)) \cong \widezgrz{(-)} \circ \R(\Homintr{B}(-,-))$ due to exactness of the functor $\zgrz{(-)}$ and associativity of derived functors.
\end{proof}

\begin{prop}
\label{prop:comparison of dualizing complexes}
Let $B$ be a noetherian connected graded $k$-algebra.
Then, $B$ has a balanced dualizing complex in the sense of \cite[Definition 3.3 and 4.1]{yekutieli1992dualizing} if and only if $\zgrz{B}$ has a balanced dualizing complex.
Moreover, in this case, we have $R_{\zgrz{B}} \cong \widezgrz{R_B}$.
\end{prop}

\begin{proof}	
Let $\zgrz{M} \in \Grr{\zgrz{B}}$.
Then, from (2) of \cref{lem:comparison of dualizing complexes} and the fact $\zgrz{k} \cong K$, we have
\begin{align*}
	\Extintr{\zgrz{B}}^i(K, \zgrz{M}) &\cong \widezgrz{\Extintr{B}^i(k, M)} \in \Grr{\zgrz{B}}.
\end{align*}
This shows that $\zgrz{B}$ satisfies the right $\chi$-condition in the sense of \cite[Definition 16.5.14]{yekutieli2019derived} (see also \cite[Definition 3.2 and 3.7, Proposition 3.11]{artin1994noncommutative}) if and only if $B$ satisfies the right $\chi$-condition.
In the same way, we can show that $\zgrz{B}$ satisfies the left $\chi$-condition if and only if $B$ satisfies the left $\chi$-condition.

As for $\cd(\torfun{\zgrz{B}})$, from (1) of \cref{lem:comparison of dualizing complexes}, we have
\[
\R^i\torfun{\zgrz{B}}(\zgrz{M}) \cong \widezgrz{\R^i\torfun{B}(M)} \in \Grr{\zgrz{B}}.
\]
This shows that $\cd(\torfun{\zgrz{B}})$ is finite if and only if $\cd(\torfun{B})$ is finite.
In the same way, we can show that $\cd(\torfun{\zgrz{B^{\op}}})$ is finite if and only if $\cd(\torfun{B^{\op}})$ is finite.

Therefore, from \cref{thm:main1} and \cite[Theorem 6.3]{vandenbergh1997existence}, we conclude that $\zgrz{B}$ has a balanced dualizing complex if and only if $B$ has a balanced dualizing complex.
Moreover, we obtain $R_{\zgrz{B}} \cong \widezgrz{R_B}$ from the above discussion.
\end{proof}

We can also compare the notion of AS-Gorenstein $\Zbb$-algebras with that of AS-Gorenstein graded algebras.
As for the case of AS-regular algebras, see \cite[Section 4]{mori2025categorical}.

\begin{prop}
\label{prop:comparison of AS-Gorenstein}
Let $B$ be a noetherian connected graded $k$-algebra.
Then, $B$ is AS-Gorenstein of dimension $d$ and Gorenstein parameter $l$ in the sense of \cite[Definition 15.4.8]{yekutieli2019derived} if and only if $\zgrz{B}$ is AS-Gorenstein of dimension $d$ and Gorenstein parameter $l$.
\end{prop}

\begin{proof}
Firstly, note that $\zgrz{B}$ has the injective dimension $d$ over $\zgrz{B}$ and $\zgrz{B^{\op}}$ if and only if $B$ has the injective dimension $d$ over $B$ and $B^{\op}$ from \cref{lem:associated-Z-algebras}.

Assume that $B$ is AS-Gorenstein of dimension $d$ and Gorenstein parameter $l$.
Then, there is an isomorphism $\nu: B \rightarrow B$ such that 
\[R_B \cong \D{\Rtorfun{B}(B)} \cong \D{\Rtorfun{B^{\op}}(B)} \cong B(-l)_\nu[d]
\]
(for example, see \cite[Corollary 17.3.14]{yekutieli2019derived}).
Since $\zgrz{B}$ is 1-periodic, we have the canonical isomorphism of $\Zbb$-algebras $\psi_1: \zgrz{B} \rightarrow \zgrz{B}(-l,-l)$.
Thus, from \cref{lem:comparison of dualizing complexes}, we have
\[
\D{\Rtorfun{\zgrz{B}}(\zgrz{B})} \cong \D{\Rtorfun{\zgrz{B}^{\op}}(\zgrz{B})} \simeq \zgrz{B}(0,-l)_{\psi_1 \circ \zgrz{\nu}}[d].
\]
Hence, $\zgrz{B}$ is AS-Gorenstein of dimension $d$ and Gorenstein parameter $l$.

Assume that $\zgrz{B}$ is AS-Gorenstein of dimension $d$ and Gorenstein parameter $l$.
Then, in $D(\Grlr{\zgrz{B}}{\zgrz{B}})$, 
\begin{align*}
\widezgrz{\RHomintr{B}(k,B)} &\cong \RHomintr{\zgrz{B}}(K,\zgrz{B})  \quad (\text{\cref{lem:comparison of dualizing complexes}})\\
&\cong \RHomintr{\zgrz{B}}(K,\D{\Rtorfun{\zgrz{B}}(\zgrz{B})}(0,l)_{\psi_2 \circ \zgrz{\nu}^{-1}}[-d]) \\
&\cong \RHomintr{\zgrz{B}}(K,\D{\Rtorfun{\zgrz{B}}(\zgrz{B})}) (0,l)_{\psi_2 \circ \zgrz{\nu}^{-1}}[-d] \\
&\cong \D{\Rtorfun{\zgrz{B}}(K)}(0,l)_{\psi_2 \circ \zgrz{\nu}^{-1}}[-d] \quad (\text{\cref{thm:local duality}}) \\
&\cong K(0,l)_{\psi_2 \circ \zgrz{\nu}^{-1}}[-d] \quad (\text{\cref{prop:basic proposition-10}}), 
\end{align*}
where $\psi_2: \zgrz{B} \rightarrow \zgrz{B}(l,l)$ is the canonical isomorphism of $\Zbb$-algebras.
Moreover, we have
\begin{align*}
	e_0\widezgrz{\RHomintr{B}(k,B)} \cong \RHomintr{B}(k,B), \\
	e_0K(0,l)_{\psi_2 \circ \nu^{-1}}[-d] \cong k(l)_{\psi_2 \circ \nu^{-1}}[-d].
\end{align*}
Thus, we obtain 
\[
\RHomintr{B}(k,B) \cong k(l)[-d] \in D(\Grr{k}).
\]
In the same way, we can obtain 
\[
\RHomintr{B^{\op}}(k,B) \cong k(l)[-d] \in D(\Grr{k}).
\]
Therefore, $B$ is AS-Gorenstein of dimension $d$ and Gorenstein parameter $l$.
\end{proof}

\section{An application to noncommutative projective geometry}
\label{sec:application to noncommutative projective geometry}

In this section, we give an application of \cref{thm:main1} to noncommutative projective geometry.
Especially, we show that the noncommutative projective scheme over a $\Zbb$-algebra has a Serre functor when it has a balanced dualizing complex and its global dimension is finite.

\subsection{Noncommutative projective schemes}

Let $A$ be a right noetherian connected $\Zbb$-algebra.
Then, the category $\grr{A}$ is an abelian category and the subcategory $\tor(A)$ of torsion graded right $A$-modules is a Serre subcategory of $\grr{A}$ (\cite[Lemma 2.9, Lemma 3.7]{mori2025categorical}).
Thus, we can consider the quotient category $\qgr(A) = \grr{A}/\tor(A)$.
We denote by $\pi_A: \grr{A} \longrightarrow \qgr(A)$ the natural projection functor.
Note that $\pi_A$ has a right adjoint functor $\omega_A: \qgr(A) \longrightarrow \grr{A}$ (\cite[Section 3.1]{mori2025categorical}) and $\pi_A \circ \omega_A \cong \Id_{\qgr(A)}$.
In particular, $\omega$ is fully faithful (\cite[Lemma 4.24.4]{stacks-project}).

We can also define the quotient category $\QGr(A) = \Gr(A)/\Tor(A)$ in the same way.
In addition, we denote by $\pi_A$ the natural projection functor and by $\omega_A: \QGr(A) \longrightarrow \Gr(A)$ its right adjoint functor.

\begin{dfn}[{cf. \cite[Section 2]{artin1994noncommutative}, \cite[Definition 3.8]{mori2025categorical}}]
	\label{dfn:noncommutative projective scheme}
	Let $A$ be a right noetherian connected $\Zbb$-algebra.
	Then, the \emph{noncommutative projective scheme} associated to $A$ is defined to be the category $\qgr(A)$.
\end{dfn}

\subsection{Serre functors of noncommutative projective schemes over $\Zbb$-algebras}

\begin{dfn}
\label{dfn:Serre-functor}
Let $\Tcal$ be a $k$-linear triangulated category.
A \emph{Serre functor} of $\Tcal$ is an autoequivalence $\Scal_{\Tcal}$ of $\Tcal$ such that there exists a natural isomorphism
\[
\Hom_{\Tcal}(X,Y) \cong \D{\Hom_{\Tcal}(Y, \Scal_{\Tcal}(X))}
\]
for all $X,Y \in \Tcal$, where $(-)'$ denotes the $k$-dual.
\end{dfn}

\begin{dfn}
\label{dfn:global dimension}
Let $\Ccal$ be an abelian category.
Assume that $\Ext^i(X,Y)$ is defined for all $X,Y \in \Ccal$ and $i \in \Nbb$.
The \emph{global dimension} $\gldim(\Ccal)$ of $\Ccal$ is defined to be
\[
\gldim(\Ccal) = \supr \{i \in \Nbb \mid \Ext^i_{\Ccal}(X,Y) \neq 0 \text{ for some } X,Y \in \Ccal \}.
\]
\end{dfn}

We often use the following lemma in the proof of \cref{thm:Serre-functor}.
We can prove this lemma in the same way as \cite[Lemma A.1]{de2004ideal}.
\begin{lem}
	\label{lem:Serre-functor-1}
	Let $A$ be a right noetherian connected $\Zbb$-algebra.
	We assume that $\gldim{(\qgr(A))}$ is finite.
	For any $\pi_A(M) \in D^b(\qgr(A))$, there exists an object $P \in D^b(\gr(A))$ such that $P^i$ is projective for all $i \in \Zbb$ and $\pi_A(M) \oplus \pi_A(L) \cong \pi_A(P)$ for some $\pi_A(L) \in D^b(\qgr(A))$.
\end{lem}

We define a functor 
\[
Q_A: \grr{A} \rightarrow \grr{A} \quad (\text{or } \Grr{A} \rightarrow \Grr{A})
\]
by $Q_A = \omega_A \circ \pi_A$. 
Note that as in the functor $\torfun{A}$, we can extend the functor $\Q_A$ on the category of bimodules (cf. \cite[Lemma 3.13]{mori2025categorical}).
Moreover, $\R^i Q_A$ commutes with direct limits for any $i$ from \cref{prop:quasi-compactness-torfunctor} and \cite[Lemma 3.12]{mori2025categorical}.

We also need the following proposition (cf. \cite[Theorem 6.6]{mori2025categorical}).

	\begin{prop}
		\label{prop:Serre-functor-2}
		Let $A$ be a right noetherian connected $\Zbb$-algebra.
		Let $M$ be an object in $D^b_{lf}(\Grr{A})$ with the property that $M$ has finite projective dimension, where $D^b_{lf}(\Grr{A}) := D^b(\Grr{A}) \cap D_{lf}(\Grr{A})$.
		Let $N$ be an object in $D^{-}_f(\Grr{A})$, where $D^{-}_f(\Grr{A}) := D^-(\Grr{A}) \cap D_{f}(\Grr{A})$.
		We assume that $\cd(\torfun{A})$ is finite.
		Then, we have the following isomorphism
		\[
		\RHomr{A}(N, M \Ltenint{A} \D{\R Q_A(A)}) \cong \D{\RHomr{A}(M, \R Q_A(N))}.
		\]
		This isomorphism is functorial in $M,N$.
	\end{prop}

	To prove \cref{prop:Serre-functor-2}, we need some lemmas below (cf. \cite[Proposition 3.27, Lemma 3.29]{mori2025categorical}, \cite[Proposition 2.1]{jorgensen1998non-commutative}, \cite[Theorem 15.3.27]{yekutieli2019derived}).

	\begin{lem}
	\label{lem:Serre-functor-2-1}
	Let $A$ be a right Ext-finite connected $\Zbb$-algebra.
	Assume that $\cd(\torfun{A})$ is finite.
	Then, we have the following isomorphism
	\[
	\RHomintr{A}(N, \D{\R Q_A(A)}) \cong \D{\R Q_A(N)}
	\]
	for all $N \in D(\Grr{A})$.
	This isomorphism is functorial in $N$.
	\end{lem}

	\begin{proof}
	\(R^iQ_A\) commutes with direct limits.
	From \cite[Lemma 3.12]{mori2025categorical}, $\cd(Q_A)$ is finite.
	Then, we can apply the same argument as in \cref{subsec:local duality}. 
	We outline the details here.

	First, letting $B$ be a connected $\Zbb$-algebra, we have the bifunctorial isomorphism 
	\begin{align*}
		N \Ltenint{A} \R Q_{A}(M) \cong \R Q_{A}(N \Ltenint{A} M)
   \end{align*}
   for $M \in D(\Grlr{A}{A})$ and $N \in D(\Grlr{B}{A})$, following the argument in \cref{lem:local duality-1}.
   Then, by using the above isomorphism, we have the functorial isomorphism
   \begin{align*}
	\D{\R Q_{A}(N)} \cong \RHomint_A(N, \D{\R Q_{A}(A)})
   \end{align*}
   for $N \in D(\Grlr{B}{A})$, following the argument in \cref{thm:local duality}.
   Finally, by using the above isomorphism, we have the functorial isomorphism 
   \[
	\D{\R Q_{A}(N)} \cong \RHomint_A(N, \D{\R Q_{A}(A)})
	\]
	for $N \in D(\Grr{A})$, following the argument in \cref{cor:local-duality-2}.
	\end{proof}

	\begin{lem}
	\label{lem:Serre-functor-2-2}
	Let $A$ be a right noetherian $\Zbb$-algebra.
	Let $M$ be an object in $D^b(\Grr{A})$ with the property that $M$ has finite projective dimension.
	Let $N$ be an object in $D^{-}_f(\Grr{A})$ and $L$ be an object $\in D^{+}(\Grlr{A}{A})$.

	Then, we have the following isomorphism
	\[
	\RHomr{A}(N, M \Ltenint{A} L) \cong M \Lten{A} \RHomintr{A}(N, L).
	\]
	This isomorphism is functorial in $L,M,N$.
	\end{lem}

	\begin{proof}
	From the assumption, $M$ is quasi-isomorphic to a bounded complex $P$ of  projective graded right $A$-modules (cf. \cite[Lemma 15.69.2]{stacks-project}, projective version of \cite[Proposition 7.6]{hartshorne1966residues} or the proof of \cref{lem:finite injective dimension}).
	In addition, $N$ is quasi-isomorphic to a bounded above complex $F$ of finite free graded right $A$-modules (cf. \cite[Proposition 7.4.9]{yekutieli2019derived}).
	Thus, we obtain
	\begin{align*}
		\RHomr{A}(N, M \Ltenint{A} L) &\cong \Homr{A}(F, P \tenint{A} L), \\
		M \Lten{A} \RHomintr{A}(N, L) &\cong P \ten{A} \Homintr{A}(F, L).
	\end{align*}
	We assume that $F^{i}=0$ for all $i > i_0$, $P^{i} = 0$ for all $i < i_1, i_2 < i$ and $L^{i} = 0$ for all $i < i_3$.
	Then, for any $n \in \Zbb$, we have
	\begin{align*}
		\Hom_{A}^n(F, P \tenint{A} L) &= \prod_{p \in \Zbb} \Homr{A}(F^p, (P \tenint{A} L)^{n+p}) \\
		&\cong \prod_{p \leq i_0} \Homr{A}\left(F^p, \bigoplus_{i_1 \leq q \leq i_2} P^{q} \tenint{A} L^{p-q+n}\right) \quad (\text{boundedness of } F,P) \\
		&\cong \bigoplus_{p=i_1+i_3-n}^{i_0} \Homr{A}\left(F^p, \bigoplus_{i_1 \leq q \leq i_2} P^{q} \tenint{A} L^{p-q+n}\right) \quad (\text{boundedness of } L) \\
		&\cong \bigoplus_{p=i_1+i_3-n}^{i_0} \bigoplus_{i_1 \leq q \leq i_2} \Homr{A}\left(F^p,  P^{q} \tenint{A} L^{p-q+n}\right).
	\end{align*}
	This isomorphism is functorial in $F,L,P$.
	Moreover, we have 
	\begin{align*}
		(P \ten{A} \Homintr{A}(F, L))^n &= \bigoplus_{q+r=n} P^{q} \ten{A} {\Homintr{A}}^r(F, L) \\
		&\cong \bigoplus_{q+r=n} P^{q} \ten{A} \prod_{p \in \Zbb} \Homintr{A}(F^p, L^{p+r}) \\
		&\cong \bigoplus_{
				i_1 \leq q \leq i_2
			} 
			P^{q} \ten{A} \prod_{p \leq i_0} \Homintr{A}(F^p, L^{p-q+n}) \quad (\text{boundedness of } F,P)\\
		&\cong \bigoplus_{i_1 \leq q \leq i_2} P^{q} \ten{A} \bigoplus_{p=q+i_3-n}^{i_0} \Homintr{A}(F^p, L^{p-q+n}) \quad (\text{boundedness of } L) \\
		&\cong \bigoplus_{i_1 \leq q \leq i_2} \bigoplus_{p=q+i_3-n}^{i_0} P^{q} \ten{A} \Homintr{A}(F^p, L^{p-q+n}).
	\end{align*}
	This isomorphism is also functorial in $F,L,P$.
	We can easily check that 
	\[
	\Homr{A}\left(F^p,  P^{q} \tenint{A} L^{p-q+n}\right)  \cong P^{q} \ten{A} \Homintr{A}(F^p, L^{p-q+n})
	\]
	since $F^p$ is a finite free graded right $A$-module.
	Thus, we obtain $\Hom_{A}^n(F, P \tenint{A} L) \cong (P \ten{A} \Homintr{A}(F, L))^n$ for all $n \in \Zbb$, which is functorial in $F, L, P$.
	Therefore, following the argument in the proof of \cref{lem:local duality-1}, we obtain the desired functorial isomorphism in \(L,M,N\).
	\end{proof}

\begin{proof}[Proof of \cref{prop:Serre-functor-2}]
	As in the proof of \cref{lem:Serre-functor-2-1}, $\cd(Q_A)$ is finite.
	This means that $\D{\RQ_A(A)} \in D^{b}(\Grlr{A}{A})$.
	Thus, we have the following bifunctorial isomorphisms
	\begin{align*}
		\RHomr{A}(N, M \Ltenint{A} \D{\R Q_A(A)}) &\cong M \Lten{A} \RHomintr{A}(N, \D{\R Q_A(A)}) \quad (\text{\cref{lem:Serre-functor-2-2}}) \\
		&\cong M \Lten{A} \D{\R Q_A(N)} \quad (\text{\cref{lem:Serre-functor-2-1}}) \\
		&\cong e_1(Ke_1 \ten{k} M) \Ltenint{A} \D{\R Q_A(Ke_0 \ten{k} N)}e_0 \\
		&\cong e_1 \DD{[(Ke_1 \ten{k} M) \Ltenint{A} \D{\R Q_A(Ke_0 \ten{k} N)}]} e_0 \quad (\text{\cref{thm:matlis-duality}}) \\
		&\cong e_1\D{\RHomintr{A}(Ke_1 \ten{k} M, \R Q_A(Ke_0 \ten{k} N)'')}e_0 \quad (\text{\cref{lem:local duality-2}}) \\
		&\cong  e_1\D{\RHomintr{A}(Ke_1 \ten{k} M, \R Q_A(Ke_0 \ten{k} N))}e_0\quad (\text{\cref{thm:matlis-duality}}) \\
		&\cong \D{\RHomr{A}(M, \R Q_A(N))}
	\end{align*}
   Note that \(\RQ_A(N)\) is locally finite, hence \(M \otimes_A^{\mathbf L} \D{\RQ_A(N)}\) is also locally finite.
   Moreover, in the third isomorphism, we use the same technique as in the proof of \cref{cor:local-duality-2} so that we can apply \cref{lem:local duality-2} in the fifth isomorphism.
\end{proof}

The following theorem is the second main result in this paper (cf. \cite[Appendix A]{de2004ideal}, \cite[Theorem 6.8]{mori2025categorical}).
\begin{thm}
	\label{thm:Serre-functor}
	Let $A$ be a noetherian connected $\Zbb$-algebra.
	Assume that $A$ has a balanced dualizing complex $R_A$ and $\gldim(\qgr(A))$ is finite.
	Then, the category $D^b(\qgr(A))$ has a Serre functor $\Scal_{\qgr(A)}$ which is given by the following formula:
	\[
	\Scal_{\qgr(A)}(\pi_A(M)) = \pi_A(M \Ltenint{A} R_A)[-1].
	\]
\end{thm}

\begin{proof}
	We follow the strategy in the proof of \cite[Appendix A]{de2004ideal}.
	The proof is divided into three steps.

	\bigskip

	\noindent
	Step 0:
	Define the functor 
	\[
	F: D^b(\qgr(A)) \longrightarrow D^b(\qgr(A))
	\]
	by $F(\pi_A(M)) = \pi_A(M \Ltenint{A} R_A)$ for all $\pi_A(M) \in D^b(\qgr(A))$.

	We need to check that the functor $F$ is well-defined, i.e. if $\pi_A(M) \in  D^b(\qgr(A))$, then $F(\pi_A(M))$ is also in $D^b(\qgr(A))$.
	To see this, we take $P, N$ as in \cref{lem:Serre-functor-1}.
	Then, we have
	\[
	\pi_A(P \Ltenint{A} R_A) \cong F(\pi_A(M)) \oplus F(\pi_A(N)).
	\]
	Since $P^i$ is projective for all $i \in \Zbb$, we have $P \Ltenint{A} R_A \cong P \tenint{A} R_A \in D^b(\grr{A})$.
	Hence, $F(\pi_A(M))$ is also in $D^b(\qgr(A))$.

	\bigskip

	\noindent
	Step 1: 
	In this step, we show that the functor $F$ is an autoequivalence of $D^b(\qgr(A))$.

	Define the functor 
	\[
	G: D^b(\qgr(A)) \longrightarrow D^b(\qgr(A))
	\]
	by $G(\pi_A(M)) = \pi_A(\RHomint_{A}(R_A, M))$ for all $\pi_A(M) \in D^b(\qgr(A))$.
	We show that $G$ is a quasi-inverse of $F$.
	Before showing this, we need to check that the functor $G$ is well-defined.
	Let $\pi_A(M) \in D^b(\qgr(A))$.
	Then, 
	\begin{align*}
		G(\pi_A(M)) &\cong \pi_A(\RHomint_{A}(R_A, M)) \\
		&\cong \pi_A(\RHomintl{A}(\dfun{A}(M), \dfun{A}(R_A))) \quad (\text{\cref{prop:duality}}) \\
		&\cong \pi_A(\RHomintl{A}(\dfun{A}(M), A)) \quad (\text{\cref{dfn:dualizing complexes}}).
	\end{align*}
	By using $A^{\op}$-version of \cref{lem:Serre-functor-1}, we take $Q, L \in D^b(\grl{A})$ such that $Q^i$ is projective for all $i \in \Zbb$ and $\pi_A(\dfun{A}(M)) \oplus \pi_A(L) \cong \pi_A(Q)$.
	Then, we have
	\begin{align*}
	\pi_A(\RHomintl{A}(Q,A)) &\cong G(\pi_A(M)) \oplus \pi_A(\RHomintl{A}(L,A)).
	\end{align*}
	Since $Q^i$ is projective for all $i \in \Zbb$, we also have 
	\[
	\RHomintl{A}(Q,A) \cong \Homintr{A}(Q,A) \in D^b(\grr{A}).
	\]
	Hence, $G(\pi_A(M))$ is in $D^b(\qgr(A))$.

	Next, we show that $F \circ G \cong \Id_{D^b(\qgr(A))}$.
	Take any $\pi_A(M) \in D^b(\qgr(A))$.
	In fact, the claim follows from the following isomorphisms:
	\begin{align*}
		F(G(\pi_A(M))) &\cong \pi_A(\RHomint_{A}(R_A, M) \Ltenint{A} R_A) \\
		&\cong \pi_A(\RHomintl{A}(\dfun{A}(M), A) \Ltenint{A} R_A) \\
		&\cong \pi_A(\RHomintl{A}(\dfun{A}(M), R_A)) \\
		&\cong \pi_A((\dfunop{A}\circ \dfun{A})(M)) \\
		&\cong \pi_A(M) \quad (\text{\cref{prop:duality}}).
	\end{align*}

	In addition, we show that $G \circ F \cong \Id_{D^b(\qgr(A))}$.
	As for this, we have a natural morphism $\eta : \Id_{D^b(\qgr(A))} \rightarrow G \circ F$. 
	Both $\Id_{D^b(\qgr(A))}$ and $G \circ F$ are way-out on both directions.
	Thus, from \cref{lem:way-out functors}, it suffices to show that $\eta$ is an isomorphism for the set of objects 
	\begin{align*}
		\Qcal := \left\{
		\bigoplus_{i \in I} \pi_A(e_i A) \ \middle| \ I:
		\text{a finite set of integers that may contain duplicates}
		\right\}.
	\end{align*}
	Take any $\pi_A(P) = \bigoplus_{i \in I} \pi_A(e_i A) \in \Qcal$.
	Then, we have
	\begin{align*}
		(G \circ F)(\pi_A(P)) &\cong \pi_A(\RHomint_{A}(R_A, P \Ltenint{A} R_A)) \\
		&\cong \pi_A\left(\RHomint_{A}\left(R_A,  \bigoplus_{i \in I} e_i A \tenint{A} R_A\right)\right) \\
		&\cong  \bigoplus_{i \in I}\pi_A (\RHomint_{A}(R_A, e_i A \tenint{A} R_A)) \\
		&\cong  \bigoplus_{i \in I}\pi_A \RHomint_{A}(R_A, e_i R_A) \\
		&\cong  \bigoplus_{i \in I}\pi_A(e_i \RHomint_{A}(R_A, R_A)) \\
		&\cong  \bigoplus_{i \in I}\pi_A(e_i A) 
		\cong \pi_A(P).
	\end{align*}
	Therefore, $F$ is an autoequivalence of $D^b(\qgr(A))$.

	\bigskip 

	\noindent
	Step 2: 
	We show that there exists a natural isomorphism
	\[
	\Hom_{D^b(\qgr(A))}(X,Y) \cong \D{\Hom_{D^b(\qgr(A))}(Y,F(X)[-1])}
	\]
	for all $X,Y \in D^b(\qgr(A))$.

	Take any $\pi_A(M), \pi_A(N) \in D^b(\qgr(A))$.
	We assume that $M$ has finite projective dimension.
	Let $\omega_A$ be the right adjoint of the quotient functor $\pi_A$. 
	Since $\pi_A$ is exact, it is left adjoint to $\mathbf{R}\omega_A$ on the derived categories. 
	Recalling that $\mathbf{R}Q_A \cong \mathbf{R}\omega_A \circ \pi_A$, we obtain the following isomorphism by adjunction:
	\begin{equation}
	\Homr{A}(M, \R Q_A(N)) \cong \Hom_{D^b(\qgr(A))}(\pi_A(M), \pi_A(N)), \tag{A} \label[isomorphism]{eq:A}
	\end{equation}
	which is functorial in $M,N$.

	On the other hand, since the noetherian connected algebra $A$ is Ext-finite (\cref{rmk:relation among finite-locally finite-Ext-finite}), we have a triangle from \cite[Lemma 3.14]{mori2025categorical}
	\[
	\Rtorfun{A}(A) \rightarrow A \rightarrow \RQ_A(A),
	\]
	which induces a triangle
	\[
	M \Ltenint{A} R_A[-1] \cong M \Ltenint{A} \D{\Rtorfun{A}(A)}[-1] \rightarrow M \Ltenint{A} \D{\RQ_A(A)} \rightarrow M \Ltenint{A} \D{A}
	\]
	in $D(\Grlr{A}{A})$.
	Because $\pi_A(M \Ltenint{A} \D{A}) =0$ from the fact that $\D{A}$ is torsion and $M \in D^b(\grr{A})$ has finite projective dimension, we have an isomorphism
	\[
	\pi_A(M \Ltenint{A} R_A)[-1] \cong \pi_A(M \Ltenint{A} \D{\RQ_A(A)}).
	\]
	Moreover, we can obtain the isomorphism
	\[
	\Homr{A}(N, M \Ltenint{A} \D{\R Q_A(A)}) \cong \Hom_{D^b(\qgr(A))}(\pi_A(N), \pi_A(M \Ltenint{A} \D{\R Q_A(A)})).
	\]
	Actually, by \cref{prop:torsion-functor} and \cite[Lemma 3.12]{mori2025categorical}, 
	\[
	\RQ_A(M \Ltenint{A} \D{\RQ_A(A)}) = \varinjlim_{n \to \infty} \RHomintr{A}(A_{\geq n}, M \Ltenint{A} \D{\RQ_A(A)}). 
	\]
	Then, we have
	\begin{align*}
		\RQ_A(M \Ltenint{A} \D{\RQ_A(A)}) &\cong \varinjlim_{n \to \infty} \RHomintr{A}(A_{\geq n}, M \Ltenint{A} \D{\RQ_A(A)}) \\
		&\cong \varinjlim_{n \to \infty} \bigoplus_{i \in \Zbb} \RHomr{A}(e_iA_{\geq n}, M \Ltenint{A} \D{\RQ_A(A)}) \\
		&\cong \varinjlim_{n \to \infty} \bigoplus_{i \in \Zbb} \D{\RHomr{A}(M,  \RQ_A(e_i A_{\geq n}))} \quad (\text{\cref{prop:Serre-functor-2}}) \\
		&\cong \varinjlim_{n \to \infty} \bigoplus_{i \in \Zbb} \D{\RHomr{A}(M, \RQ_A(e_iA))} \quad (\RQ_A(e_i A_{\geq n}) \cong \RQ_A(e_i A)) \\
		&\cong \varinjlim_{n \to \infty} \bigoplus_{i \in \Zbb} \RHomr{A}(e_iA, M \Ltenint{A} \D{\RQ_A(A)})  \quad (\text{\cref{prop:Serre-functor-2}})\\
		&\cong \varinjlim_{n \to \infty} M \Ltenint{A} \D{\RQ_A(A)} \\
		&\cong M \Ltenint{A} \D{\RQ_A(A)}.
	\end{align*}
	From adjunction,
	\begin{align*}
		\Homr{A}(N, M \Ltenint{A} \D{\RQ_A(A)}) &\cong \Homr{A}(N, \RQ_A(M \Ltenint{A} \D{\RQ_A(A)})) \\
		&\cong \Hom_{D^b(\qgr(A))}(\pi_A(N), \pi_A(M \Ltenint{A} \D{\RQ_A(A)})).
	\end{align*}
	Hence, we obtain
	\[
	\Homr{A}(N, M \Ltenint{A} \D{\RQ_A(A)}) \cong \Hom_{D^b(\qgr(A))}(\pi_A(N), F(\pi_A(M))[-1]),\tag{B} \label[isomorphism]{eq:B}
	\]
	which is functorial in $M,N$, since the above isomorphisms are functorial in $M,N$.
	Finally, from \cref{eq:A} and \cref{eq:B}, together with \cref{prop:Serre-functor-2}, we have a natural isomorphism
	\[
	\Hom_{D^b(\qgr(A))}(\pi_A(M), \pi_A(N)) \cong \D{\Hom_{D^b(\qgr(A))}(\pi_A(N), F(\pi_A(M))[-1])}.
	\]

	For a general $\pi_A(M) \in D^b(\qgr(A))$, we can take $P,L$ and a decomposition $\pi_A(P) \cong \pi_A(M) \oplus \pi_A(L)$ as in \cref{lem:Serre-functor-1}.
	We write the natural isomorphism as 
	\[
	\eta : \Hom_{D^b(\qgr(A))}(\pi_A(P), \pi_A(N)) \xrightarrow{\cong} \D{\Hom_{D^b(\qgr(A))}(\pi_A(N), F(\pi_A(P))[-1])}.
	\]
	We denote by $i: \pi_A(M) \rightarrow \pi_A(P)$ and $p: \pi_A(P) \rightarrow \pi_A(M)$ the inclusion and projection morphisms, respectively.
	We also put $e = i \circ p: \pi_A(P) \rightarrow \pi_A(P)$.
	Note that $e$ is an idempotent morphism and $p \circ i = \Id_{\pi_A(M)}$.
	Then, from the naturality of $\eta_{}$, we have the following commutative diagram:
	\[
	\small
	\begin{tikzcd}
		\Hom_{D^b(\qgr(A))}(\pi_A(P), \pi_A(N)) \ar[r, "\eta_{}"] \ar[d, "- \circ e"'] \ar[dd, bend right=90, "- \circ i"']& \D{\Hom_{D^b(\qgr(A))}(\pi_A(N), F(\pi_A(P))[-1])} \ar[d, "({F(e)[-1]} \circ -)' "] \ar[dd, bend left=90, "({F(i)[-1]} \circ -)' "] \\ 
		\Hom_{D^b(\qgr(A))}(\pi_A(P), \pi_A(N)) \ar[r, "\eta_{}"] & \D{\Hom_{D^b(\qgr(A))}(\pi_A(N), F(\pi_A(P))[-1])} \\
		\Hom_{D^b(\qgr(A))}(\pi_A(M), \pi_A(N)) \ar[u, "- \circ p"] & \D{\Hom_{D^b(\qgr(A))}(\pi_A(N), F(\pi_A(M))[-1])} \ar[u, "({F(p)[-1]} \circ -)' "'].
	\end{tikzcd}
	\]
	From the equation $p \circ i = \Id_{\pi_A(M)}$, 
	$- \circ p$ and $(F(p)[-1] \circ -)'$ are injective, and $- \circ i$ and $(F(i)[-1] \circ -)'$ are surjective.
	Thus, $\eta$ restricts to the direct summands via the morphisms $- \circ e$ and $(F(e)[-1] \circ -)'$, and hence induces an isomorphism
	\[
	\tilde{\eta} : \Hom_{D^b(\qgr(A))}(\pi_A(M), \pi_A(N)) \xrightarrow{\cong} \D{\Hom_{D^b(\qgr(A))}(\pi_A(N), F(\pi_A(M))[-1])},
	\]
	which is given by the formula 
	\[
	\tilde{\eta}(f)(h) = ((F(i)[-1] \circ -)' (\eta(f \circ p)))(h) = ( \eta(f \circ p))(F(i)[-1]\circ h)
	\]
	for every 
	$f: \pi_A(M) \rightarrow \pi_A(N)$ and 
	$h \in \Hom_{D^b(\qgr(A))}(\pi_A(N), F(\pi_A(M))[-1])$.

	$\tilde{\eta}$ does not depend on the choice of $P$ because $\eta$ is functorial in $P$.
	Indeed, if $P'$ is another choice of $P$, $i', p'$ and $e'$ are the inclusion, projection and idempotent morphism, respectively.
	We also write $\eta'$ and $\tilde{\eta}'$ for the isomorphisms corresponding to the pairs $(\pi_A(P'), \pi_A(N))$ and $(\pi_A(M), \pi_A(N))$, respectively.
	Then, we have
	\begin{align*}
	\tilde{\eta}(f)(h) 
	&= (F(i)[-1] \circ -)' (\eta(f \circ p))(h) \qquad (\text{definition of $\tilde{\eta}$})\\
	&= (F(i)[-1] \circ -)' (\eta(f \circ p' \circ i' \circ p))(h) \qquad (p' \circ i' = \Id_{\pi_A(M)}) \\
	&=(F(i)[-1] \circ -)' ((F(i' \circ p)[-1] \circ -)' (\eta(f \circ p')))(h) \qquad (\text{naturality of $\eta$}) \\
	&= \eta'(f \circ p')(F(i' \circ p)[-1] \circ F(i)[-1] \circ h) \\
	&= \eta'(f \circ p')(F(i')[-1] \circ h) \qquad (p \circ i = \Id_{\pi_A(M)}) \\
	&= \tilde{\eta}'(f)(h) \qquad (\text{definition of $\tilde{\eta}'$}).
	\end{align*}
	Thus, $\tilde{\eta}$ does not depend on the choice of $P$.

	Moreover, it is functorial in $N$ because $\eta$ is functorial in $N$.
	If $N'$ is another object in $D^b(\qgr(A))$, $\psi: \pi_A(N) \rightarrow \pi_A(N')$ is a morphism in $D^b(\qgr(A))$ and $\eta', \tilde{\eta}'$ are the natural isomorphisms corresponding to the pairs $(\pi_A(P), \pi_A(N'))$ and $(\pi_A(M), \pi_A(N'))$, respectively, then we can prove that the following diagram commutes:
	\[
	\begin{tikzcd}
		\Hom_{D^b(\qgr(A))}(\pi_A(M), \pi_A(N)) \ar[r, "\tilde{\eta}"] \ar[d, "\psi \circ - "'] & \D{\Hom_{D^b(\qgr(A))}(\pi_A(N), F(\pi_A(M))[-1])} \ar[d, "({- \circ \psi})' "] \\
		\Hom_{D^b(\qgr(A))}(\pi_A(M), \pi_A(N')) \ar[r, "\tilde{\eta}' "] & \D{\Hom_{D^b(\qgr(A))}(\pi_A(N'), F(\pi_A(M))[-1])}.
	\end{tikzcd}
	\]
	Indeed, for any  $f : \pi_A(M) \rightarrow \pi_A(N)$ and $h : \pi_A(N') \rightarrow F(\pi_A(M))[-1]$, we have
	\begin{align*}
		\tilde{\eta}'(\psi \circ f)(h) 
		&= (F(i)[-1] \circ -)' (\eta'(\psi \circ f \circ p))(h) \qquad (\text{definition of $\tilde{\eta}'$}) \\
		&= (F(i)[-1] \circ -)'((- \circ \psi)' (\eta(f \circ p)))(h) \qquad (\text{naturality of $\eta$})\\
		&= \eta(f \circ p)(F(i)[-1] \circ h \circ \psi) \\
		&= (F(i)[-1] \circ -)' (\eta(f \circ p))(h \circ \psi)\\
		&= \tilde{\eta}(f)(h \circ \psi) \\
		&= (- \circ \psi)' (\tilde{\eta}(f))(h).
	\end{align*}
	Thus, $\tilde{\eta}$ is natural in $N$.

	Finally, we show that $\tilde{\eta}$ is natural in $M$.
	Let $\varphi : \pi_A(M) \rightarrow \pi_A(M')$ be a morphism in $D^b(\qgr(A))$.
	For $M'$, we take $P'$ as in \cref{lem:Serre-functor-1}.
	We denote by $i': \pi_A(M') \rightarrow \pi_A(P')$, $p': \pi_A(P') \rightarrow \pi_A(M')$ and $e': \pi_A(P') \rightarrow \pi_A(P')$ the inclusion, projection and idempotent morphism, respectively.
	We also write $\eta'$ and $\tilde{\eta}'$ for the isomorphisms corresponding to the pairs $(\pi_A(P'), \pi_A(N))$ and $(\pi_A(M'), \pi_A(N))$, respectively.
	Then, we prove that the following diagram commutes:
	\[
	\begin{tikzcd}
		\Hom_{D^b(\qgr(A))}(\pi_A(M'), \pi_A(N)) \ar[r, "\tilde{\eta}'"] \ar[d, "- \circ \varphi"'] & \D{\Hom_{D^b(\qgr(A))}(\pi_A(N), F(\pi_A(M'))[-1])} \ar[d, "({F(\varphi)[-1] \circ -})' "] \\
		\Hom_{D^b(\qgr(A))}(\pi_A(M), \pi_A(N)) \ar[r, "\tilde{\eta} "] & \D{\Hom_{D^b(\qgr(A))}(\pi_A(N), F(\pi_A(M))[-1])}.
	\end{tikzcd}
	\]
	Let $f : \pi_A(M') \rightarrow \pi_A(N)$ be a morphism in $D^b(\qgr(A))$.
	Then, for every $h: \pi_A(N) \rightarrow F(\pi_A(M))[-1]$, we have 
	\begin{align*}
		\tilde{\eta}(f \circ \varphi)(h) 
		&= (F(i)[-1] \circ -)' (\eta(f \circ \varphi \circ p))(h) \qquad (\text{definition of $\tilde{\eta}$}) \\
		&= (F(i)[-1] \circ -)' (\eta((f \circ p') \circ (i' \circ \varphi \circ p)))(h) \qquad (p' \circ i' = \Id_{\pi_A(M')}) \\
		&= (F(i)[-1] \circ -)' ((F(i' \circ \varphi \circ p)[-1] \circ -)' (\eta'(f \circ p')))(h) \qquad (\text{naturality of $\eta$})\\
		&= \eta'(f \circ p')(F(i' \circ \varphi \circ p)[-1] \circ F(i)[-1] \circ h) \\
		&= \eta'(f \circ p')(F(i')[-1] \circ F(\varphi)[-1] \circ h) \qquad (p \circ i = \Id_{\pi_A(M)})\\
		&= (F(\varphi)[-1] \circ - )'((F(i')[-1] \circ -)' (\eta'(f \circ p')))(h) \\
		&= (F(\varphi)[-1] \circ - )' (\tilde{\eta}'(f))(h) \qquad (\text{definition of $\tilde{\eta}'$}).
	\end{align*}
	Thus, $\tilde{\eta}$ is natural in $M$.


	\bigskip

	Hence, $F[-1]$ is a Serre functor of $D^b(\qgr(A))$ from Step 0, Step 1 and Step 2.
\end{proof}

\begin{rmk}
    \label{rmk:Serre-functor}
    As a final note, we explain why $\pi_A(M \Ltenint{A} R_A)[-1]$ 
    and $\pi_A(\RQ_A(M) \Ltenint{A} R_A)[-1]$ are isomorphic in $D(\qgr(A))$.
    The former appears in \cref{thm:Serre-functor}, while the latter appears in \cref{thm:intro-Serre-functor}.

    Indeed, there is a canonical triangle in $D(\Grr{A})$ 
    (\cite[Lemma 3.12]{mori2025categorical})
    \[
    \Rtorfun{A}(M) \to M \to \RQ_A(M).
    \]
    Tensoring with $R_A$ and applying $\pi_A$, we obtain
    \[
    \pi_A(\Rtorfun{A}(M) \Ltenint{A} R_A) \to 
    \pi_A(M \Ltenint{A} R_A) \to 
    \pi_A(\RQ_A(M) \Ltenint{A} R_A).
    \]
    Since $\pi_A(\Rtorfun{A}(M) \Ltenint{A} R_A) = 0$, the desired isomorphism follows.
\end{rmk}

\printbibliography 


\end{document}

%% file: mizuno_preamble2.tex
\usepackage{amscd}
\usepackage{amsmath, amssymb}
\usepackage{mathrsfs}
\usepackage{mathtools}
\usepackage{mathabx}
\usepackage{ifthen}
\usepackage{bm}
\usepackage{amsthm}
\usepackage[dvipsnames]{xcolor}
\usepackage{url}
\usepackage{paralist} 
\usepackage{enumerate}
\usepackage[inline]{enumitem}
\usepackage{latexsym}
\usepackage{arydshln}

\usepackage{listings}
\usepackage{lineno}
\usepackage{tikz}
\usepackage{comment}
\usetikzlibrary{cd}
\usepackage{relsize}
\usepackage{stmaryrd}

\newtheorem{thm}{Theorem}[section]

\newaliascnt{lem}{thm}            
\newtheorem{lem}[lem]{Lemma}      
\aliascntresetthe{lem}            
\crefname{lem}{lemma}{lemmas}
\Crefname{lem}{Lemma}{Lemmas}

\newaliascnt{prop}{thm}
\newtheorem{prop}[prop]{Proposition}
\aliascntresetthe{prop}
\crefname{prop}{proposition}{propositions}
\Crefname{prop}{Proposition}{Propositions}

\newaliascnt{cor}{thm}
\newtheorem{cor}[cor]{Corollary}
\aliascntresetthe{cor}
\crefname{cor}{corollary}{corollaries}
\Crefname{cor}{Corollary}{Corollaries}

 \theoremstyle{definition}
\newaliascnt{dfn}{thm}
\newtheorem{dfn}[dfn]{Definition}
\aliascntresetthe{dfn}
\crefname{dfn}{definition}{definitions}
\Crefname{dfn}{Definition}{Definitions}

\newaliascnt{con}{thm}

\aliascntresetthe{con}
\crefname{con}{condition}{conditions}
\Crefname{con}{Condition}{Conditions}

\theoremstyle{remark}
\newaliascnt{rmk}{thm}
\newtheorem{rmk}[rmk]{Remark}
\aliascntresetthe{rmk}
\crefname{rmk}{remark}{remarks}
\Crefname{rmk}{Remark}{Remarks}

\newaliascnt{eg}{thm}
\newtheorem{eg}[eg]{Example}
\aliascntresetthe{eg}
\crefname{eg}{example}{examples}
\Crefname{eg}{Example}{Examples}

\numberwithin{equation}{subsection}

\newcommand{\Nbb}{\mathbb{N}} 
\newcommand{\Zbb}{\mathbb{Z}}

\newcommand{\Lbb}{\mathbb{L}} 

\newcommand{\Acal}{\mathcal{A}}
\newcommand{\Bcal}{\mathcal{B}}
\newcommand{\Ccal}{\mathcal{C}}

\newcommand{\Pcal}{\mathcal{P}}
\newcommand{\Qcal}{\mathcal{Q}}

\newcommand{\Scal}{\mathcal{S}}
\newcommand{\Tcal}{\mathcal{T}}

\newcommand{\Rbf}{\mathbf{R}}

\newcommand{\R}{\Rbf}

\newcommand{\Q}{Q}

\newcommand{\RQ}{\R Q}

\newcommand{\Romega}{\R\omega}

\newcommand{\D}[1]{#1^{\prime}}
\newcommand{\DD}[1]{#1^{\prime\prime}}

\newcommand{\mydot}{{\scriptscriptstyle \bullet}}

\DeclareMathOperator{\supr}{sup}

\DeclareMathOperator{\Hom}{Hom}
\DeclareMathOperator{\Id}{Id} 
\DeclareMathOperator{\Ker}{Ker}
\DeclareMathOperator{\Cok}{Cok}
\DeclareMathOperator{\Img}{Im} 
\DeclareMathOperator{\Ext}{Ext}

\DeclareMathOperator{\Res}{Res}
\DeclareMathOperator{\cd}{cd}

\newcommand{\dfun}[1]{D_{A}}
\newcommand{\dfunop}[1]{D_{A^{\op}}}

\newcommand{\Resl}[1]{{}_{#1}\Res}
\newcommand{\Resr}[1]{\Res_{#1}}

\newcommand{\Resll}[2]{{}_{#1}\Res^{#2}}
\newcommand{\Resrr}[2]{\Res_{#1}^{#2}}

\newcommand{\Irr}[3]{L_{#1-#2}^{r, #3}}
\newcommand{\Ill}[3]{L_{#1-#2}^{l, #3}}
\newcommand{\Jrr}[3]{R_{#1-#2}^{r, #3}}
\newcommand{\Jll}[3]{R_{#1-#2}^{l, #3}}

\newcommand{\Homr}[1]{\Hom^{}_{#1}}
\newcommand{\Homl}[1]{\Hom^{}_{#1^{\op}}}
\newcommand{\Homlr}[2]{\Hom^{}_{#1-#2}}
\newcommand{\Homrcpx}[1]{{\Hom^{\mydot}_{#1}}}
\newcommand{\Homlcpx}[1]{{\Hom^{\mydot}_{#1^{\op}}}}
\newcommand{\Homlrcpx}[2]{{\Hom^{\mydot}_{#1-#2}}}

\newcommand{\RHomr}[1]{\RHom^{}_{#1}}

\newcommand{\Homintr}[1]{\Homint^{}_{#1}}

\newcommand{\RHomintr}[1]{\RHomint^{}_{#1}}
\newcommand{\RHomintl}[1]{\RHomint^{}_{#1^{\op}}}

\newcommand{\Extr}[1]{\Ext_{#1}}
\newcommand{\Extl}[1]{\Ext_{#1^{\op}}}

\newcommand{\Extintr}[1]{\Extint_{#1}}
\newcommand{\Extintl}[1]{\Extint_{#1^{\op}}}

\newcommand{\RHom}{\R\Hom}
\DeclareMathOperator{\Homint}{\underline{Hom}}
\DeclareMathOperator{\RHomint}{\R\underline{Hom}}
\DeclareMathOperator{\Extint}{\underline{Ext}}

\newcommand{\ten}[1]{ \otimes_{#1}}
\newcommand{\tenint}[1]{\, \underline{\otimes}_{#1}}
\newcommand{\Lten}[1]{ \otimes^{\Lbb}_{#1}}
\newcommand{\Ltenint}[1]{\, \underline{\otimes}^{\Lbb}_{#1}}

\newcommand{\torfun}[1]{\Gamma_{\mathfrak{m}_{#1}}}
\newcommand{\Rtorfun}[1]{\R\Gamma_{\mathfrak{m}_{#1}}}

\DeclareMathOperator{\Gr}{Gr}
\DeclareMathOperator{\gr}{gr}

\newcommand{\Grlr}[2]{\Gr^{}(#1-#2)}

\newcommand{\Grl}[1]{\Gr^{}(#1^{\op})}
\newcommand{\Grr}[1]{\Gr^{}(#1)}

\newcommand{\grl}[1]{\gr^{}(#1^{\op})}
\newcommand{\grr}[1]{\gr^{}(#1)}

\newcommand{\grcofr}[1]{\gr^{\prime}_{}(#1)}
\newcommand{\grcofl}[1]{\gr^{\prime}_{}(#1^{\op})}

\newcommand{\Grlflr}[2]{\Gr_{lf}(#1-#2)}
\newcommand{\Grlfr}[1]{\Gr_{{lf}}(#1)}
\newcommand{\Grlfl}[1]{\Gr_{lf}(#1^{\op})}

\DeclareMathOperator{\Tor}{Tor}
\DeclareMathOperator{\tor}{tor}
\DeclareMathOperator{\QGr}{QGr}
\DeclareMathOperator{\qgr}{qgr}
\DeclareMathOperator{\Soc}{Soc}

\DeclareMathOperator{\gldim}{gl_\cdot dim}
\DeclareMathOperator{\injdim}{inj_\cdot dim}
\DeclareMathOperator{\projdim}{pd_\cdot dim}

\DeclareMathOperator{\op}{op}

\newcommand{\zgrz}[1]{\bar{#1}}
\newcommand{\widezgrz}[1]{\widebar{#1}}

